\documentclass[review,hidelinks,onefignum,onetabnum]{siamart220329}

\usepackage{lipsum}
\usepackage{amsfonts}
\usepackage{graphicx}
\usepackage{epstopdf}
\usepackage{algorithmic}
\ifpdf
  \DeclareGraphicsExtensions{.eps,.pdf,.png,.jpg}
\else
  \DeclareGraphicsExtensions{.eps}
\fi

\newcommand{\TheTitle}{Energy regularized models for logarithmic SPDEs and their numerical approximations}
\newcommand{\TheAuthors}{J. Cui, D. Hou and Z. Qiao}

\headers{Energy regularized models for logarithmic SPDEs}{\TheAuthors}

\title{{\TheTitle}\footnotemark[1]\thanks{The research is partially supported by the Hong Kong Research Grant Council ECS grant 25302822, start-up funds (P0039016, P0041274) from Hong Kong Polytechnic University and the CAS AMSS-PolyU Joint Laboratory of Applied Mathematics.
		D. Hou's work is partially supported by NSFC grant 12001248, NSF of Jiangsu Province grant
		BK20201020, and Hong Kong Polytechnic University grant 1-W00D.
		Z. Qiao's work is partially supported by Hong Kong Research Council RFS grant RFS2021-5S03 and GRF grant 15302919, Hong Kong Polytechnic University grant 4-ZZLS, and CAS AMSS-PolyU Joint Laboratory of Applied Mathematics.
}}

\author{
  Jianbo Cui\footnotemark[2]\thanks{Department of Applied Mathematics, The Hong Kong Polytechnic University, Hung Hom, Hong Kong. Corresponding author: \email{jianbo.cui@polyu.edu.hk}}
  \and
  Dianming Hou\footnotemark[3]\thanks{School of Mathematics and Statistics, Jiangsu Normal University, Xuzhou, Jiangsu 221116, China. Current address:Department of Applied Mathematics, The Hong Kong Polytechnic University, Hung Hom, Hong Kong.  \email{dmhou@stu.xmu.edu.cn}}
  \and
  Zhonghua
  Qiao\footnotemark[4]\thanks{{Department of Applied Mathematics, The Hong Kong Polytechnic University, Hung Hom, Hong Kong. \email{zhonghua.qiao@polyu.edu.hk} }
}
}

\usepackage{amsopn}

\ifpdf
\hypersetup{
  pdftitle={\TheTitle},
  pdfauthor={\TheAuthors}
}
\fi

\begin{document}
\graphicspath{{figures/},}
\maketitle

\begin{abstract}
Understanding the properties of the stochastic phase field models is crucial to model processes in several practical applications, such as soft matters and phase separation in random environments. To describe such random evolution, this work proposes and studies two mathematical models and their numerical approximations for parabolic stochastic partial differential equation (SPDE) with a logarithmic Flory--Huggins energy potential.
These multiscale models are built based on a regularized energy technique and thus avoid possible singularities of coefficients. According to the large deviation principle, we show that the limit of the proposed models with small noise naturally recovers the classical dynamics in deterministic case. Moreover, when the driving noise is multiplicative, the Stampacchia maximum principle holds which indicates the robustness of the proposed model.
One of the main advantages of the proposed models is that they can admit
the energy evolution law and asymptotically preserve the Stampacchia maximum bound of the original problem.
To numerically capture these asymptotic behaviors,  we investigate the  semi-implicit discretizations for regularized logrithmic SPDEs.
Several numerical results  are presented to verify our theoretical findings.
\end{abstract}



\section{Introduction}
The phase field model, such as  the Allen-Cahn equation \cite{ALLEN19791085}, has become an popular tool to model processes involving thin interface layers between almost homogeneous
regions emerged from many
scientific, engineering, and industrial applications. The unknown solution of the
model, also called the order parameter, could represent the normalized density of the involved phase for phase separation process \cite{brokate2012hysteresis}. For example, the sets of $\{u=1\}$ and $\{u=-1\}$ often denotes the pure regions. Nowadays, more attention has been paid on stochastic
phase field models (see, e.g., \cite{MR3765890,MR1359472,Cer03,KORV07,Web10,Yip98}) since uncertainty and randomness may arise from thermal effects,
impurities of materials, intrinsic instabilities of dynamics, etc.

In this paper, we are mainly interested in the following parabolic SPDE with a Flory--Huggins logarithmic potential,
\begin{align}\label{log-sac}
	du(t)-\sigma A u(t)dt+F_{\log}'(u(t))dt=\epsilon B(u(t))dW(t).
\end{align}
Here $A$ is the Laplacian operator defined on a Lipschitz bounded domain $\mathcal D\subset \mathbb R^d, d\le 3,$ equipped with the homogeneous Neumann boundary condition or periodic boundary condition, and $\sigma>0$ is associated with the diffuse interface width.
The logarithmic Flory--Huggins  potential  $\int_{\mathcal D}F_{\log}(u) dx$ is determined by the following function
$$
F_{log}(\xi):=(1+\xi)\ln(1+\xi)+(1-\xi)\ln(1-\xi)-c\xi^2, \; \xi\in (-1,1),\; c\ge 0,
$$
which has two global minima points in the interior of the physically relevant  domain. The driving noise $W$ is a cylindrical Wiener process, $\epsilon$ is the noise intensity and $B$ is a suitable Hilbert--Schmidt operator.
This kind of noise is important in soft matters since the elements constituting soft matter (colloidal particles, polymer molecules, etc.) are doing Brownian motion, and their behaviour determines the dynamical response of soft matter (also known as the fluctuation-dissipation theorem in \cite{Doi13}).  It is known that the free energy with the logarithmic potential is often considered to be more physically realistic than that with a smooth polynomial potential since the free energy with the logarithmic potential can be derived from the regular or ideal solution theories \cite{Doi13}. The main difficulties of building mathematical models for logarithmic SPDEs lie on the singularity of the logarithmic Flory--Huggins potential and random effect which emerge from the polymer science community.

In deterministic case ($\epsilon=0$), some strategies have been successfully used to deal with the logarithmic Flory--Huggins potential, for example, replacing it by a polynomial of even
degree with a strictly positive dominant coefficient or using the penalization
method (see, e.g., \cite{MR1327930}).
For parabolic SPDEs of phase field type in the smooth nonlinearity case,
there have also been fruitful theoretical and numerical results  \cite{MR4224928,AKM16,BCH18,MR3986273,MR3906990,CCZZ18,CH18,CH19,CH20,CHS21S,FKLL18,FLZ17,KLL18,LQ18,MR4092279,MP17,MR4102717,MR4140034}, just to name a few.
 However, much less is known in the stochastic case, especially on the modeling and simulation of \eqref{log-sac}. In \cite{MR2349572,MR3373303}, the authors obtained the well-posedness of a stochastic Cahn--Hilliard equation with logarithmic potential by imposing additional reflection measures.  A variational approach based on suitable monotonicity and compactness arguments to study parabolic SPDEs with logarithmic potential driving by multiplicative noise was proposed (see \cite{MR3778274,MR4281433} for the stochastic Cahn--Hilliard equation and \cite{MR4151195} for the stochastic Allen--Cahn equation). To better understand the properties of SPDE with the Flory--Huggins potential, one useful and important tool is through the numerical approximation. However, this approach requires a well-defined mathematical model which is also easy to handle and maintains as much of the nature and properties of the original system as possible. This is very challenging in the stochastic case for the soft matter dynamics with the logarithmic Flory--Huggins potential.
 First, different from the deterministic case, one can not expect that the classical maximum principle of parabolic PDE or the convexity of the energy functional could be used to obtain the point-wise bound (see, e.g., \cite{MR4116082,MR4443516,MR4417002,DJLQ19,DJLQ21}) due to the random effect.
 Second, the existing results either use the polynomial approximation or add the reflection measure for the drift coefficient of the stochastic dynamics, which may changes the properties of the original problem, especially on the ideal solution theory \cite{Doi13} and energy evolution. Last but not least, it is still unknown how to design a suitable numerical approximation and analyze its convergence properties  for parabolic SPDE with the  Flory--Huggins potential.

To overcome the above issues, we will first adopt the energy regularization approach in \cite{MR3813596,C2020} to propose a convergent and well-posed multiscale problem with the regularization parameter $\delta$. In particular, we will prove the convergence of the regularized problems in two cases, i.e., the small noise ($\epsilon\sim o(1)$) in Case 1 and a special multiplicative noise in Case 2.
We would like mention that Case 1 is inspired by the large deviation principle of random perturbations of classical dynamical systems \cite{MR1652127} and the observation
that the simulation of \eqref{log-sac} may escape the maximum bound in applications.
The assumptions in Case 2 is similar to that of \cite{MR3373303} where the  Stampacchia maximum principle is expected to hold (see \cite{MR3427686} for the Stampacchia maximum principle).
Then applying numerical discretization to the regularized problem will not suffer from the numerical vacuum issue.
Specifically, we study the stabilized semi-implicit numerical method, where the spectral Galerkin method  is chosen as the space discretization.
To analyze the convergence of the energy regularized numerical method, the key ingredient of our approach lies on the estimate of the probability that the solution escapes
Stampacchia maximum bound. As a consequence, the error analysis of the proposed methods is presented.

 We would like to remark that this energy regularized technique could directly provide a
well-posed and natural regularized  model for \eqref{log-sac} which may also works for fourth order logarithmic SPDE, like the stochastic Cahn--Hilliard equation with logarithmic potential. Meanwhile, the Stampacchia maximum principle compensates for the deficiencies of regularization technique that it may not capture the singular behavior of the solution. The numerical simulation also indicates the advantages of the proposed energy regularized models and numerical approximations. For example, it is observed that the  Stampacchia maximum principle is asymptotically preserved by the energy regularized scheme for the first time.
The limit of \eqref{log-sac} with small noise could recover the Allen--Cahn equation with Flory--Huggins potential (see, e.g., \cite{MR4443516}).

The rest of this paper is organized as follows. In section \ref{sec-2}, we present the main setting and the useful properties of regularized parabolic SPDEs. In section \ref{sec-3}, we propose  semi-implicit discretizations for the regularized models, and study their convergence. Several numerical examples are shown in section \ref{sec-4} to verify the theoretical findings.

\section{Regularized parabolic SPDEs with logarithmic potential}
\label{sec-2}
This section is devoted to presenting the useful properties of the regularized SPDEs with logarithmic potential via the energy regularization technique. Before that, let us first introduce the main setting of this paper.

\subsection{Setting}
\label{set-1}
Let $\bs\ge 0, T>0$.
Assume that there exists an orthonormal basis of $e_i$ of $H=L^2(\mathcal D)$ such that $-Ae_i=\lambda_i e_i$ with $\lambda_0=0< \lambda_1<\cdots<\lambda_j<\lambda_{j+1}\to\infty.$
For example,
$\mathcal D$ can be chosen as the cuboid in $\mathbb R^d, d\in \mathbb N^+.$
Denote $H^{\bs}:=W^{\bs,2}(\mathcal D)$ which
is equivalent to the fractional Sobolev space $\mathbb L H \oplus (-A)^{\frac \bs 2}$ (see, e.g., \cite{CHS21S}). Here $\mathbb L H$ is the projection from $H$ to Span$\{e_0\}$.
Let $\mathbb H$ and $\widetilde H$ be two Hilbert space.
The set of bounded linear operators from $\HH$ to $\widetilde H$ is denoted by $\mathcal{L}(\HH,\widetilde H)$, and the set of Hilbert--Schmidt linear operators from $\HH$ to $\widetilde H$ is denoted by  $\mathcal{L}_2(\HH,\widetilde H)$, with associated norm denoted by $\|\cdot\|_{\mathcal{L}_2(\HH,\widetilde H)}$. We also denote by $\mathcal{L}_2(\mathbb H):=\mathcal{L}_2(\HH,\HH).$
For convenience, we assume that the initial value $u(0)\in H^{\bs}$ for some $\bs\ge 0$ satisfying $\|u(0)\|_{E}<1$, where $E=\mathcal C(\mathcal D)$. In this paper, the following two type of noises will be considered.
In Case 1, we assume that $\epsilon\sim o(1)$ and the operator $B$ satisfies that
\begin{align}\label{lip-hs}
	\|B(u)-B(v)\|_{\mathcal L_2(H,H^{\gamma})}&\le C_{\bs}\|u-v\|_{H^{\gamma}}, \\\nonumber
	\|B(u)\|_{\mathcal L_2(H,H^{\gamma})}&\le C_{\bs}(1+\|u\|_{H^{\gamma}})
\end{align}
for any $0\le \gamma \le \bs $ and some $C_{\bs}\ge 0$.
In Case 2, we assume that $\epsilon \sim O(1)$, and $B$ is defined by
\begin{align*}
	B(v)e_i:=h_i(v), \; h_i\in W_0^{\bs+1,\infty}[-1,1],\; \sum_{i\in \mathbb N}\|h_i\|_{W^{\bs+1,\infty}_0[-1,1]}^2< \infty.
\end{align*}
As a consequence, Case 2 also implies \eqref{lip-hs} with some $\bs\ge 0$ (see \cite{MR4151195}).

Now we introduce the energy regularized SPDE,
\begin{align}\label{log-reg-sac}
	du^{\delta}(t)-\sigma A u^{\delta}(t)dt+(F_{\log}^{\delta})'(u^{\delta}(t))dt=\epsilon B(u^{\delta}(t))dW(t), \; u^{\delta}(0)=u(0)
\end{align}
with a regularized logarithmic potential function
\begin{align*}
	F_{log}^{\delta}(\xi):=&\frac 1 2 (\xi+1)\ln((\xi+1)^2+\delta)+\sqrt{\delta} \arctan( \frac {\xi+1}{\sqrt{\delta}})-\xi\\
	&-(\frac 12 (\xi-1)\ln((\xi-1)^2+\delta)+\sqrt{\delta} \arctan (\frac {\xi-1}{\sqrt{\delta}})-\xi)-c\xi^2, c\ge 0.
\end{align*}
Here $\delta> 0$ is a small parameter.
The mild solution $u^{\delta}$  of \eqref{log-reg-sac} is  a adapted stochastic process satisfying that for $t\in [0,T],$
\begin{align*}
	u^{\delta}(t)=S(t)u(0)-\int_0^t S(t-s) (F_{\log}^{\delta})'(u^{\delta}(s))ds+\epsilon \int_{0}^t S(t-s)B(u^{\delta}(s))ds,\; a.s.
\end{align*}
Here $S(\cdot)$ is the semigroup generated by $A$ on $H$.
We denote
$F_{\log,2}^{\delta}(\xi):=-c\xi^2, F_{\log,1}^{\delta}=F_{\log}^{\delta}-F_{\log,2}^{\delta}.$
Furthermore, direct calculations yields that
\begin{align*}
	(F_{log}^{\delta})'(\xi)
	&=\frac 12 \ln((\xi+1)^2+\delta)-\frac 12 \ln((\xi-1)^2+\delta)-2c\xi,\\
	(F_{log}^{\delta})''(\xi) &= \frac {2(1-\xi^2)+2\delta }{((1+\xi)^2+\delta)((1-\xi)^2+\delta)}-2c.
\end{align*}
As we can see, this regularization will let the drift  coefficient satisfy the Lipschitz condition. Thus, similar to \cite{C2020},
the modified energy $H_{\delta}(v)=\frac{\sigma}{2}\|\nabla v\|^2+\int_{\mathcal{D}}F_{\log}^{\delta}(v(\xi))d\xi$ is well-defined.  However, as a cost,
the bound of the first derivative of $F_{\log}^{\delta}\sim O(\ln(\frac 1{\delta}))$ and that of higher derivatives may depend on $\frac 1\delta$
polynomially.

\subsection{A priori bound of regularized SPDE with logarithmic potential}
In this part, we present some useful properties, including the well-posedness and moment bounds in the energy space,  of \eqref{log-reg-sac}. In what follows, let $\sigma=1$ for the convenience of the presentation.
In order to apply the It\^o formula rigorously, one could use
the spectral Galerkin approximation of \eqref{log-reg-sac} and then taking limit on the  Galerkin projection parameter.  For convenience, this procedure has been omitted in this paper.

\begin{prop}\label{prop-wel}
	Let Setting \ref{set-1} hold with some $\bs\ge 0$.
	There exists a unique mild solution of \eqref{log-reg-sac} satisfying for any $p\ge 1$,
	\begin{align*}
		\sup_{t\in [0,T]}\|u^{\delta}(t)\|_{L^{2p}(\Omega;H)}^{2p} \le C(p,T,u(0)).
	\end{align*}
	Let Setting \ref{set-1} hold with some $\bs\ge 1$. In addition assume that $\epsilon \sim O(\delta^{\frac 12})$ in Case 1. It holds that
	\begin{align*}
		\sup_{t\in [0,T]}\|H_{\delta}(u^{\delta}(t))\|_{L^p(\Omega;\mathbb R)}\le C(p,T,u(0)).
	\end{align*}
\end{prop}

\begin{proof}
	The proof of existence and uniqueness of the mild solution is standard due to the Lipschitz continuity of $F_{\log}^{\delta}$ and the assumption \eqref{lip-hs} on $B$ (see, e.g., \cite{Kru14a}). It suffices to show the desired  a priori bound. By applying It\^o's formula to $\frac 12\|u^{\delta}(t)\|^2$ and using the integration by parts,  it follows that
	\begin{align*}
		\frac 12 \|u^{\delta}(t)\|^2 =&\frac 12 \|u(0)\|^2
		-\int_0^t \<\nabla u^{\delta}(s),\nabla u^{\delta}(s)\> ds
		- \int_0^t \<u^{\delta}(s),(F_{\log}^{\delta})'(u^{\delta}(s))\> ds\\
		&+\frac {\epsilon^2}2 \int_0^t \|B(u^{\delta}(s))\|_{\mathcal L_2(H)}^2 ds
		+\int_0^t \epsilon \<u^{\delta}(s),B(u^{\delta}(s))dW(s)\>.
	\end{align*}
	Notice that $\ln(\frac {(\xi+1)^2+\delta} {(\xi-1)^2+\delta})$ is monotone on $(-1,1)$ and that  $\ln(\frac {(\xi+1)^2+\delta} {(\xi-1)^2+\delta })\xi> 0$ when  $|\xi|\ge 1$.
	By taking the $p$th moment, using Burkholder's inequality, Young's inequality and H\"older's inequality, it holds that for $p\ge 1,$
	\begin{align*}
		&\tps\|u^{\delta}(t)\|_{L^{2p}(\Omega;H)}^{2p}
		+\Big\|\int_0^t\|\nabla u^{\delta}(s)\|^2ds\Big\|_{L^{p}(\Omega ;\mathbb R)}^p+\Big\|\int_0^t\<\ln(\frac {(u^{\delta}(s)+1)^2+\delta}{(u^{\delta}(s)-1)^2+\delta}),u^{\delta}(s)\>ds\Big\|_{L^p(\Omega;\mathbb R)}^p
		\\
		&\le C(p,c,T)(1+\epsilon^{2p})\int_0^t (1+\|u^{\delta}(s)\|_{L^{2p}(\Omega;H)}^{2p}) ds.
	\end{align*}
	Then Gronwall's inequality leads to
	\begin{align}
		\label{pri-first}
		&\sup_{s\in [0,T]}\|u^{\delta}(t)\|_{L^{2p}(\Omega;H)}^{2p}
		+\Big\|\int_0^T\|\nabla u^{\delta}(s)\|^2ds\Big\|_{L^{p}(\Omega ;\mathbb R)}^p\\\nonumber
		&+\Big\|\int_0^T\<\ln(\frac {(u^{\delta}+1)^2+\delta}{(u^{\delta}-1)^2+\delta}),u^{\delta}(s)\>ds\Big\|_{L^p(\Omega;\mathbb R)}^p
		\le C(p,c,T)\|u(0)\|_{L^{2p}(\Omega;H)}^{2p}.
	\end{align}
	We proceed to estimating the bound of the modified energy in Case 1 and Case 2.
	By applying the Ito's formula to $H_{\delta}(u^\delta(t))$, we have that
\begin{align*}
		H_{\delta}(u^{\delta}(t))=& H_{\delta}(u^{\delta}(0))-\int_0^t \|A u^{\delta}(s)-(F_{\log}^{\delta})'(u^{\delta }(s))\|^2ds\\
		&+\int_0^t\<(F_{log}^{\delta})'(u^{\delta}(s)), \epsilon B(u^{\delta}(s))dW(s)\>+\epsilon \int_0^t\<\nabla  u^{\delta}(s), \nabla B(u^{\delta}(s))dW(s)\>\\\nonumber
		&+ \frac 12\int_0^t\sum_{j\in \mathbb N} \epsilon^2 \<\nabla  (B(u^{\delta}(s)) e_j), \nabla( B(u^{\delta}(s)) e_j)\> ds\\\nonumber
		&+ \frac 12\int_0^t \epsilon^2 \sum_{j\in \mathbb  N}\<(F_{log}^{\delta})''(u^{\delta}(s))  B(u^{\delta}(s))e_j,  B(u^{\delta}(s))e_j\>ds.
	\end{align*}
	In Case 1,  taking $p$th moment yields that
	\begin{align*}
		& \| H_{\delta}(u^{\delta}(t))\|_{L^p(\Omega;\mathbb R)}^p
		+\Big\|\int_0^t \|A u^{\delta}(s)-(F_{\log}^{\delta})'(u^{\delta }(s))\|^2ds\Big\|_{L^p(\Omega;\mathbb R)}^p\\
		&\le C(p,T) \| H_{\delta}(u^{\delta}(0))\|_{L^p(\Omega;\mathbb R)}^p+ C(p,T)\epsilon^{2p} \int_0^t \|\nabla B(u^{\delta}(s))\|_{L^{2p}(\Omega;{\mathcal L_2(H)})}^{2p}ds
		\\
		&+C(p,T)   \int_0^t \frac {\epsilon^{2p}}{\delta^p}\|B(u^{\delta}(s))\|_{L^{2p}(\Omega;{\mathcal L_2(H)})}^{2p}ds.
	\end{align*}
	This, together with \eqref{lip-hs} and \eqref{pri-first}, implies that for $\epsilon\sim O(\delta^{\frac 12}),$
	it holds that
	\begin{align*}
		\sup_{t\in [0,T]}\|H_{\delta}(u^{\delta}(t))\|_{L^p(\Omega;\mathbb R)}\le C(p,T,u(0)).
	\end{align*}
	Similarly, in Case 2 using the fact that $h_i(\pm 1)=0$ and $h_i\in W_0^{\bs+1,\infty}[-1,1]$,  it holds that
	\begin{align*}
		& \| H_{\delta}(u^{\delta}(t))\|_{L^p(\Omega;\mathbb R)}^p
		+\Big\|\int_0^t \|A u^{\delta}(s)-(F_{\log}^{\delta})'(u^{\delta }(s))\|^2ds\Big\|_{L^p(\Omega;\mathbb R)}^p\\
		&\le C(p,T) \| H_{\delta}(u^{\delta}(0))\|_{L^p(\Omega;\mathbb R)}^p+ C(p,T,u(0))\epsilon^{2p}.
	\end{align*}
	Combining the above estimates, we complete the proof.
\end{proof}
As a consequence of Proposition \ref{prop-wel}, the averaged evolution law of the modified energy holds,
\begin{align}\label{cor-ene-evo}
	\E[H_{\delta}(u^{\delta}(t))]
	&=\E[H_{\delta}(u^{\delta}(0))]
	- \int_0^t \E[\|A u^{\delta}(s)-(F_{\log}^{\delta})'(u^{\delta }(s))\|^2]ds \\\nonumber
	&+ \frac 12 \epsilon^2\int_{0}^t \E [\|\nabla B(u^{\delta}(s))\|_{\mathcal L_2(H)}^2 ] ds\\\nonumber
	&+ \frac 12 \epsilon^2 \int_{0}^t \sum_{j\in \mathbb N} \E [\<(F_{log}^{\delta})''(u^{\delta}(s)) B(u^{\delta}(s))e_j,  B(u^{\delta}(s))e_j\>] ds.
\end{align}
It should be remarked that the energy regularization may be not available for  the SPDE with logarithmic potential driven by  the space-time white noise. This will be studied in the future.
The following is devoted to the moment bounds of the solution in the Sobolev spaces.

\begin{prop}\label{prop-h1}
	Let Setting \ref{set-1} hold with some $\bs\ge 1$. In addition assume that $\epsilon \sim O(\delta^{\frac 12})$ in Case 1. Then it holds that for $p\ge 1,$
	\begin{align*}
		\sup_{t\in [0,T]}\|u^{\delta}(t)\|_{L^{2p}(\Omega;H^1)} \le C(p,T,u(0)).
	\end{align*}
\end{prop}

\begin{proof}
	Since $\ln \xi\le C \xi^{\theta}$ if $\xi \ge 1$ and $|\ln \xi| \le C \xi^{-\theta}$ if $0<\xi<1$  for $\theta>0$, it holds that
	\begin{align*}
		&|\int_{\mathcal D}F_{log}^{\delta}(u^{\delta}(t)) dx|\\
		&\le \Big|\int_{\mathcal D}\frac 1 2 (u^{\delta}(t)+1)\ln((u^{\delta}(t)+1)^2 +\delta)
		-(\frac 12 (u^{\delta}(t)-1)\ln((u^{\delta}(t)-1)^2+\delta) dx\Big|\\
		&+ \Big|\int_{\mathcal D}\sqrt{\delta} \arctan( \frac {u^{\delta}(t)+1}{\sqrt{\delta}})
		-\sqrt{\delta} \arctan (\frac {u^{\delta}(t)-1}{\sqrt{\delta}}))dx \Big|+c \|u^{\delta}(t)\|^2.
	\end{align*}
	The boundedness of $H_{\delta}(u^{\delta}(t))$ and $\|u^{\delta}(t)\|$ in Proposition \ref{prop-wel} implies the desired result.
\end{proof}

\begin{lm}\label{lm-con-h2}
	Let Setting \ref{set-1} hold with $\bs\ge {1+\beta}$, where $\beta \in [0,1)$. Then it holds that for $p\ge 2,$
	\begin{align*}
		\|u^{\delta}(t)\|_{L^p(\Omega;H^{\beta+1})}&\le C(p,T,u(0))|\ln(\delta)|,\\
		\|u^{\delta}(t)-u^{\delta}(s)\|_{L^p(\Omega;H)} &\le C(p,T,u(0))\Big(|\ln(\delta)|(t-s)^{\frac {\beta+1}2}+\epsilon(t-s)^{\frac 12}\Big).
	\end{align*}
\end{lm}

\begin{proof}
	Applying the property that $\|S(t)(-A)^{\alpha}v\|\le C (1+t^{- \alpha })\|v\|_{H^{2\alpha}},$ $\alpha>0$ (see e.g. \cite{Kru14a}),  we have that
	\begin{align*}
		\|(-A)^{\frac {\beta+1}2}u^{\delta}(t)\|
		&\le \|(-A)^{\frac {\beta+1}2}S(t)u(0)\|
		+\|\int_{0}^t (-A)^{\frac {\beta+1}2}S(t-s)(F_{\log}^{\delta})'(u^{\delta}(s))ds\|\\
		&+ \|\int_{0}^t (-A)^{\frac {\beta+1}2}S(t-s)\epsilon B(u^{\delta}(s))dW(s)\|\\
		&\le C\|u(0)\|_{H^{\beta+1}}+ C\int_{0}^t (t-s)^{-\frac {\beta+1}2}(1+|\ln(\delta)|+\|u^{\delta}(s)\|^2) ds\\
		&+\|\int_{0}^t (-A)^{\frac {\beta+1}2}S(t-s)\epsilon B(u^{\delta}(s))dW(s)\|.
	\end{align*}
	Therefore, it suffices to bound the stochastic integral term. According to Burkholder's inequality, Minkowski inequality and \eqref{lip-hs}, we obtain
	that for any $p\ge 2$,
	\begin{align*}
		&\Big\|\int_{0}^t (-A)^{\frac {\beta+1}2}S(t-s)\epsilon B(u^{\delta}(s))dW(s)\Big\|_{L^p(\Omega; H)}\\
		&\le C\epsilon \Big(\int_{0}^t \Big(\E \Big[\Big(\sum_{i\in \mathbb N}\|S(t-s)(-A)^{\frac {\beta+1}2} B(u^{\delta}(s))e_i\|^2\Big)^{\frac p2}\Big]\Big)^{\frac 2p} ds\Big)^{\frac 12}\\
		&\le C\epsilon \Big(\int_0^t (t-s)^{-\beta} (\E[\|(-A)^{\frac 12}B(u^{\delta}(s))\|^{p}_{\mathcal L_2(H)}])^{\frac 2p}ds\Big)^{\frac 12}\le C(T,u(0),p)\epsilon.
	\end{align*}
	As a consequence, it holds that
	\begin{align*}
		\|(-A)^{\frac {\beta+1}2}u^{\delta}(t)\|_{L^p(\Omega;H)}&\le C\|u(0)\|_{H^{\beta+1}}+  C(T,u(0),p) (|\ln(\delta)|+\epsilon).
	\end{align*}
	Using the property that
	$\|(S(t)-I)v\|\le Ct^{\alpha}\|v\|_{H^{2\alpha}}$ with $\alpha\in [0,1]$ (see e.g. \cite{Kru14a}),  Burkholder's inequality and Proposition \ref{prop-wel}, it follows that for $0\le s\le t\le T,$
	\begin{align*}
		&\|u^{\delta}(t)-u^{\delta}(s)\|_{L^p(\Omega;H)}\\
		\le&
		\Big\|\int_s^t S(t-r) (F_{\log}^{\delta})'(u^{\delta}(r))dr\Big\|_{L^p(\Omega;H)}\\
		&+ \Big\|\int_0^s (S(t-r)-S(s-r)) (F_{\log}^{\delta})'(u^{\delta}(r))dr\Big\|_{L^p(\Omega;H)}\\
		&+ \Big\|\int_s^t S(t-r)\epsilon B(u^{\delta}(r))dW(r)\Big\|_{L^p(\Omega;H)}\\
		&+\Big\|\int_0^s (S(t-r)-S(s-r)) \epsilon B(u^{\delta}(r))dW(r)\Big\|_{L^p(\Omega;H)}\\
		\le& C(T,u(0),p)|\ln(\delta)|((t-s)+(t-s)^{\frac {\beta+1}2})+C(T,u(0),p)\epsilon(t-s)^{\frac 12}.
	\end{align*}
	Combining the above estimates, we complete the proof.
\end{proof}
\subsection{Limit of regularized SPDE with logarithmic potential}
Now, we are in a position to present the strong convergence of the regularized models. More precisely, we prove that
in Case 1 ($\epsilon \to 0,\delta \to0$), the limit equation of \eqref{log-reg-sac} is
\begin{align}\label{det-sac}
	du_{det}(t)-A u_{det}(t)dt+(F_{\log})'(u_{det}(t))dt=0, \;
\end{align}
while in Case 2 ($\epsilon\sim O(1),\delta\to0$), the limit equation is
\begin{align}\label{sto-sac}
	du(t)-A u(t)dt+(F_{\log})'(u(t))dt=\epsilon B(u(t))dW(t).
\end{align}The well-posedness of \eqref{sto-sac} has been shown in \cite{MR4151195}.
Our strategy is constructing a sequence of solutions $u^{\delta}$ and showing its convergence (see, e.g., \cite{C2020}).
The key step of our approach lies on the upper bound estimate of the probability that the solution escapes the Stampacchia maximum bound, i.e., $u^{\delta}(t,x)\notin [-1,1]$.

To this end, we first show the tail estimate of $u^{\delta},$ which indicates that $(1-u^{\delta})^{-}$ and $(u^{\delta}+1)^{-}$ are small in the sense of $L^2(\mathcal O),$
by  setting  up the It\^o formula of some suitable functionals as in
\cite{MR3427686}.
We define the upper bound functional  $\int_{\mathcal D}|f^1(v(\xi))|^2d\xi:=\int_{\mathcal D} |(1-v(\xi))^-|^2 d\xi$ and the lower bound functional $\int_{\mathcal D}|f^2(v(\xi))|^2d\xi:=\int_{\mathcal D} |(v(\xi)+1)^{-}|^2 d\xi$.
Let $\kappa$ be a small positive parameter.
Define the regularization approximations of $f^j$ by  $f_{\kappa}^{j},j=1,2$ such that $f_{\kappa}^j\in \mathcal C^2(\mathbb R)$ satisfying
\begin{align}\label{con-reg}
	&f_{\kappa}^j(\xi)=0, \; \xi\in [-1,1], \; f_{\kappa}^{j}(\xi)=0 \; \text{otherwise}, \\\nonumber
	&|f_{\kappa}^{j}(\xi)|\le c_1(1+|\xi|), \; |(f_{\kappa}^{j})'(\xi)|\le c_1, \\\nonumber
	& |(f_{\kappa}^{j})''(\xi)|\le c_1, |f^{j}_{\kappa}(\xi)(f_{\kappa}^{j})'(\xi) |\le c_1.
\end{align}
Here $c_1$ is some positive constant.
We in addition suppose that
\begin{align*}
	f^1_{\kappa}(\xi)= \xi-1, \;\text {if} \;  1-\xi<-\kappa,\;
	f^1_{\kappa}(\xi)=0, \; \text{if} \; 1-\xi\ge 0.
\end{align*}
and
\begin{align*}
	f^2_{\kappa}(\xi)= -\xi-1, \;\text {if} \;  \xi+1<-\kappa,\;
	f^2_{\kappa}(\xi)=0, \; \text{if} \; \xi+1\ge  0.
\end{align*}
It can be seen that the set of $f^j_{\kappa}$ is not empty by using a $C^2$ interpolation (see, e.g., \cite{MR3427686}). Below is one concrete example.
\begin{ex}
	\begin{equation}
		f_{\kappa}^1(\xi)=\left\{ \begin{aligned}
			&\xi-1 &\; \text{if}\; 1-\xi <-\kappa, \\
			&-\frac {3}{\kappa^4}(1-\xi)^5-\frac {8}{\kappa^3}(1-\xi)^4-\frac 6 {\kappa^2} (1-\xi)^3 & \; \text{if}\; -\kappa \le 1-\xi<0,\\
			&0 & \; \text{if} \; 1-\xi \ge 0,
		\end{aligned}
		\right.
	\end{equation}
	\begin{equation}
		f_{\kappa}^2(\xi)=\left\{ \begin{aligned}
			&-(\xi+1) &\; \text{if}\; \xi+1<-\kappa, \\
			&-\frac {3}{\kappa^4}(\xi+1)^5-\frac {8}{\kappa^3}(\xi+1)^4-\frac 6 {\kappa^2} (\xi+1)^3 & \; \text{if}\; -\kappa \le  \xi+1< 0,\\
			&0 & \; \text{if} \; \xi+1\ge 0.
		\end{aligned}
		\right.
	\end{equation}
\end{ex}

\begin{lm}\label{lm-tail}
	Let Setting \ref{set-1} hold with some $\bs>\frac d2$. Then the following tail estimates holds,
	\begin{align*}
		&\mathbb P(\sup_{s\in[0,T]}|u^{\delta}(s)|_E>1+\delta_0)\to 0, \; \text{as} \;  \epsilon  \sim o(\delta_0)\to 0
	\end{align*}
	in Case 1, and
	\begin{align*}
		\mathbb P(\sup_{s\in [0,T]} |u^{\delta}(s)|_E>1)=0.
	\end{align*}
	in Case 2.
\end{lm}

\begin{proof}

	
	Since the derivative of $F_{\log}^{\delta}$ is finite, one can show that $u^{\delta}\in E, a.s.$ by a standard procedure (see, e.g., \cite{BCH18,CH18}). However, its $E$-norm may be not uniformly bounded with respect to $\delta$ by this argument.
	Instead,
	we will consider the evolution of $\|f^1_{\kappa}(u(t))\|^2$ and $\|f^2_{\kappa}(u(t))\|^2$, and then take limits on $\kappa.$
	Let us illustrate the procedures to estimate  $\|f^1(u(t))\|^2$ since that of the case $\|f^2(u(t))\|^2$ is similar.
	
	Applying the It\^o formula  to $\frac 12 \|f^1_{\kappa}(u(t))\|^2$, we obtain
	\begin{align*}
		&\frac 12\|f^1_{\kappa}(u^{\delta}(t))\|^2\\
		&~= \frac 12\|f^1_{\kappa}(u(0))\|^2
		+  \int_0^t \<(f^1_{\kappa})'(u^{\delta}) f^1_{\kappa}(u^{\delta}), Au^{\delta}(s)-(F_{log}^{\delta})'(u^{\delta}(s))\> ds\\
		&~~+\int_{0}^t  \epsilon  \<(f^1_{\kappa})'(u^{\delta}(s))f^1_{\kappa}(u^{\delta}(s)),B(u^{\delta}(s))dW(s)\>\\
		&~~+\int_0^t \frac 12 \epsilon^2 \sum_{j\in \mathbb N}\<(f^1_{\kappa})''(u^{\delta}(s))f_{\kappa}^1(u^{\delta}(s))+|(f^1_{\kappa})'(u^{\delta}(s))|^2)B(u^{\delta}(s))e_j,B(u^{\delta}(s))e_j\>ds.
	\end{align*}
	Taking $\kappa \to 0$, thanks to Propositions \ref{prop-wel}, \eqref{con-reg} and \eqref{prop-h1}, we have that
	\begin{align*}
		&\frac 12\|f^1(u^{\delta}(t))\|^2\\
		&~= \frac 12\|f^1(u(0))\|^2
		-\int_0^t \<\mathbb I_{\{u^{\delta}(s)>1\}} (u^{\delta}(s)-1),(F_{log}^{\delta})'(u^{\delta}(s))\> ds\\
		&~~- \int_0^t \|\mathbb I_{\{u^{\delta}(s)>1\}} \nabla u^{\delta}(s)\|^2ds+\int_{0}^t  \epsilon  \<\mathbb I_{\{u^{\delta}(s)>1\}}(u^{\delta}(s)-1),B(u^{\delta}(s))dW(s)\>\\
		&~~+\int_0^t \frac 12 \epsilon^2 \sum_{j\in \mathbb N}  \<\mathbb I_{\{u^{\delta}(s)>1\}}  B(u^{\delta}(s))e_j,B(u^{\delta}(s))e_j\>ds.
	\end{align*}
	Taking supreme over $t\in [0,t_1]$ for $t_1\le T$,  then
	taking $p$th moment and using Burkerhold's inequality, we obtain that
	\begin{align*}
		&\E [\sup_{t\in [0,t_1]}\|f^1(u^{\delta}(t))\|^{2p}]+\E \Big[\Big(\int_0^{t_1} \mathbb I_{\{u^{\delta}(s)>1\}}\|\nabla u^{\delta}(s)\|^2ds\Big)^p\Big]\\
		&+ \E \Big[\Big(\int_0^{t_1} \mathbb I_{\{u^{\delta}(s)>1\}}\<(u^{\delta}(s)-1),(F_{log,1}^{\delta})'(u^{\delta}(s))\> ds\Big)^p\Big]\\
		&\le C(p) \E [\|f^1(u^{\delta}(0))\|^{2p}]
		+ C(p,T)\epsilon^{2p}\int_0^{t_1} (1+\E[\|B(u^{\delta}(s)))\mathbb I_{\{u^{\delta}(s)>1\}}\|_{\mathcal L_2(H)}^{2p} ]) ds.
	\end{align*}
	As a consequence, it holds that in Case 1,
	\begin{align}\label{pri-case1}
		\E [\sup_{t\in [0,T]}\|f^1(u^{\delta}(t))\|^{2p}]\le C(T,p)(\|f^1(u(0))\|^{2p}+\epsilon^{2p}\|u(0)\|^{2p})
	\end{align}
	and that in Case 2,
	\begin{align}\label{pri-case2}
		\E [\sup_{t\in [0,T]}\|f^1(u^{\delta}(t))\|^{2p}]\le C(T,p)(\|f^1(u(0))\|^{2p}).
	\end{align}
	By repeating the above procedures to $f^2$ and using the fact that $f^j(u(0))=0,j=1,2$, it is not hard to see that in Case 2,
	$$\mathbb P(\sup_{s\in [0,T]} |u^{\delta}(s)|_E>1)=0.$$
	Let $\delta_0>0$ be a small number. Denote $\mathcal A_{s}=\{x\in\mathcal D | u^{\delta}(s)>1+\delta_0\}$, $\mathcal B_s=\{x\in\mathcal D | u^{\delta}(s)<-1-\delta_0\}.$
	In Case 1, notice that for any $s\in [0,T]$,
	\begin{align}\label{cor-use}
		|\mathcal A_s|+|\mathcal B_s|\le \frac {\|f^1(u^{\delta}(s))\|^2+\|f^2(u^{\delta}(s))\|}{\delta_0^2}\le \eta_s \frac {\epsilon^2}{\delta_0^2},
	\end{align}
	where $\eta_s$ is a positive stochastic process with any finite $p$th moment.
	As a consequence, the Lebesgue measures  $|\mathcal A_s|,|\mathcal B_s| \to 0, a.s.$ as $\epsilon \sim o(\delta_0)\to 0,$
	which completes the proof.
\end{proof} Next, we provide the strong convergence analysis of \eqref{log-reg-sac}.

\begin{tm}\label{tm-con1}
	Let Setting \ref{set-1} hold with some $\bs>\frac d2$, $p\ge 1$.  It holds that in Case 1,
	\begin{align*}
		\E[\|u^{\delta}(t)-u_{\det}(t)\|^{2p}]\le C(\delta^2+\epsilon^2+(1+\ln^2(\delta)) \frac {\epsilon^2}{\delta_0^2})^p,
	\end{align*}
	where $\delta_0\sim O(\delta)$ is the small parameter in Lemma \ref{lm-tail}, and that in Case 2,
	\begin{align*}
		\E[\|u^{\delta}(t)-u(t)\|^{2p}]\le C\delta^{2p},
	\end{align*}
\end{tm}

\begin{proof}
	We only present the details of $p=1$ for simplicity.
	We first show the convergence in Case 1.
	Denote by $u_{det}(t)$ the exact solution of the deterministic Allen-Cahn equation with $F_{log}$, i.e.,
	\begin{align*}
		du_{det}(t)-A u_{det}(t)dt+F'_{log}(u_{det}(t))dt=0, \;
	\end{align*}
	which
	satisfies the  maximum bound principle
	\begin{align*}
		|u_{\det}(t,x)|< 1, \;\text{if}\;  |u_{\det}(0,x)|< 1,
	\end{align*}
	and $u(t)\in  H^{\bs}(\mathcal D)$ if $u(0)\in H^{\bs}(\mathcal D)$. Considering the difference between $u^{\delta}$ and $u$,  taking expectation and applying Burkholder's inequality , we have that
	\begin{align*}
		&\frac 12 \E [\|u^{\delta}(t)-u_{det}(t)\|^2]+\int_0^t\E[ \|\nabla (u^{\delta}(s)-u_{det}(s))\|^2]ds \\
		&=-\int_0^t \E [\<(F_{\log}^{\delta})'(u^{\delta}(s))- (F_{\log})'(u_{det}(s)), u^{\delta}(s)-u_{det}(s)\>] ds\\
		&+\int_0^t \frac {\epsilon^2}2 \E [\| B(u^{\delta }(s)) \|_{\mathcal L_2(H)}^2] ds+ \epsilon \E [\int_0^t  \<u_{det}(s)-u^{\delta}(s),B(u^{\delta}(s))dW_s\>]\\
		&\le -\int_0^t \E [\int_{\mathcal A_s\cup \mathcal B_s} (F_{\log}^{\delta})'(u^{\delta}(s))- (F_{\log})'(u_{det}(s)))(u^{\delta}(s)-u_{det}(s)) dx] ds\\
		& -\int_0^t \E [\int_{\mathcal A_s^c\cap \mathcal B_s^c}((F_{\log}^{\delta})'(u^{\delta}(s))- (F_{\log})'(u_{det}(s)))(u^{\delta}(s)-u_{det}(s))dx] ds\\
		&+ \int_0^t \frac {\epsilon^2} 2\E [\| B(u^{\delta }(s))\|_{\mathcal L_2(H)}^2] ds=:
		II_1+II_2+II_3.
	\end{align*}
	Since $II_3\sim O(\epsilon^2),$ it suffices to estimate $II_1$-$II_2.$
	Thanks to Proposition \ref{prop-wel} and the proof of Lemma \ref{lm-tail}, applying H\"older's inequality, it holds that
	\begin{align*}
		|II_1|
		&\le C \int_0^t\E [\|u^{\delta}(s)-u_{det}(s)\|^2] ds\\
		&+\int_{0}^t \E\Big[\|(u^{\delta}(s)-u_{det}(s))\| \sqrt{\int_{\mathcal A_s\cup \mathcal B_s}|(F_{\log}^{\delta})'(u^{\delta}(s))- (F_{\log})'(u_{det}(s)))|^2dx }\Big]ds\\
		&\le C \int_0^t\E [\|u^{\delta}(s)-u_{det}(s)\|^2] ds+ C(T)(1+\ln^2(\delta)) \frac {\epsilon^2}{\delta_0^2}.
	\end{align*}
	Thanks to the fact that
	\begin{align*}
		(F_{\log}^{\delta})'(\xi_1)- (F_{\log}^{\delta})'(\xi_2)&=\frac {2(1-\chi^2)+2\delta }{((1+\chi)^2+\delta)((1-\chi)^2+\delta)}(\xi_1-\xi_2),
	\end{align*}
	where $\chi=\theta \xi_1+(1-\theta)\xi_2$ for some $\theta\in (0,1)$, it follows that if $|\chi|\in [1,1+\delta_0]$, then
	\begin{align*}
		|2(1-\chi^2)|\le 2((1+\delta_0)^2-1)\le 4\delta_0+2\delta_0^2.
	\end{align*}
	In the following, we take $\delta_0\sim O(\delta).$
	Notice that
	\begin{align*}
		II_2&\le C \int_0^t\E[\|u^{\delta}(s)-u_{det}(s)\|^2] ds\\
		&-\int_0^t \E [\int_{\mathcal A_s^c\cap \mathcal B_s^c} ((F_{\log,1}^{\delta})'(u^{\delta}(s))- (F_{\log,1}^{\delta})'(u_{det}(s)))(u^{\delta}(s)-u_{det}(s)) dx] ds \\
		&-\int_0^t \E [\int_{\mathcal A_s^c\cap \mathcal B_s^c} ((F_{\log,1}^{\delta})'(u_{det}(s))- ((F_{\log,1})'(u_{det}(s)))(u^{\delta}(s)-u_{det}(s)) dx] ds,
	\end{align*}
	where $(F_{\log,1})'(\xi):=\ln(1+\xi)-\ln(1-\xi).$
	Therefore, we have that
	\begin{align*}
		II_2&\le C \int_0^t \E[\|u^{\delta}(s)-u_{det}(s)\|^2] ds  +C\delta^2.
	\end{align*}
	Combining the above estimates with Gronwall's inequality yields that
	\begin{align*}
		\E[\|u^{\delta}(t)-u_{det}(t)\|^2]\le C(\delta^2+\epsilon^2+(1+\ln^2(\delta)) \frac {\epsilon^2}{\delta_0^2}).
	\end{align*}
	Similar steps as in Case 1, together with Lemma \ref{lm-tail},  yield the desired result in Case 2.
\end{proof}

To end this section, we show the random effect on the behaviors of energy thanks to  Proposition \ref{lm-con-h2} and Theorem \ref{tm-con1}.

\begin{prop}\label{prop-con}
	Let Setting \ref{set-1} hold with some $\bs=1+\beta>\frac d2$, $\beta\in (0,1)$ and $\delta_0\sim O(\delta)$. Then it holds that in Case 1, for $\epsilon \sim o(\delta_0),$
	\begin{align}
		\E[|H(u^{\delta}(t))-H(u_{\det}(t))|]&\le C(\delta+\epsilon+|\ln(\delta)|\frac {\epsilon}{\delta_0})^{\frac {(\bs-1)}{\bs}}|\ln(\delta)|,
	\end{align}
	and that in Case 2,
	\begin{align}
		\E[|H(u^{\delta}(t))-H(u(t))|] &\le C\delta^{\frac {(\bs-1)}{\bs}}|\ln(\delta)|.
	\end{align}
\end{prop}
\begin{proof}
	Using the Sobolev interpolation inequality 
 $$\|(-A)^{\frac 12} v\|\le C\|v\|^{\frac {\bs-1}{\bs}}\|(-A)^{\frac \bs 2} v\|^{\frac {1}{\bs}}, $$ and applying Proposition \ref{lm-con-h2} and Theorem \ref{tm-con1}, it follows that in Case 1,
	\begin{align*}
		\|\nabla (u^{\delta}(t)-u_{\det}(t))\|_{L^2(\Omega; H)}
		&\le C(\delta+\epsilon+|\ln(\delta)|\frac {\epsilon}{\delta_0})^{\frac {(\bs-1)}{\bs}}|\ln(\delta)|^{\frac 1{\bs}}.
	\end{align*}
	and that in Case 2,
	\begin{align*}
		\|\nabla (u^{\delta}(t)-u(t))\| _{L^2(\Omega; H)} &\le C\delta^{\frac {(\bs-1)}{\bs}} |\ln(\delta)|^{\frac 1{\bs}}.
	\end{align*}
	On the other hand,
	by Proposition \ref{tm-con1}, H\"older's and Young's inequalities,
	we have that for $v(t)=u(t)$ or $u_{det}(t),$
	\begin{align*}
		&\E [\|F_{\log}^{\delta}(u^{\delta}(t))-F_{\log}(v(t))\|_{L^1}]\\
		&\le \E [\|F_{\log}^{\delta}(v(t))-F_{\log}(v(t))\|_{L^1}]
		+\E [\|F_{\log}^{\delta}(u^{\delta}(t))-F_{\log}^{\delta}(v(t))\|_{L^1}]\\
		&\le C(\delta+\E [\|v(t)\|^2])
		+\E [\|F_{\log}^{\delta}(u^{\delta}(t))-F_{\log}^{\delta}(v(t))\|_{L^1}]\\
		&\le C(\delta+(1+\ln(\delta)) \sqrt{\E [\|v(t)-u^{\delta}(t)\|^2]}).
	\end{align*}
	Thus, we get that in Case 1,
	\begin{align*}
		\E [\|F_{\log}^{\delta}(u^{\delta}(t))-F_{\log}(u_{det}(t))\|_{L^1}]\le  C(\delta+(1+|\ln(\delta)|)(\epsilon+\delta+|\ln(\delta)| \frac {\epsilon}{\delta_0}) )
	\end{align*}
	and that in Case 2,
	\begin{align*}
		\E [\|F_{\log}^{\delta}(u^{\delta}(t))-F_{\log}(u(t))\|_{L^1}]\le C(\delta+(1+|\ln(\delta)|)\delta).
	\end{align*}
	Combining the above estimates, we complete the proof.
\end{proof}

\section{Numerical  approximation}
\label{sec-3}

There have been a lot of numerical results on the discretization for SPDEs with Lipschitz nonlinearities. But less attention has been paid on parabolic SPDEs with logarithmic potentials. Thanks to the regularized models in section \ref{sec-2}, we are able to do numerical analysis for \eqref{log-sac}.
Although the implicit discretization is a suitable choice of \eqref{log-sac}, it often requires to solve nonlinear algebraic equations with randomness. It is still highly desirable to develop stable numerical approximations for \eqref{log-sac} which could be solved explicitly to save computational costs.
However, due to the singularity of the drift coefficient, it is not easy to achieve this goal.  

 In this section, we propose a semi-implicit scheme for the considered model. We present a detailed analysis for the stabilized full discretization. Similar arguments are also applicable for the classical implicit full discretizations.
In the spatial direction, we take the spectral Galerkin method  for simplicity.

\subsection{Semi-implicit scheme with stabilization term}
\label{sub-sec-imp}

In order to propose a stable semi-implicit scheme, we add a small stabilization term into the numerical discretization. For convenience, let us use the following stabilized scheme with a parameter $\alpha\ge 0$,
\begin{align}\label{sta-semi}
	u_{j+1}^N=u_j^N+(A+2c)\tau u_{j+1}^N-(F^{\delta}_{\log,1})'(u_j^N)\tau-\alpha(u_{j+1}^N-u_{j}^N)\tau+\epsilon P^N B(u_j^N)\delta W_j.
\end{align}
which uses the backward Euler discretization for the linear unbounded part and the forward Euler discretization for the potential and stochastic terms.
Here $u_0^N=P^N u(0)$, $S_{\tau}=(I-A\tau)^{-1},$ $\delta W_j=W(t_{j+1})-W(t_j), t_j=j\tau$,  $j\le J$, $J\tau=T$, and
$N$ is the parameter of the spectral Galerkin method, i.e., the dimension of the projection on to the span of first $N+1$ eigenspace.
Note that the above full discretization of \eqref{log-sac} can be rewritten as
\begin{align}\label{imp-euler}
	u^N_{j+1}=S_{\tau}\Big(u^N_{j}-P^N(F_{\log,1}^{\delta})'(u^{N}_{j})\tau-(\alpha-2c)\tau u_{j+1}^N+\alpha u_j^N\tau+P^N\epsilon B(u_j^N) \delta W_j\Big).
\end{align}

The key steps to derive the convergence of stable numerical scheme lies on the regularity estimates in time and space.

\begin{lm}\label{dis-h1}
	Let Setting \ref{set-1} hold with some $\bs\ge 0$ and $p\ge 1$.  There exists $C:=C(p,T,u(0))$ such that
	\begin{align*}
		\sup_{N\in \mathbb N^+}\sup_{j\le J}\|u_j^N\|_{L^{2p}(\Omega;H)}\le C(1+|\ln(\delta)|\tau^{\frac 12}). \;
	\end{align*}
\end{lm}

\begin{proof}
	
	For convenience, we only give the proof for $p=1$.
	By taking $H$-inner product on the both sides of \eqref{sta-semi} with $u_{j+1}^N$, and using integration by parts and \eqref{imp-euler}, one can obtain that
\begin{align*}
		&\frac 12 (1+\alpha\tau) \|u_{j+1}^N\|^2 +\|\nabla u_{j+1}^N\|^2\tau
	\\
	&\le \frac 12 (1+\alpha\tau)\|u_j^N\|^2
	+2c \|u_{j+1}^N\|^2\tau	-\<(F^{\delta}_{\log,1})'(u_j^N),u_{j+1}^N\>\tau
	+\epsilon\<B(u_j^N)\delta W_j,u_{j+1}^N\>.
\end{align*}

Using H\"older's and Young' inequality, as well as the property that $\|(S_{\tau}-I)v\|\le C\tau^{\alpha}\|v\|_{H^{2\alpha}},$ $ \alpha\in [0,1],$ it follows that
\begin{align*}
	&\frac 12 (1+\alpha\tau) \|u_{j+1}^N\|^2 +\|\nabla u_{j+1}^N\|^2\tau
	\\
	\le& \frac 12 (1+\alpha\tau)\|u_j^N\|^2
	+2c \|u_{j+1}^N\|^2\tau	-\<(F^{\delta}_{\log,1})'(u_j^N),u_{j}^N\>\tau-
	\\
	&\<(F^{\delta}_{\log,1})'(u_j^N),(S_{\tau}-I)u_{j}^N\>\tau-\<(F^{\delta}_{\log,1})'(u_j^N),-S_{\tau}P^N(F^{\delta}_{\log,1})'(u_j^N)\tau 
	\\
	&-S_{\tau}(\alpha-2c)u_{j+1}^N+\alpha u_j^N\tau\>\tau
	-\<(F^{\delta}_{\log,1})'(u_j^N),S_{\tau}\epsilon B(u_j^N)\delta W_j \>\tau
	\\
	&+\epsilon\<B(u_j^N)\delta W_j,S_{\tau}u_{j}^N\>+\epsilon\<B(u_j^N)\delta W_j,-S_{\tau} (F_{\log,1}^{\delta})'(u_j^N)\tau
	\\
	&-(\alpha-2c)\tau u_{j+1}^N+\alpha u_j^N\tau \>+\epsilon^2\<B(u_j^N)\delta W_j,S_{\tau}B(u_j^N)\delta W_j\>\\
	\le &	\frac 12 (1+\alpha\tau)\|u_j^N\|^2
	+2c \|u_{j+1}^N\|^2\tau
+C|\ln(\delta)|^2\tau^2	+C(\|u_{j+1}^N\|^2+\|u_j^N\|^2)\tau^2\\
	&-\<(F^{\delta}_{\log,1})'(u_j^N),S_{\tau}\epsilon B(u_j^N)\delta W_j \>\tau
	+\epsilon\<B(u_j^N)\delta W_j,S_{\tau}u_{j}^N\>\\
	&+\epsilon\<B(u_j^N)\delta W_j,-S_{\tau} (F_{\log,1}^{\delta})'(u_j^N)\tau-(\alpha-2c)\tau u_{j+1}^N+\alpha u_j^N\tau \>
	\\
	&+\epsilon^2\<B(u_j^N)\delta W_j,S_{\tau}B(u_j^N)\delta W_j\>.
\end{align*}
	Taking expectation and using Gronwall's inequality, it follows that
	\begin{align}
		\E[\|u_{j+1}^N\|^{2}]+\E[\sum_{i=0}^{j}\|\nabla u_{j+1}^N\|^2]\tau &\le C(T)\|u_0^N\|^2(1+\ln^2(\delta)\tau+\epsilon^2).
	\end{align}

\end{proof}

Now, we are in a position to present its moment bounds and regularity estimates in Sobolev norms.

\begin{lm}\label{lm-dis-h2}
	Let Setting \ref{set-1} hold with some $\bs=1+\beta$, $\beta\in [0,1)$  and $p\ge 2.$
	It holds that
	\begin{align*}
		\sup_{N\in \mathbb N^+}\sup_{j\le J-1}\|u_{j+1}^N-u_j^N\|_{L^p(\Omega; H)}&\le C(T,u(0),p)|\ln(\delta)|(\tau^{\frac {\beta+1}2}+\epsilon\tau^{\frac 12} ), \;\\
		\sup_{N\in \mathbb N^+}\sup_{j\le J}\|u_{j+1}^N\|_{L^p(\Omega;H^{\bs})}&\le C(T,u(0),p,\beta)(\epsilon+|\ln(\delta)|).
	\end{align*}
\end{lm}

\begin{proof}
	We first show the regularity  estimate in space direction.
Using the fact that
$$\|S_{\tau}^{k} (-A)^{\alpha}v\|\le C(1+(k\tau)^{-\alpha}) \|v\|, \alpha\in (0,1), k\in\mathbb N^+,$$
and Burkholder's inequality, we have that for $\beta\in (0,1),$
	\begin{align*}
		&\|(-A)^{\frac{\beta+1}2}u_{j+1}^N\|_{L^p(\Omega;H)}\\
		&=\|S_{\tau}^{j+1}(-A)^{\frac{\beta+1}2}u_0^N\|+\|P^N\sum_{i=0}^j (-A)^{\frac{\beta+1}2} S_{\tau}^{j+1-i} (F_{\log,1}^{\delta})'(u_i^N)\|_{L^p(\Omega;H)}\tau \\
		&+\|P^N\sum_{i=0}^j(-A)^{\frac {\beta+1}2}S_{\tau}^{j+1-i}(\alpha-2c)u_{j+1}^N+\alpha u_j^N\|_{L^p(\Omega;H)}\tau \\
		&+\|P^N\sum_{i=0}^j (-A)^{\frac{\beta+1}2} S_{\tau}^{j+1-i} \epsilon B(u_i^N)\delta W_j\|_{L^p(\Omega;H)}\\
		&\le C\|u(0)\|_{H^{\frac {\beta+1}2}}
		+C\sum_{i=0}^j (t_{j+1}-t_i)^{-\frac {\beta+1}2}\tau
		(1+|\ln (\delta)|)\\
		&+C\Big(\sum_{i=0}^j (t_{j+1}-t_i)^{-\beta}
		\tau (\E[\|(-A)^{\frac 12}B(u_{i}^N)\|_{\mathcal L_2(H)}^{p}])^{\frac 2p} \Big)^{\frac 12}\epsilon \\
		&\le C(T,u(0),p,\beta)(\epsilon+|\ln(\delta)|).
	\end{align*}
		Next we prove the discrete time regularity estimate.
	By Burkholder's inequality, Lemma \ref{dis-h1}, and the property that $\|(S_{\tau}-I)v\|\le C\tau^{\alpha}\|v\|_{H^{2\alpha}}, \alpha\in [0,1],$ we obtain
	\begin{align*}
		&\|u_{j+1}^N-u_j^N\|_{L^p(\Omega;H)}\\
		&\le \|(S_{\tau}-I)u_j^N\|_{L^p(\Omega;H)}+\tau \|(F_{\log,1}^{\delta})'(u_{j}^N)\|_{L^p(\Omega;H)}+\epsilon \|B(u_j^N)\delta W_j\|_{L^p(\Omega;H)}\\
		&+C(\|u_j^N\|_{L^p(\Omega;H)}+\|u_{j+1}^N\|_{L^p(\Omega;H)})\tau\\
		&\le C\tau^{\frac 12}\|u_j^N\|_{L^p(\Omega;H^1)}
		+C\tau|\ln(\delta)|+C\epsilon \|B(u_j)\|_{L^p(\Omega;\mathcal L_2(H))}\tau^{\frac 12} \\
		&\le C(T,u(0),p)|\ln(\delta)|(\tau^{\frac {\beta+1}2}+\epsilon\tau^{\frac 12} ).
	\end{align*}

\end{proof}

\subsection{Strong convergence analysis}
Now we present its convergence result of the considered scheme under $L^{2p}(\Omega;H)$ with $p=1$. Same convergence rate result still holds for other $p>1.$

\begin{tm}\label{main-con}
	Let Setting \ref{set-1} hold with some $\bs=1+\beta>\frac d2, \beta \in [0,1)$.
	It holds that for $\delta_0\sim O(\delta), $
	\begin{align*}
		\sup_{j\le J}\|u^{\delta}(t_{j})-u_j^N\|_{L^{2}(\Omega;H)}^2
		\le&
		C(T,u(0))|\ln(\delta)|^2\Big(\frac{\lambda_N^{-{\beta+1}}}{\delta^2}+\frac {\|f^1(u_0^N)\|^{2}+\|f^2(u_0^N)\|^{2}}{\delta_0^2}\\
       &+\frac {\epsilon^2}{\delta_0^2}+\frac {\tau^{\beta+1}}{\delta_0^{2}\delta^2}+\frac {\tau^2}{\delta_0^2\delta^2}+\frac {\tau\epsilon^2}{\delta_0^2\delta^2}\Big)
	\end{align*}
	in Case 1, and
	\begin{align*}
		&\sup_{j\le J}\|u^{\delta}(t_{j})-u_j^N\|_{L^{2}(\Omega;H)}^2\\
		&\le
		C(T,u(0))|\ln(\delta)|^2\Big(\frac{\lambda_N^{-{\beta+1}}}{\delta^2}+\frac {\|f^1( u_0^N)\|^{2}+\|f^2(u_0^N)\|^{2}}{\delta_0^2}+\frac {\tau^{\beta+1}}{\delta_0^{2}\delta^2}+\frac {\tau^2}{\delta_0^2\delta^2}\Big)
	\end{align*}
	in Case 2.
\end{tm}

\begin{proof}
	We only present the proof of Case 1 since that of Case 2 is similar.
	Notice that $u^{\delta}(t_{j+1})=(I-P^N)u^{\delta}(t_{j+1})+P^Nu(t_{j+1}).$
	The error bound of $(I-P^N)u^{\delta}(t_{j+1})$ is standard due to the fact that
	\begin{align}\label{spe-err}
		\|(I-P^N)v\|\le C\lambda_N^{-\frac \bs 2}\|v\|_{H^{\bs}}.
	\end{align}
	It suffices to derive an iterative formula for $Er_{j+1}:=P^Nu^{\delta}(t_{j+1})-u^N_{j+1}$.
	To this end, we consider $\E[\<Er_{j+1}-Er_j,Er_{j+1}\>]$ and obtain that
	\begin{align*}
		\frac 12\E [\|Er_{j+1}\|^2]
		&\le \frac 12 \E[\|Er_j\|^2]-\int_{t_j}^{t_{j+1}}\E [\<\nabla (P^Nu^{\delta}(s)-u^N_{j+1}),\nabla Er_{j+1}\>] ds \\
		&-\int_{t_j}^{t_{j+1}}\E[\<(F_{\log,1}^{\delta})'(u^{\delta}(s))-(F_{\log,1}^{\delta})'(u_{j}^N),Er_{j+1}\>]ds\\
		&+\E[\int_{t_j}^{t_{j+1}}\epsilon \<Er_{j+1}-Er_j,(B(u^{\delta}(s))-B(u_j^N))dW(s))\>]\\
		&+\E [\int_{t_j}^{t_{j+1}}\<-2cu^{\delta}(s)+2cu_{j+1}^N-\alpha(u_{j+1}^N-u_j^N),Er_{j+1}\>ds]\\
		&=:\frac 12 \E[\|Er_j\|^2]+III_1+III_2+III_3+III_4.
	\end{align*}
	By using Young's inequality and the mild formulation of $u^{\delta}(s)$, we get that for small $\kappa<1$,
	{\small
		\begin{align*}
			III_1
			&\le -(1-\kappa)\int_{t_j}^{t_{j+1}}\E [\|\nabla Er_{j+1}\|^2] ds\\
			&+
			C(\kappa)\int_{t_j}^{t_{j+1}} \E \Big[\Big\|(-A)^{\frac 12}\int_{0}^s (S(t_{j+1}-r)-S(s-r)) (F_{\log}^{\delta})'(u^{\delta}(r))dr\Big\|^2\Big] ds\\
			&+C(\kappa)\int_{t_j}^{t_{j+1}} \E \Big[\Big\|(-A)^{\frac 12}\int_{s}^{t_{j+1}} (S(t_{j+1}-r)(F_{\log}^{\delta})'(u^{\delta}(r))dr\Big\|^2\Big] ds\\
			&+C(\kappa)\int_{t_j}^{t_{j+1}} \E \Big[\Big\|(-A)^{\frac 12}\int_{0}^{s} (S(t_{j+1}-r)-S(s-r))\epsilon B(u^{\delta}(r))dW(r)\Big\|^2\Big] ds\\
			&+C(\kappa)\int_{t_j}^{t_{j+1}} \E \Big[\Big\|(-A)^{\frac 12}\int_{s}^{t_{j+1}} (S(t_{j+1}-r)\epsilon B(u^{\delta}(r))dW(r)\Big\|^2\Big] ds.
		\end{align*}
	}By Proposition \ref{prop-h1} and the property  that $\|A^{-\alpha}(S(t)-I)v\|\le Ct^{\alpha}\|v\|$, $\alpha\in (0,1)$, we obtain that
	\begin{align*}
		III_1&\le -(1-\kappa)\int_{t_j}^{t_{j+1}}\E [\|\nabla Er_{j+1}\|^2] ds+
		C(\kappa,u(0),T)\tau^{2} (\epsilon^2+\ln^2(\delta)).
	\end{align*}
	Similarly, one can obtain that
	\begin{align*}
		III_4&\le CEr_{j+1}^2\tau+C|\ln(\delta)|^2(\tau^{\frac {\beta+1}2}+\epsilon\tau^{\frac 12}+\lambda_N^{-\frac {\beta+1}2} )^2.
	\end{align*}

	For the term $III_2$, by Young's inequality, \eqref{spe-err}, \eqref{cor-use} and Lemma \ref{lm-con-h2}, it can be seen that
	\begin{align*}
		III_2\le& C(\kappa) \E[\|Er_{j+1}\|^2]\tau
		+C(\kappa,T,u(0))\tau |\ln(\delta)|^2(\frac {\epsilon^2}{\delta^2}+\frac{\lambda_N^{-\beta-1}}{\delta^2}) \\
		&-\int_{t_j}^{t_{j+1}}\E[\int_{ D_{j+1}\cup \mathcal E_{j+1}}((F_{\log,1}^{\delta})'(P^N u^{\delta}(t_{j}))-(F_{\log,1}^{\delta})'(u_{j}^N))Er_{j+1} dx]ds\\
		&-\int_{t_j}^{t_{j+1}}\E[\int_{D_{j+1}^c\cap \mathcal E_{j+1}^c}((F_{\log,1}^{\delta})'(P^N u^{\delta}(t_{j}))-(F_{\log,1}^{\delta})'(u_{j}^N))Er_{j+1} dx]ds\\
		\le& C(\kappa)\E [\|Er_{j+1}\|^2]\tau
		+C(\kappa,T,u(0))\tau |\ln(\delta)|^2(\frac {\epsilon^2}{\delta^2}+\frac {\lambda_N^{-\beta-1}}{\delta^2}) \\
		&+\Big|\int_{t_j}^{t_{j+1}}\E[\int_{\mathcal E_{j+1}}((F_{\log,1}^{\delta})'(P^N u^{\delta}(t_{j}))-(F_{\log,1}^{\delta})'(u_{j}^N))Er_{j+1} dx]ds\Big|,
	\end{align*}
	where $D_{j+1}=\{x\in \mathcal D| |u^{\delta}(t_{j+1})|\ge 1+\delta_0\},$ $\mathcal E_{j+1}=\{x\in \mathcal D| |u_{j+1}^N|\ge 1+\delta_0 \}.$
	Note that the term $C(\kappa,T,u(0))\tau |\ln(\delta)|^2\frac {\epsilon^2}{\delta^2}$
	disappears in Case 2 thanks to Lemma \ref{lm-tail}.

	Next we show a priori bound on $\mathcal E_{j+1}$.
	To this end, we construct a continuous interpolation of the considered numerical scheme via the flow of a local continuous problem defined on $[t_j,t_{j+1}],$
	\begin{align*}
		\widetilde u(t)=&\widetilde u(t_j)+P^NAS_{\tau}\widetilde u(t_j)(t-t_j)-P^NS_{\tau}(F_{\log,1}^{\delta})'(u_{j}^N)(t-t_j)\\
		&+2cS_{\tau}P^Nu_{j+1}^N(t-t_j)-\alpha S_{\tau}P^N(u_{j+1}^N-u_j^N)(t-t_j)\\
        &+P^NS_{\tau}\epsilon B(\widetilde u(t_j))(W(t)-W(t_j))
	\end{align*}
	with $\widetilde u(t_j)=u_j^N.$
Similar to the proof of Lemma \ref{lm-dis-h2}, one can obtain the regularity estimate of $\widetilde u$, i.e., for $\beta\in (0,1),$ $s\le t,$ $p\ge 2,$
	\begin{align}\label{int-dis-time}
		\|\widetilde u(t)-\widetilde u(s)\|_{L^p(\Omega; H)}&\le C(T,u(0),p)(|\ln(\delta)|+\epsilon)|\max(t-s,\tau)|^{\min(\frac {\beta+1}2,\frac 12)}, \\\label{int-dis-space}
		\|\widetilde u(t)\|_{L^p(\Omega;H^{1+\beta})}&\le C(T,u(0),p,\beta)(\epsilon+|\ln(\delta)|).
	\end{align}
	Denote $[s]$ the largest integer $j$ such that $j\tau \le s$, $s\in [0,T]$.
	Following the steps in the proof of Lemma \ref{lm-tail}, it follows that
	\begin{align*}
		&\frac 12 \|f^1(\widetilde u(t))\|^2\\
		\le&\tps  \frac 12 \|f^1(\widetilde u(0))\|^2
		-\int_0^t \<\mathbb I_{\widetilde u(s)>1} \nabla \widetilde u(s),\mathbb I_{\widetilde u(s)>1} \nabla S_{\tau}\widetilde u([s]\tau)\>ds
		\\
		&\tps-\int_0^t  \<\widetilde u(s)-1, \mathbb I_{\widetilde u(s)>1} P^NS_{\tau}(F_{\log,1}^{\delta})'(\widetilde u([s]\tau)\> ds\\
       &\tps+\int_0^t \epsilon  \<\mathbb I_{\widetilde u(s)>1}(\widetilde u(s)-1), P^N S_{\tau} B(\widetilde u[s]\tau)dW(s)\> +\frac {\epsilon^2}2 \int_0^t \mathbb \|\mathbb I_{\widetilde u(s)>1} B(\widetilde u([s]\tau))\|^2_{\mathcal L_2(H)} ds\\
		&\tps+\int_0^t \<\widetilde u(s)-1,\mathbb I_{\widetilde u(s)>1} P^N S_\tau(2c\widetilde u([s]\tau+\tau)-\alpha(\widetilde u([s]\tau+\tau)-\widetilde u([s]\tau))\>ds.
	\end{align*}
	Using the property of $(F_{\log,1}^{\delta})'$, Burkholder's inequality and \eqref{int-dis-time}-\eqref{int-dis-space}, we get
	\begin{align*}
		&\E [\sup_{t\in [0,T]} \|f^1(\widetilde u(t))\|^{2p}]\\
		\le& C(T,u(0),p) \Big(\|f^1(\widetilde u(0))\|^{2p}+|\ln(\delta)|^{2p}(\epsilon^{2p}+(\frac {{\tau^{\frac {\beta+1}2}}}{\delta}+\frac {\tau}{\delta}+\frac {\epsilon\tau^{\frac 12}}{\delta})^{2p}) \Big).
	\end{align*}
	Similar results hold for $\E [\sup_{t\in [0,T]} \|f^2(\widetilde u(t))\|^{2p}].$
	Thus,  we conclude that there exists a positive stochastic process $\eta_2(s)$ with any finite $p$th-moment such that
	\begin{align*}
		\mathcal E_{j+1}\le \eta_2(s) \Big(\frac{\|f^1(\widetilde u(0))\|^{2}+\|f^2(\widetilde u(0))\|^{2}}{\delta_0^2}+\frac {|\ln(\delta)|^2}{\delta_0^2}( {\epsilon^2}+\frac {\tau^{\beta+1}}{\delta^2}+\frac {\tau^2}{\delta^2}+\frac {\tau\epsilon^2}{\delta^2})\Big),\; a.s.
	\end{align*}
	Substituting the above inequality into the estimate of $III_2$, we have that
	\begin{align*}
		III_2\le& C(\kappa)\E [\|Er_{j+1}\|^2]\tau
		+C(\kappa,T,u(0))\tau |\ln(\delta)|^2(\frac {\epsilon^2}{\delta^2}+\frac {\lambda_N^{-\beta-1}}{\delta^2})+C(T,u(0),p)\\
		&\Big(\frac{\|f^1(\widetilde u(0))\|^{2}+\|f^2(\widetilde u(0))\|^{2}}{\delta_0^2}+\frac {|\ln(\delta)|^2}{\delta_0^2}( {\epsilon^2}+\frac {\tau^{\beta+1}}{\delta^2}+\frac {\tau^2}{\delta^2}+\frac {\tau\epsilon^2}{\delta^2})\Big).
	\end{align*}

	Thanks to the properties of the conditional expectation and stochastic integral,
	it follows that
    $$III_3= III_{31}+III_{32}+III_{33}+III_{34}+III_{35}, $$
    where
    \begin{align*}
		III_{31} :=&-\E\Big[\int_{t_j}^{t_{j+1}}\epsilon \<\int_{t_j}^{t_{j+1}}S(t_{j+1}-r)P^N(F_{\log}^{\delta})'(u^{\delta}(r)) dr,\\
        &(B(u^{\delta}(s))- B(u^{\delta}(t_j)))dW(s)\>\Big], \\
		III_{32}:=&-\E\Big[\int_{t_j}^{t_{j+1}}\epsilon \<\int_{t_j}^{t_{j+1}}S(t_{j+1}-r)P^N(F_{\log}^{\delta})'(u^{\delta}(r)) dr,\\
       &P^N(B(u^{\delta}(t_j))- B(u_j^N))dW(s)\>\Big], \\
		III_{33}:=&\E\Big[\int_{t_j}^{t_{j+1}}\epsilon \<S_{\tau}(\alpha u_{j+1}^N)\tau,P^N(B(u^{\delta}(s))- B(u_j^N))dW(s)\>\Big],\\
		III_{34}:=&\E\Big[\int_{t_j}^{t_{j+1}}\epsilon \<S_{\tau}(\alpha u_{j+1}^N)\tau,P^N(B(u^{\delta}(t_j))- B(u_j^N))dW(s)\>\Big],
	\end{align*}
and $III_{35}$ is defined as
\begin{align*}
		III_{35}:=&\E\Big[\int_{t_j}^{t_{j+1}} \epsilon^2 \<\int_{t_j}^{t_{j+1}}(S(t_{j+1}-r)B(u^{\delta}(r))-S_{\tau}B(u_{j}^N))dW(r),P^N(B(u^{\delta}(s))\\
        &- B(u_j^N))dW(s)\>\Big].
\end{align*}
	Similar to the estimates of $III_1$-$III_2$, by using \eqref{spe-err}  and Proposition \ref{prop-h1}, one can show that
	\begin{align*}
		&III_{31}+ III_{32}+III_{33}+III_{34}\\
  \le & C\epsilon \tau |\ln(\delta)|(\tau|\ln(\delta)|+\epsilon\tau)+ C\epsilon \lambda_N^{-1-\beta}|\ln(\delta)|\tau+
		C\|Er_j\|^2\tau.
	\end{align*}
	Next, we deal with $III_{35}.$
	Using the property that $\|(S_{\tau}-I)v\|\le C\|v\|_{ H^{\alpha}}\tau^{\frac \alpha 2}, \alpha \in [0,2]$, Burkholder's inequality  and  Proposition \ref{prop-h1}, we have that
	\begin{align*}
		&III_{35}\\
  &\le
		\epsilon^2\E [ \int_{t_j}^{t_{j+1}} \<\int_{t_j}^{t_{j+1}}B(u^{\delta}(r))-B(u^{\delta}(t_j)) dW(r),P^N(B(u^{\delta}(s))-B(u^{\delta}(t_j)))dW(s)\>]\\
		&+\epsilon^2\E [ \int_{t_j}^{t_{j+1}} \<\int_{t_j}^{t_{j+1}}B(u^{\delta}(r))-B(u^{\delta}(t_j)) dW(r),P^N(B(u^{\delta}(t_j))-B(u^N_j))dW(s)\>]\\
		&+ \epsilon^2\E [ \int_{t_j}^{t_{j+1}}\<\int_{t_j}^{t_{j+1}}B(u^{\delta}(t_j))-B(u_j^N) dW(r),P^N(B(u^{\delta}(s))-B(u^{\delta}(t_j)))dW(s)\>]\\
		&+ \epsilon^2\E [ \int_{t_j}^{t_{j+1}}\<\int_{t_j}^{t_{j+1}}B(u^{\delta}(t_j))-B(u_j^N) dW(r),P^N(B(u^{\delta}(t_j))-B(u_j^N))dW(s)\>]\\
		&+\epsilon^2\E [ \int_{t_j}^{t_{j+1}}\<\int_{t_j}^{t_{j+1}} ((S(t_{j+1}-r)-I)B(u^{\delta}(r))-(S_\tau-I)B(u_j^N)) dW(r),\\
		&\quad P^N(B(u^{\delta}(s))-B(u^{\delta}(t_j)))dW(s)\>]\\
		&+\epsilon^2 \E[\int_{t_j}^{t_{j+1}}\<\int_{t_j}^{t_{j+1}} ((S(t_{j+1}-r)-I)B(u^{\delta}(r))-(S_\tau-I)B(u_j^N)) dW(r),\\
		&\quad P^N(B(u^{\delta}(t_j))-B(u_j^N))dW(s)\>]\\
		&\le C\epsilon^2(\tau^2|\ln(\delta)|^2+\epsilon^2\tau^2
		+\lambda_N^{-1-\beta}|\ln(\delta)|\tau) +C\|Er_j\|^2\tau.
	\end{align*}
	Summarizing the above estimates,
	Gronwall's inequality yields the desired result.
\end{proof}

Thanks to Theorem \ref{main-con}, we are in a position to present the convergence of the energy functional as in Proposition \ref{prop-con}.

\begin{prop}\label{prop-con-ene}
	Let Setting \ref{set-1} hold with some $\bs=1+\beta>\frac d2$, $\beta\in [0,1)$. It holds that  in Case 1, for a small $\delta_0\sim O(\delta)$,
	\begin{align*}
		&\E[|H_{\delta}(u^{\delta}(t_j))-H_{\delta}(u_j^N)|]\\\nonumber
		&\le C\Big(\frac{\lambda_N^{-\frac {\beta+1}2}}{\delta}+\frac {\|f^1(u_0^N)\|+\|f^2(u_0^N)\|}{\delta_0}+\frac {|\ln(\delta)|}{\delta_0}( {\epsilon}+\frac {\tau^{\frac{\beta+1}2}}{\delta}+\frac {\tau}{\delta}+\frac {\tau^{\frac 12}\epsilon}{\delta}) \Big)^{\frac {(\bs-1)}{\bs}}|\ln(\delta)|,
	\end{align*}
	and that in Case 2,
	\begin{align*}
		&\E[|H_{\delta}(u^{\delta}(t_j))-H_{\delta}(u_j^N)|] \\\nonumber
		&\le C\Big(\frac{\lambda_N^{-\frac {\beta+1}2}}{\delta}+\frac {\|f^1(u_0^N)\|+\|f^2(u_0^N)\|}{\delta_0}+\frac {|\ln(\delta)|}{\delta_0}( \frac {\tau^{\frac{\beta+1}2}}{\delta}+\frac {\tau}{\delta})\Big)^{\frac {(\bs-1)}{\bs}}|\ln(\delta)|.
	\end{align*}
\end{prop}

\begin{proof}
	Following the steps as in the proof of Proposition \ref{prop-con}, we have that
	\begin{align*}
		\|\nabla(u^{\delta}(t_j)-u_j^N)\|_{L^2(\Omega;H)}&\le C\|(u^{\delta}(t_j)-u_j^N)\|_{L^2(\Omega;H)}^{\frac {\bs-1}{\bs}} |\ln(\delta)|^{\frac 1\bs},
	\end{align*}
	and that
	\begin{align*}
		&\E [\|F_{\log}^{\delta}(u^{\delta}(t_j))-F_{\log}(u_j^N)\|_{L^1}]\\
		&\le C(1+|\ln(\delta)| \sqrt{\E [\|u(t)-u^{\delta}(t)\|^2]}).
	\end{align*}
	Combining the above estimates with Theorem \ref{main-con} and \eqref{spe-err}, we complete the proof.
\end{proof}

The above result also implies that the proposed scheme is stable in the energy space.
Thanks to Theorem \ref{main-con} and Theorem \ref{tm-con1}, one can obtain the overall strong error estimate.

\begin{cor}\label{main-cor}
	Let Setting \ref{set-1} hold with some $\bs=1+\beta>\frac d2, \beta \in [0,1)$.
	It holds that for $\delta_0\sim O(\delta), $
	\begin{align*}
		&\sup_{j\le J}\|u_{det}(t_{j})-u_j^N\|_{L^{2}(\Omega;H)}^2\\
		&\le
		C|\ln(\delta)|^2\Big(\frac{\lambda_N^{-{\beta+1}}}{\delta^2}+\frac {\|f^1(u_0^N)\|^{2}+\|f^2(u_0^N)\|^{2}}{\delta_0^2}+\frac {\epsilon^2}{\delta_0^2}+\frac {\tau^{\beta+1}}{\delta_0^{2}\delta^2}+\frac {\tau^2}{\delta_0^2\delta^2}+\frac {\tau\epsilon^2}{\delta_0^2\delta^2}\Big)\\
		&+C(\delta^2+\epsilon^2)
	\end{align*}
	in Case 1, and
	\begin{align*}
		&\sup_{j\le J}\|u(t_{j})-u_j^N\|_{L^{2}(\Omega;H)}^2\\
		&\le
		C|\ln(\delta)|^2\Big(\frac{\lambda_N^{-{\beta+1}}}{\delta^2}+\frac {\|f^1( u_0^N)\|^{2}+\|f^2(u_0^N)\|^{2}}{\delta_0^2}+\frac {\tau^{\beta+1}}{\delta_0^{2}\delta^2}+\frac {\tau^2}{\delta_0^2\delta^2}\Big)+C\delta^2
	\end{align*}
	in Case 2.
\end{cor}

\begin{rk}
	It should be noticed that the derived convergence rate may be not optimal since it relies on $\frac 1{\delta}$.
	The upper bound of $\|f^1(u_0^N)\|$ and $\|f^2(u_0^N)\|$ could be estimated as follows. Let us illustrate the estimate of  $\|f^1(u_0^N)\|$ since that of $\|f^2(u_0^N)\|$ is similar.
	By H\"older's inequality, Chebyshev's inequality, the Sobolev embedding theorem $L^4(\mathcal D)\longrightarrow H^1$ for $d\le 3$ and \eqref{spe-err}, it holds that
	\begin{align*}
		\|f^1(u_0^N)\|&\le |\{x\in \mathcal D| |u_0^N|>1\}|^{\frac 12}\|1-u_0^N\|_{L^4(\mathcal D)}\\
		&\le C(u(0))|\{x\in \mathcal D| |u_0^N-u(0)|>c_0\}|^{\frac 12}\\
		&\le C(u(0))\frac {\|u_0^N-u(0)\|}{c_0} \le  C(u(0),c_0)\lambda_N^{-\frac {\bs}2},
	\end{align*}
	where $c_0=1-|u(0)|_E>0$.
	In particular, if $\bs>\frac d2$, one can obtain an improved bound,
	\begin{align*}
		\|f^1(u_0^N)\|&\le |\{x\in \mathcal D| |u_0^N|>1\}|\|1-u_0^N\|_{E}\\
		&\le C(u(0),c_0)\lambda_N^{-\bs}.
	\end{align*}
	One can slightly modify the condition of $u(0)$ and consider the random initial data.
\end{rk}

To end this part, we show that the proposed scheme nearly preserve the Stampacchhia maximum principle via a strong convergence result in $L^p(\Omega; E).$
Thanks to the Gagliardo--Nirenberg inequality, one can see that the following asymptotic maximum principle holds.
Due to Lemma \ref{lm-dis-h2} and Proposition \ref{prop-con-ene}, from the Gagliardo--Nirenberg inequality, it holds that
\begin{align*}
	&\|u^{\delta}(t_{j+1})-u_{j+1}^N\|_{L^p(\Omega;E)}\\
	&\le C \|u^{\delta}(t_{j+1})-u_{j+1}^N\|_{L^p(\Omega;H)}^{\theta} (\|u^{\delta}(t_{j+1})\|_{L^p(\Omega;H^{ {\beta+1}})}^{1-\theta}+\|u_{j+1}^N\|_{L^p(\Omega;H^{ {\beta+1}})}^{1-\theta}),
\end{align*}
where $\theta=\frac {d}{2(\beta+1)}\in (0,1).$
This implies that $d<2\beta+2$ with $\beta\in (0,1).$
Applying Theorem \ref{main-con}, we get
\begin{align*}
	\|u^{\delta}(t_{j+1})-u_{j+1}^N\|_{L^p(\Omega;E)}\to 0
\end{align*}
as $|\ln(\delta)|(\frac{\lambda_N^{-\frac {\beta+1}2}}{\delta}+\frac {\|f^1(u_0^N)\|+\|f^2(u_0^N)\|}{\delta_0}+\frac {\tau^{\frac {\beta+1} 2}}{\delta_0\delta}+\frac {\tau}{\delta_0\delta}+\frac {\tau^{\frac 12}}{\delta_0\delta})\sim o(1)$.
Up to a subsequence, we  conclude that  in Case 1,
\begin{align*}
	\lim_{(\epsilon, \delta,\tau,N)\to (0,0,0,+\infty)}\mathbb P(|u_{j+1}^N|<1)=1,
\end{align*}
and in Case 2,
\begin{align*}
	\lim_{(\delta,\tau,N)\to (0,0,+\infty)}\mathbb P(|u_{j+1}^N|<1)=1.
\end{align*}

\section{Numerical experiments}
\label{sec-4}
In this section, we present several numerical experiments to validate our theoretical findings for   \eqref{log-sac} and \eqref{log-reg-sac}.
For simplicity, in what follows, the computational domains $\mathcal{D}$ are set to be $(0,2\pi)\times(0,2\pi)$  for the periodic boundary condition and  $(-1,1)\times(-1,1)$ for the homogeneous Neumann boundary condition, respectively.
The operator  $B(u)$ satisfies
\begin{align*}
	B(u)e_{l,k}=B(u)\sqrt{q_{k,l}}e_{l,k},
\end{align*}
where $q_{k,l}=\frac 1{1+k^2+l^2}$, $e_{l,k}=\exp\big(i(kx+ly)\big)$ for the periodic boundary condition
and $e_{l,k}=\phi_{k}(x)\phi_{l}(y)$
for  the homogeneous Neumann boundary condition. Here  $\phi_{k}(x)$ is defined by
$$\phi_{k}(x):=L_{k}(x)-\frac{k(k +1)}{(k +2)(k +3)}L_{k+2}(x)$$
with $L_{k}(x)$ being the Legendre polynomial with degree $k$, satisfying the homogeneous Neumann boundary condition. For the evaluation of the noise term, we have used the Algorithm 10.6 reported in \cite{LPS14}.
Moreover, we always set $c=3/2$ and choose the stabilizing parameter as $\alpha=2$ for the stabilized scheme \eqref{sta-semi} unless specified otherwise.

\begin{figure*}[!t]
	\centerline{\includegraphics[scale=0.35]{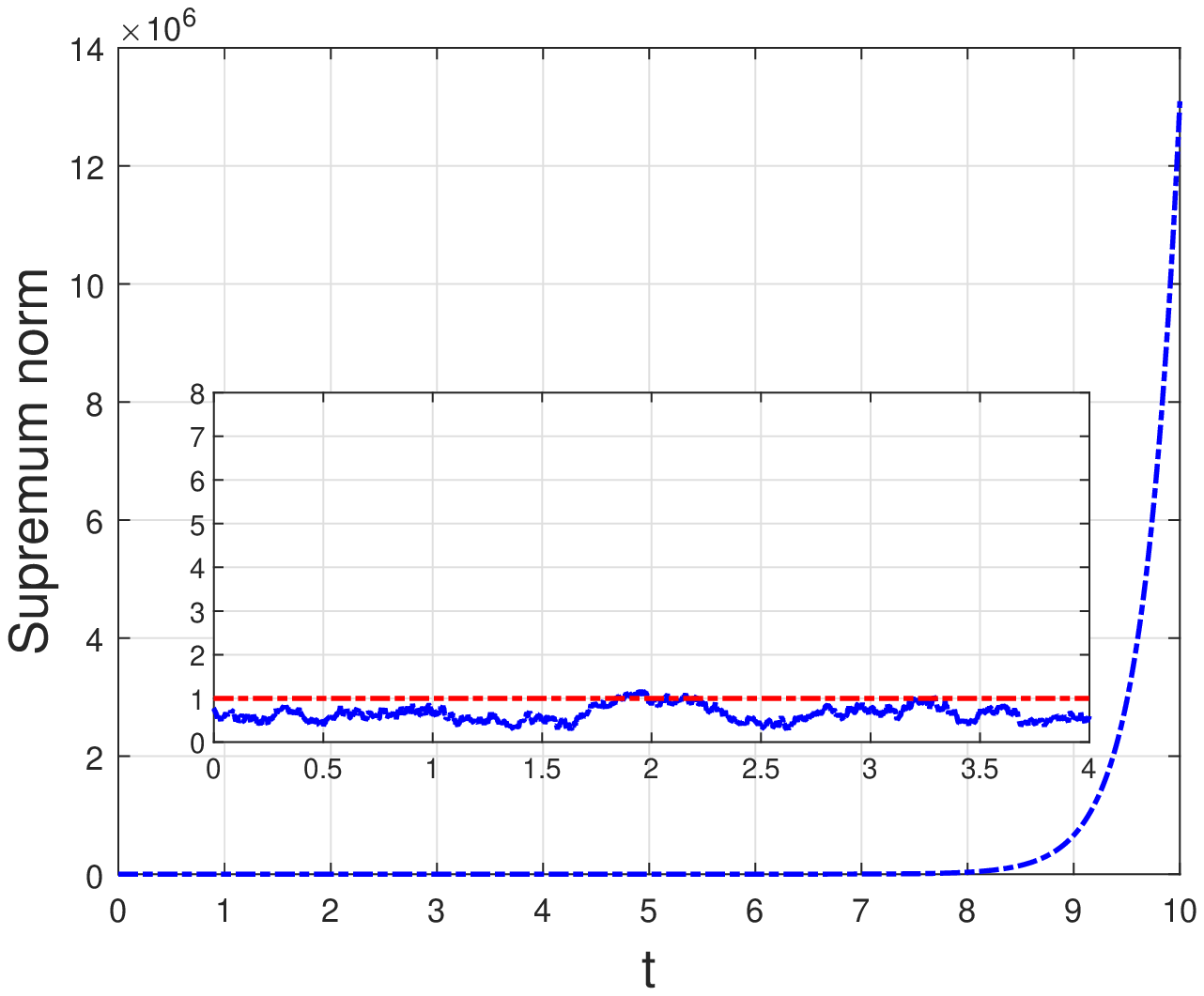}\includegraphics[scale=0.35]{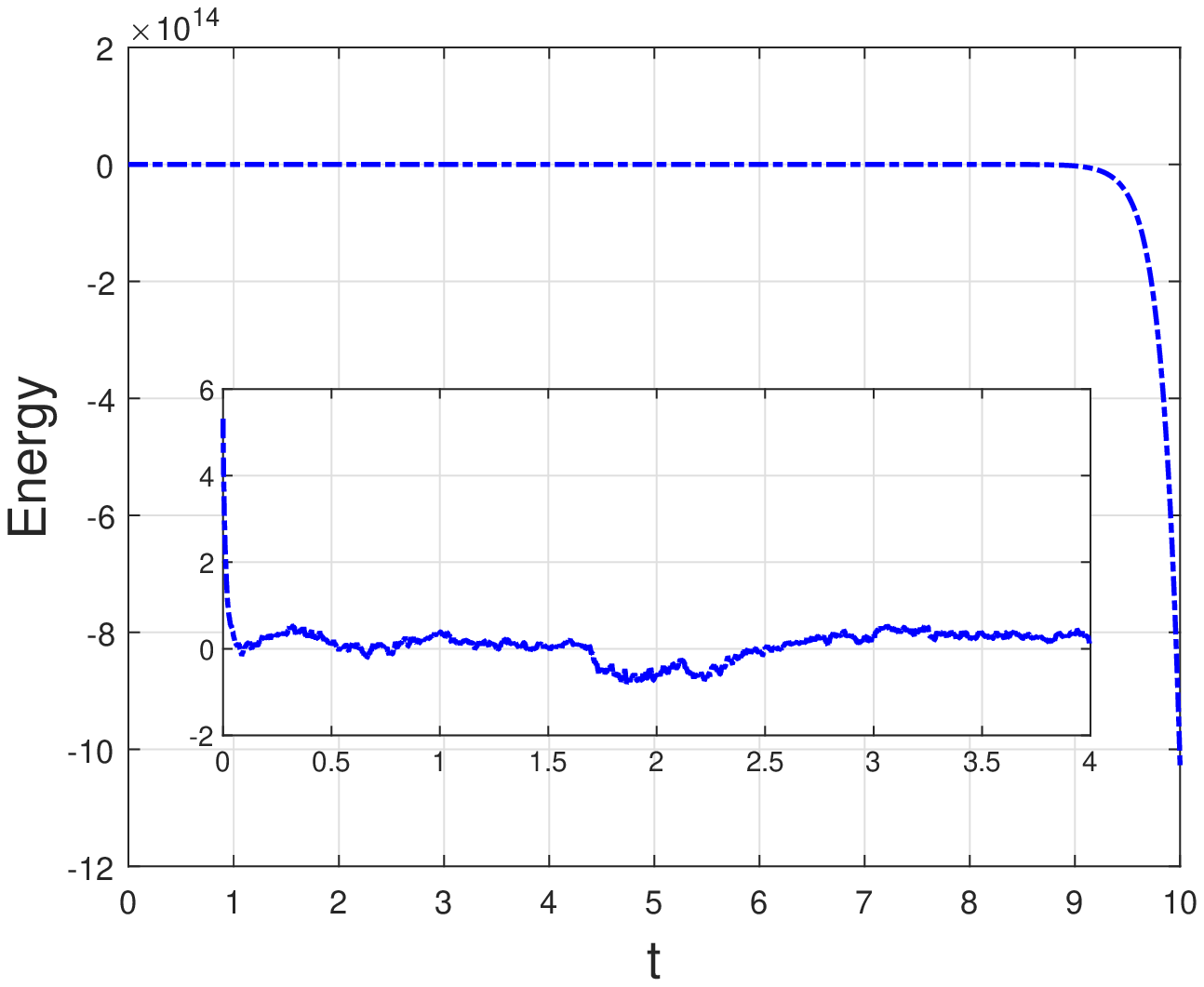}}
	\centerline{\includegraphics[scale=0.35]{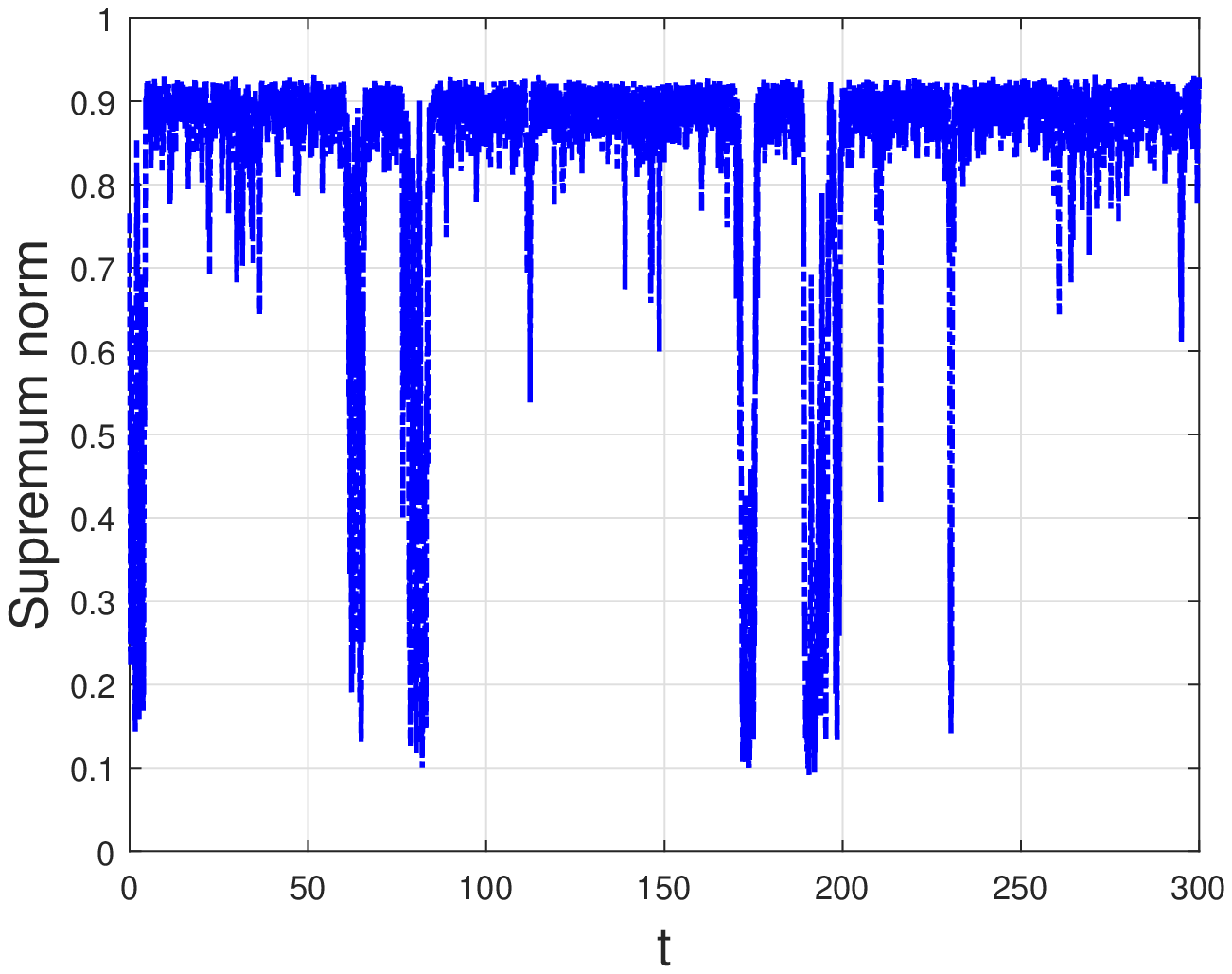}\includegraphics[scale=0.35]{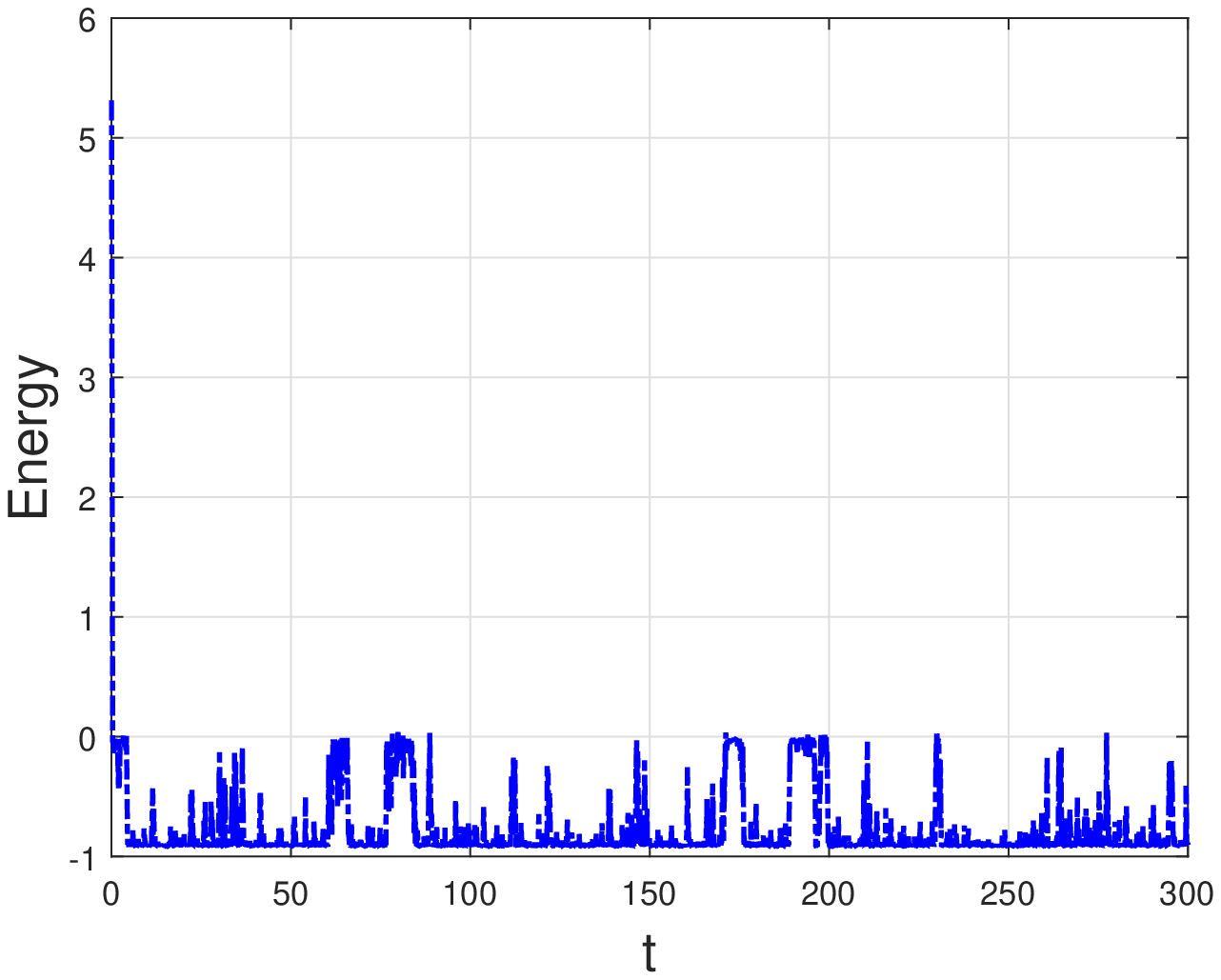}}
	\caption{The evolutions in time of the supremum norm (left) and the energy (right)  of the simulated solution with $B(u)=1$ (top) and $B(u)=\sin^{2}(\pi u)$ (bottom).
	}\label{fig1}
\end{figure*}
\subsection{Comparison for two type of noises:}
In this test, two type of noises $B(u)=1$ (in Case 1) and $B(u)=\sin^{2}(\pi u)$ (in Case 2) are considered for the model \eqref{log-reg-sac} with $\sigma=1,\delta=1e-18, \epsilon=1$, and the initial data $u(0)=0.4[\cos(\pi x)\sin(\pi y)+\sin(\pi x)\cos(2\pi y)]+0.2.$ The operator $A$ is subject to the homogeneous Neumann boundary condition. The simulation is performed by the stabilized scheme \eqref{sta-semi} with time step size $\tau=1e-3$, in which the Legendre spectral method with $128\times128$ modes is used for the spatial discretization. In Figure \ref{fig1}, we plot the evolutions in time of the supremum norm and the energy of the numerical solutions for both two types of noises up to $T=300.$ It is observed that the numerical solution and the corresponding energy will blow up at a finite time (about $t=10$) for the case of $B(u)=1$, which indicates the ill-posedness of the model with $\epsilon=1$ and $B(u)=1$. As shown in the last line of Figure \ref{fig1}, the regularized model  \eqref{log-reg-sac} with $B(u)=\sin^{2}(\pi u)$ is a more reasonable model, in which the Stampacchia maximum bound and the energy evolution law are almost preserved. Therefore, in what follows, we will focus our attention on Case 2 to test the performance of  proposed numerical schemes. For simplicity, we always choose $B(u)=\sin^{2}(\pi u)$ in the rest of this section.

\subsection{Test of the convergence accuracy}
For simplicity, we only numerically study the convergence order of  \eqref{sta-semi} for  \eqref{log-reg-sac} under the homogeneous Neumann boundary condition. Both the diffuse interface width $\sigma$ and the noise density $\epsilon$ are set to be $1$, and the initial data considered in this test is given by
$$u(0)=0.4[\cos(\pi x)\sin(\pi y)+\sin(\pi x)\cos(2\pi y)]+0.2.$$
The Legendre spectral method with $128\times128$ modes is used for the spatial discretization.
Since there is no exact solution available, we will use the numerical solution by
stabilized semi-implicit scheme \eqref{sta-semi} with a small enough time step size $\tau=1e-6$ as a reference solution.
The averaged $L^{2}$ errors over 1000 realizations as functions of the time step sizes are plotted in log-log scale  for several different tested $\delta$ in Figure \ref{fig3}. It is observed that the stabilized semi-implicit scheme \eqref{sta-semi} achieves half order  in time with several different tested $\delta=0.1, 1e-4$ and $1e-18$.

\begin{figure*}[t!]
	\begin{minipage}[t]{0.32\linewidth}
		\centerline{\includegraphics[scale=0.35]{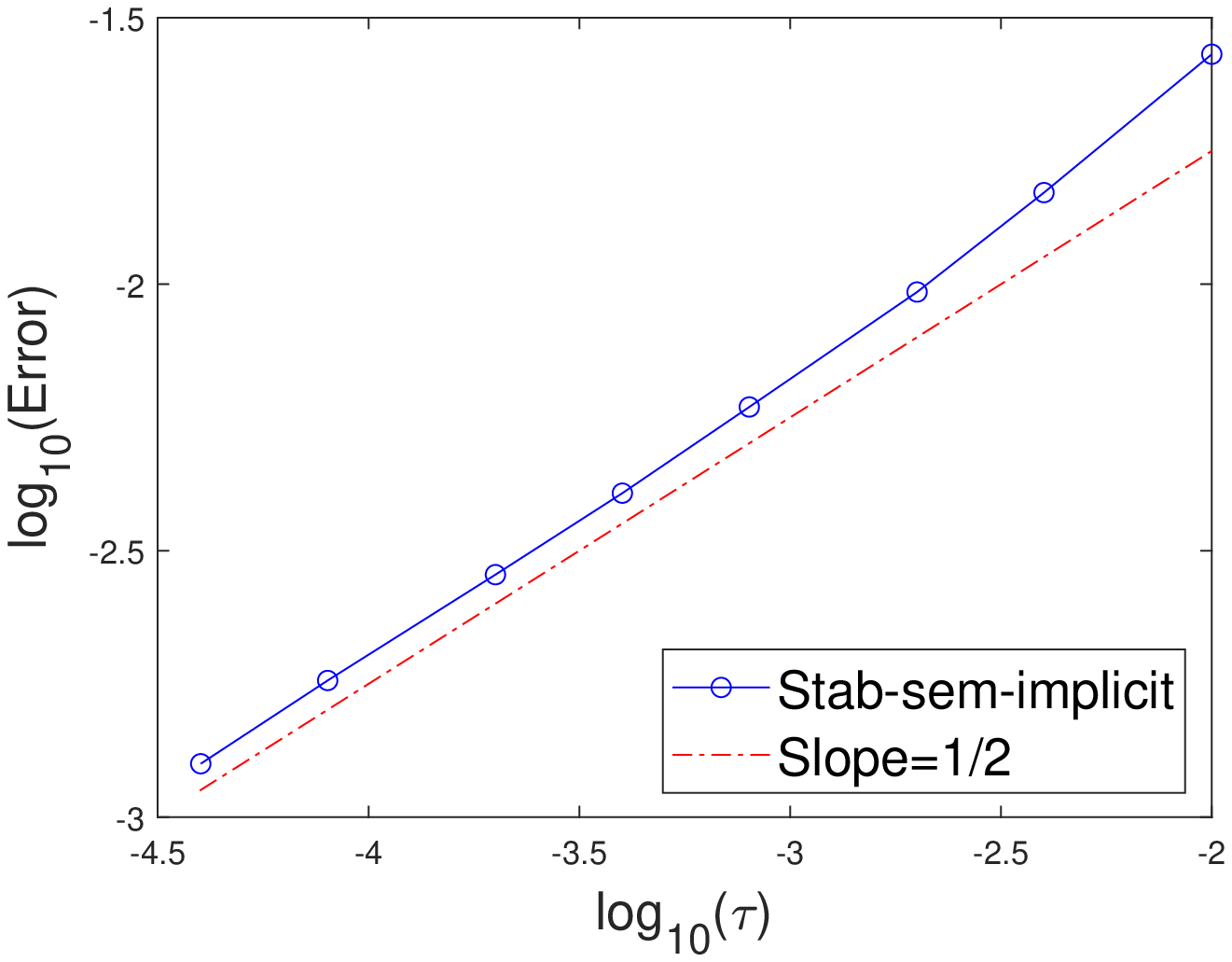}}
		\centerline{(a) $\delta=0.1$}
	\end{minipage}
	\begin{minipage}[t]{0.32\linewidth}
		\centerline{\includegraphics[scale=0.35]{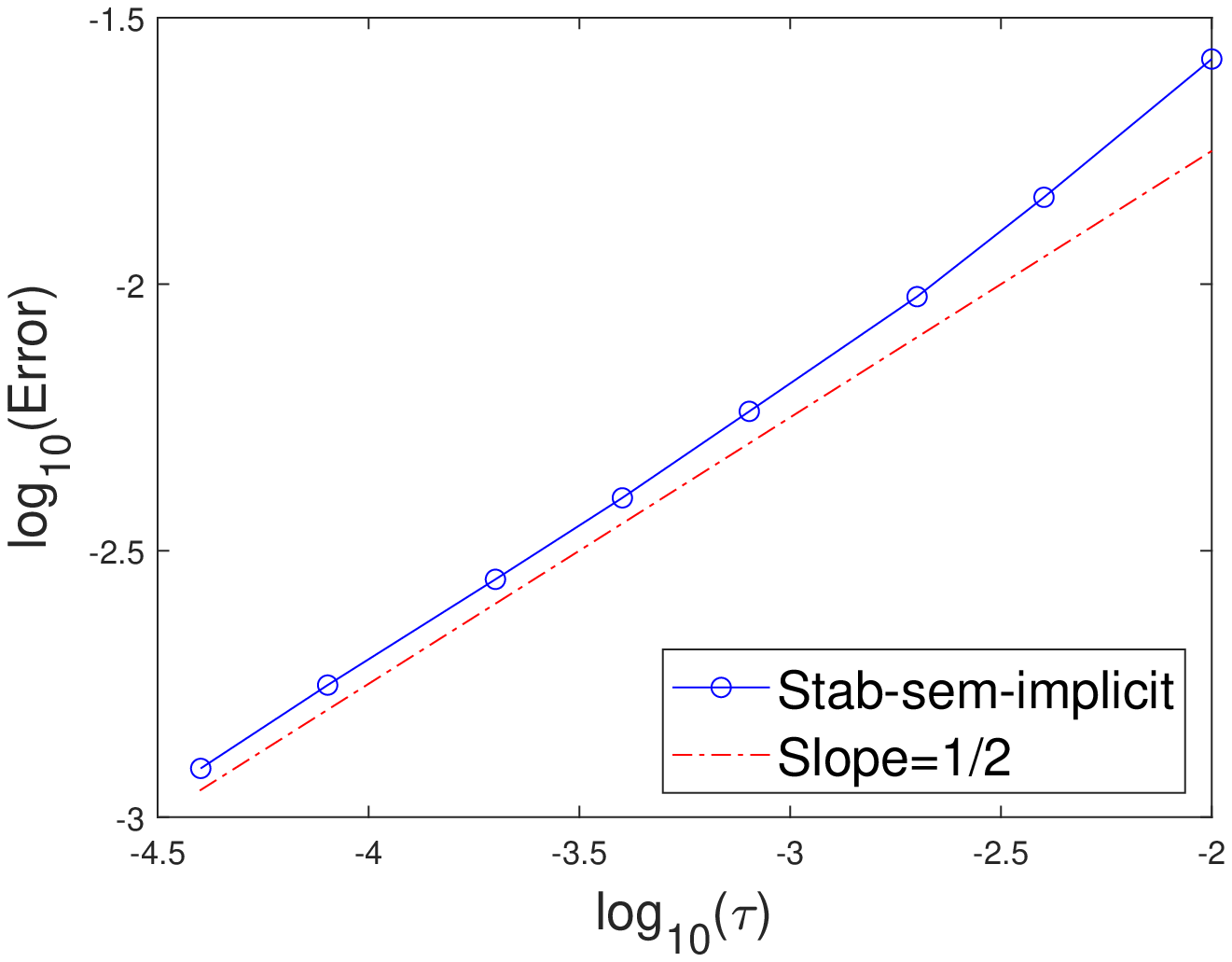}}
		\centerline{(b) $\delta=1e-4$}
	\end{minipage}
	\begin{minipage}[t]{0.32\linewidth}
		\centerline{\includegraphics[scale=0.35]{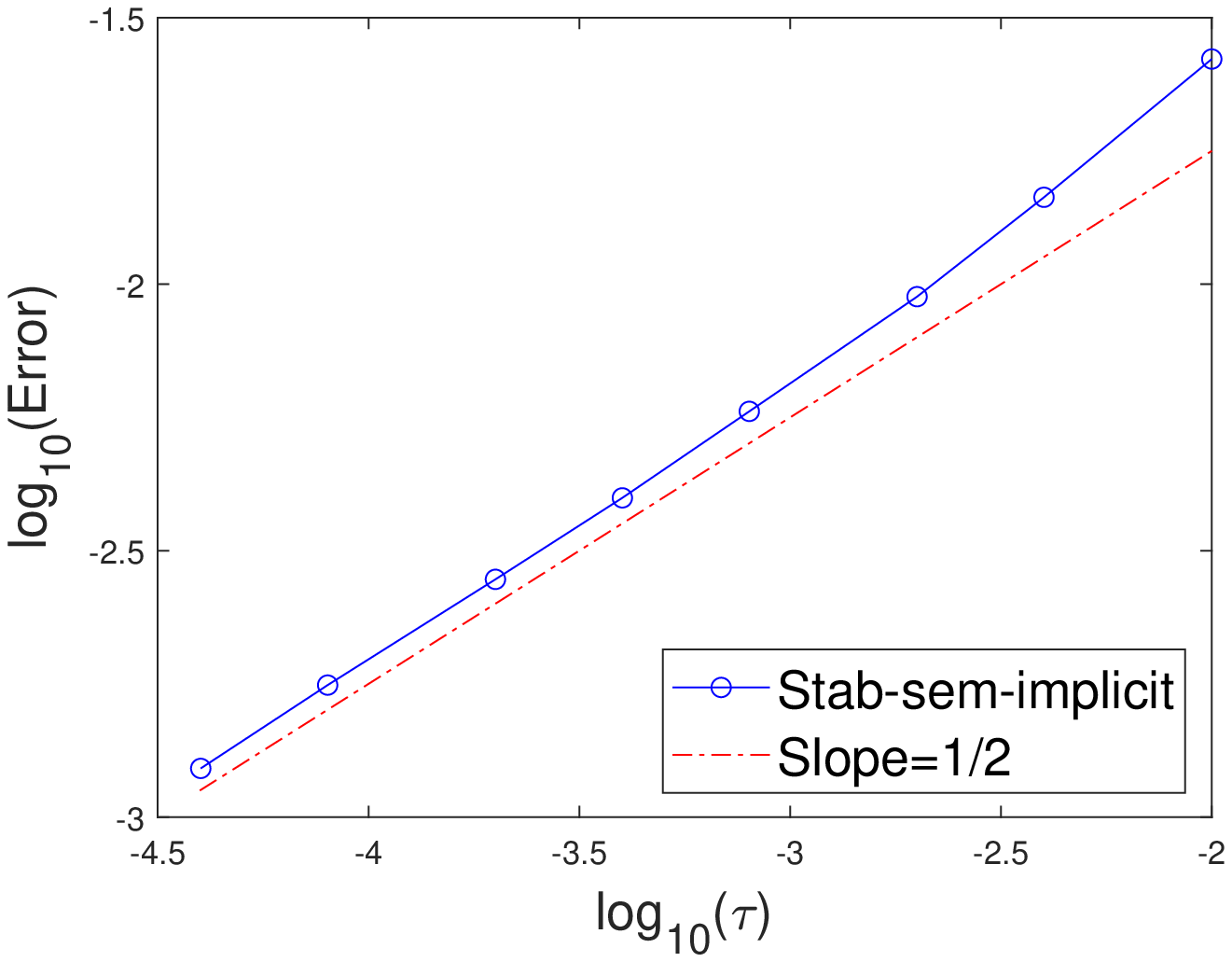}}
		\centerline{(c) $\delta=1e-18$}
	\end{minipage}
	\caption{The error behavior with respect to the time step size for the stabilized semi-implicit scheme $B(u)=\sin^{2}(\pi u)$ at $T=0.1$.
	}\label{fig3}
\end{figure*}

\subsection{Test on the effect of $\delta$ on the regularized model}
In this test, we investigate the influence of the regularized parameter $\delta$ on the model \eqref{log-reg-sac}.  The considered problem is subject to the homogeneous Neumann boundary condition with following initial condition:
$$u(0)=0.5\cos(\pi x)\cos(\pi y)+0.3.$$
The simulation up to $T=300$ is performed by the stabilized semi-implicit scheme \eqref{sta-semi} with $\tau=0.01$ and $128\times 128$ Legendre-type basis modes.
In Figure \ref{fig4}, the evolutions of the computed averaged energy over 500 realizations are plotted for the model with several different regularized parameters, that is $\delta= 1e-2, 1e-4, 1e-6, 1e-8, 1e-12,$ and $1e-18$. It is observed that the computed average energies converge to a limited one as $\delta\rightarrow0.$

\begin{figure*}[htbp]
	\begin{minipage}[t]{0.49\linewidth}
		\centerline{\includegraphics[scale=0.45]{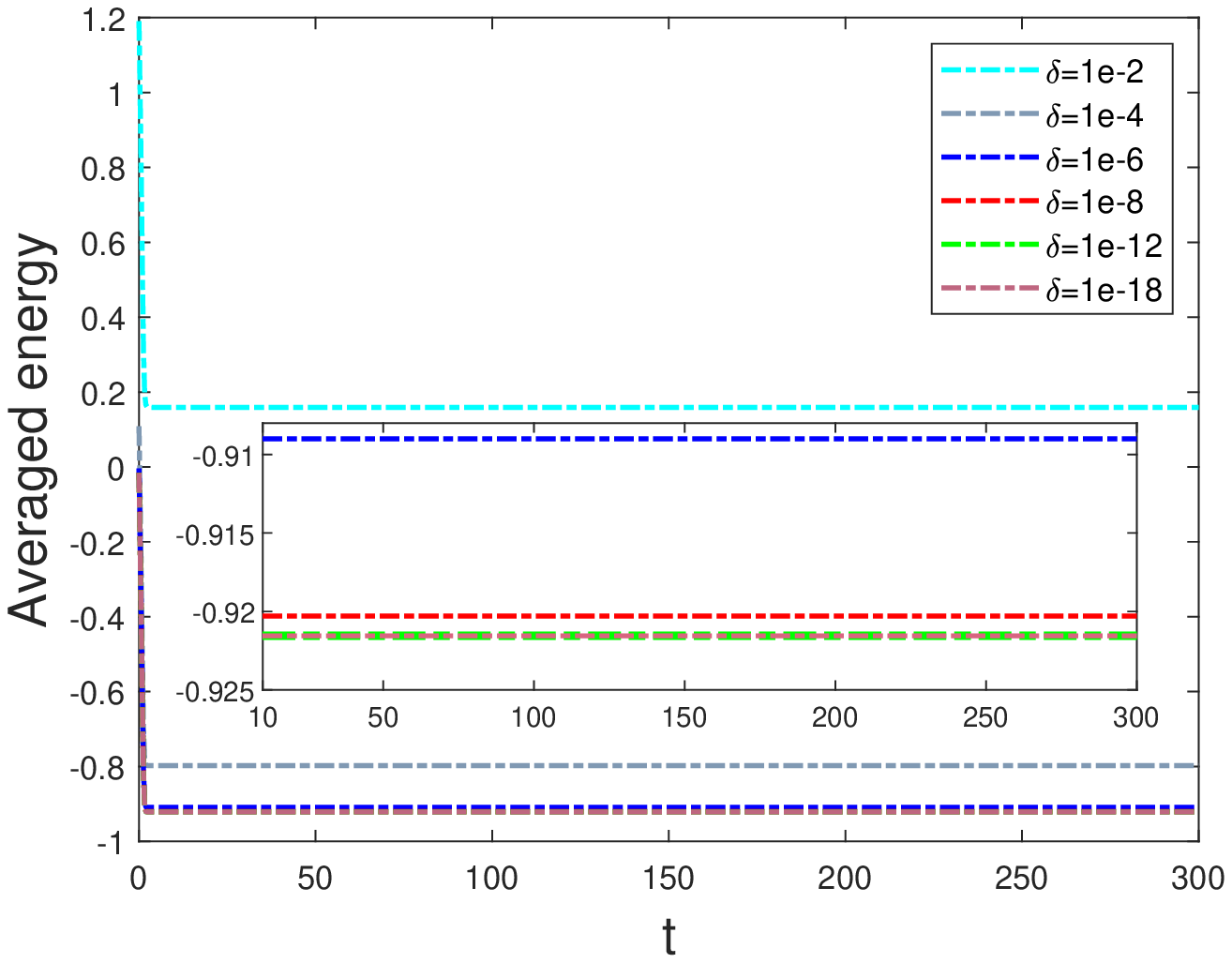}}
		\centerline{(a) energy evolution in time }
	\end{minipage}
	\begin{minipage}[t]{0.49\linewidth}
		\centerline{\includegraphics[scale=0.45]{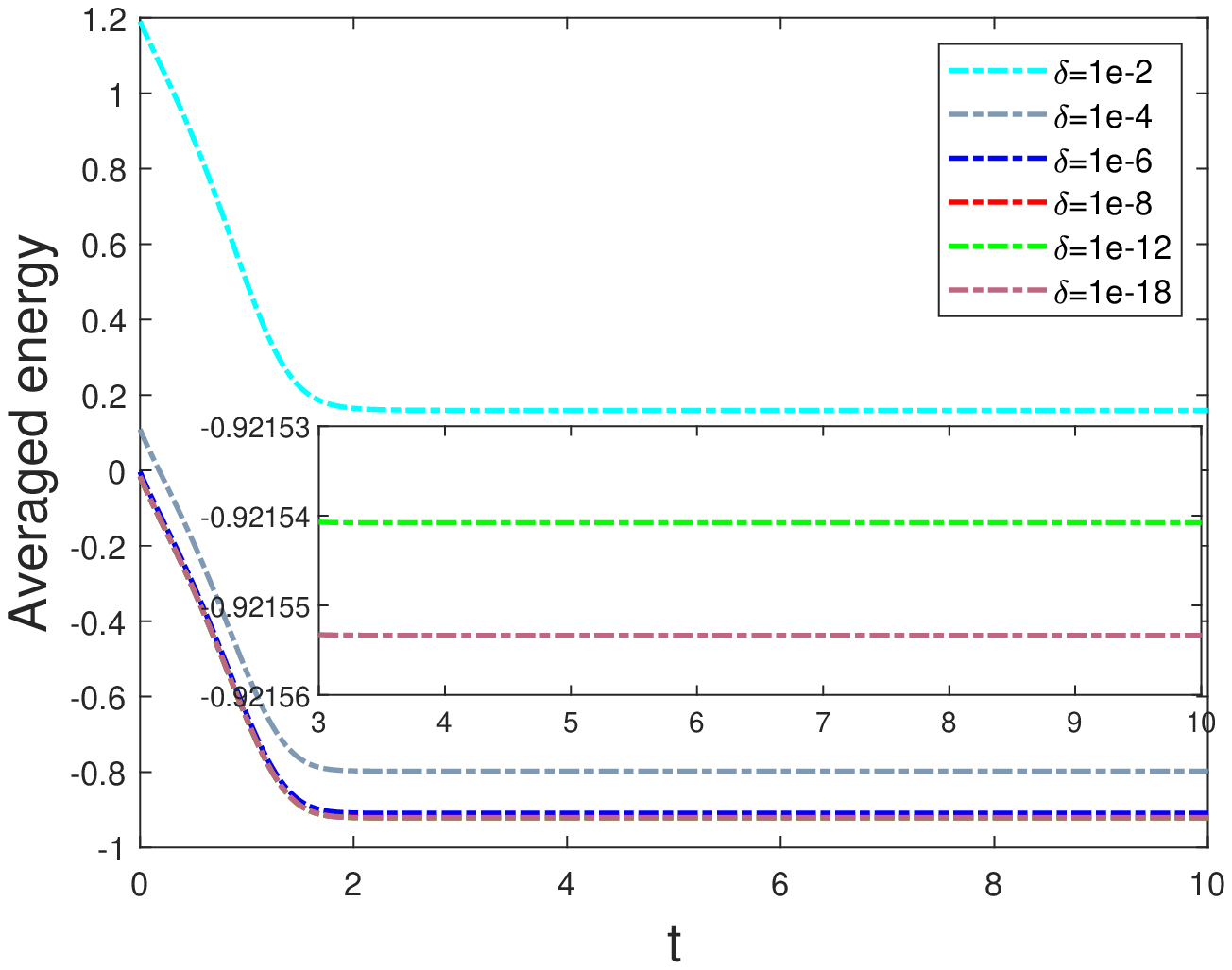}}
		\centerline{(b) zoom in at $[0,10]$}
	\end{minipage}
	\caption{The averaged energy evolution in time of \eqref{log-reg-sac} with $\sigma=0.1$, $\epsilon=1e-4$ and several different values of $\delta= 1e-2, 1e-4, 1e-6, 1e-8, 1e-12, 1e-18$.
	}\label{fig4}
\end{figure*}

\begin{figure*}[!b]
	\begin{minipage}[t]{0.32\linewidth}
		\centerline{\includegraphics[scale=0.35]{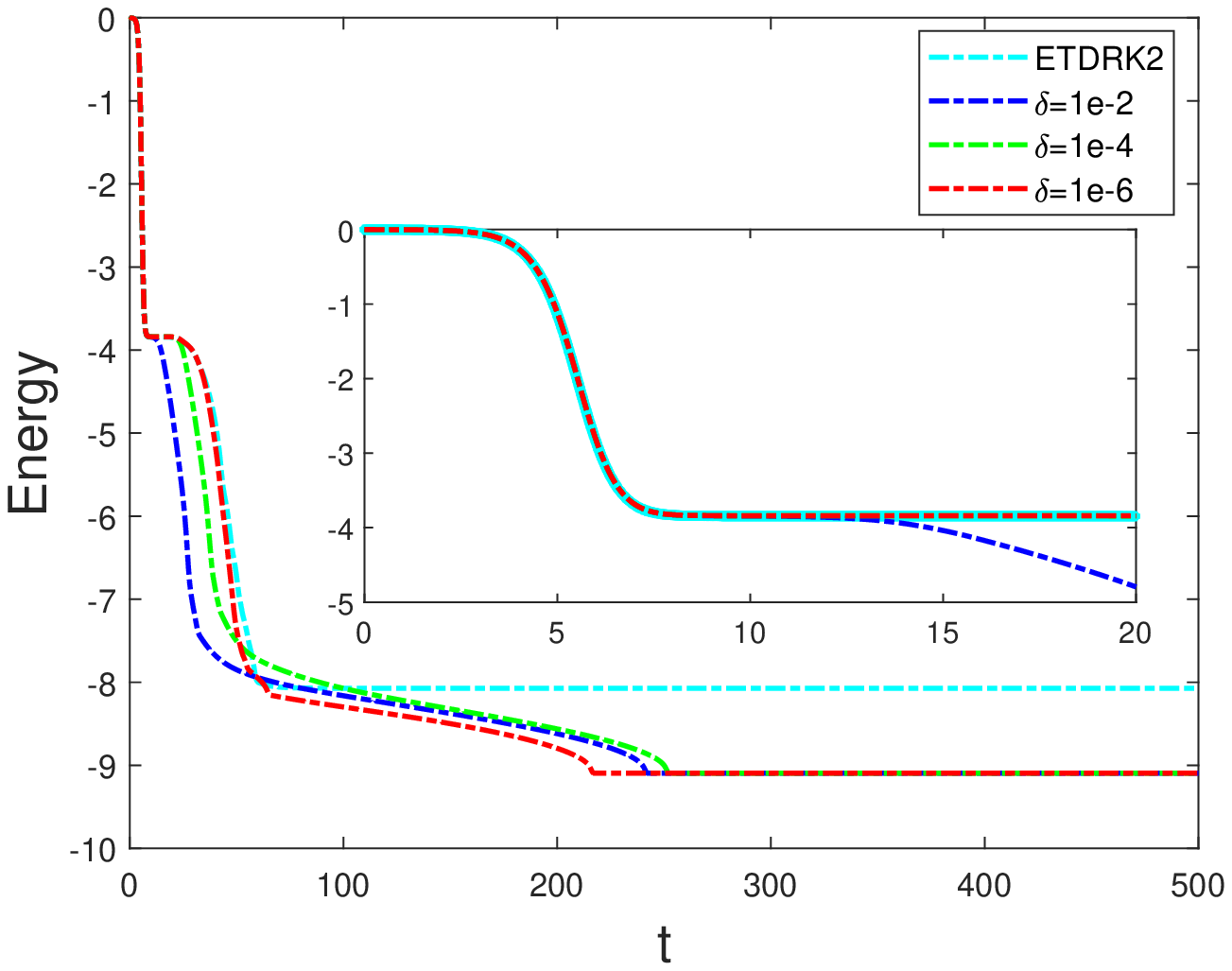}}
		\centerline{ (a) }
	\end{minipage}
	\begin{minipage}[t]{0.32\linewidth}
		\centerline{\includegraphics[scale=0.35]{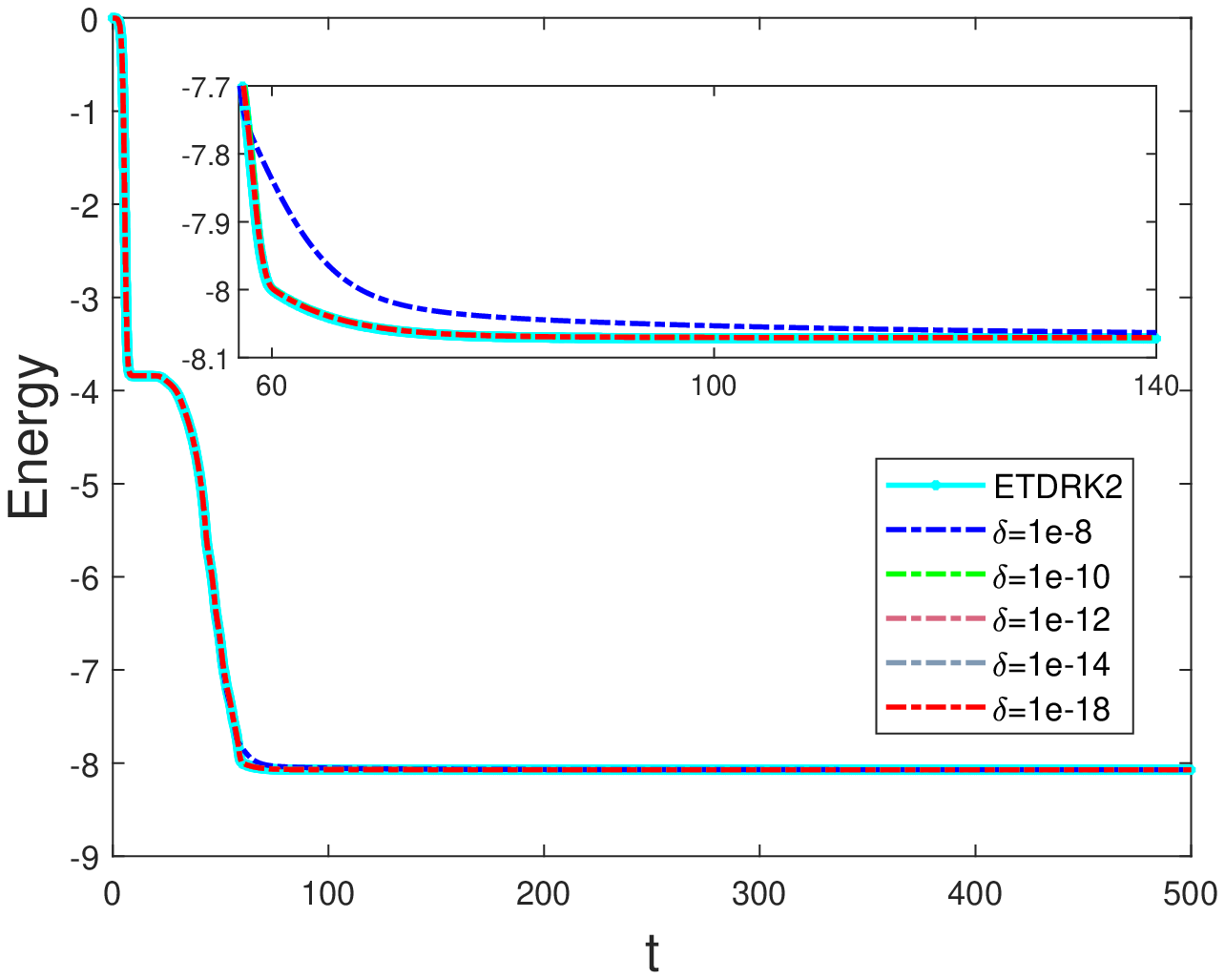}}
		\centerline{ (b)}
	\end{minipage}
	\begin{minipage}[t]{0.32\linewidth}
		\centerline{\includegraphics[scale=0.35]{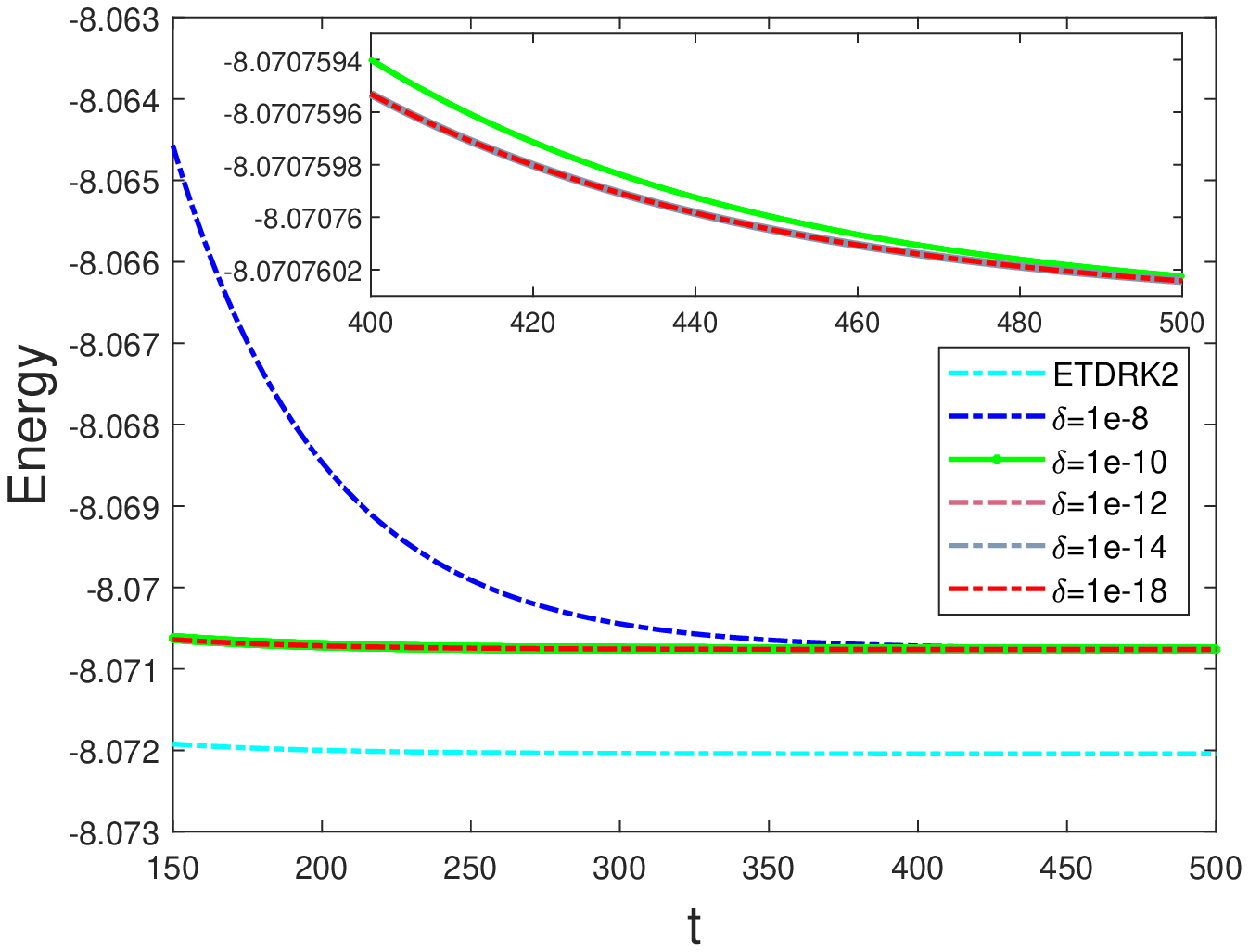}}
		\centerline{ (c)}
	\end{minipage}
	\caption{The energy evolution in time of \eqref{log-reg-sac} with $\delta=1e-18$, $\sigma=1e-2$ and several different values of $\epsilon= 1e-n$ with $n=2,4,6,8,10,12,14,16,18$.
	}\label{fig5}
\end{figure*}

\subsection{Coarsening dynamics}

The coarsening dynamics of the regularized model \eqref{log-reg-sac} and the deterministic model are numerically investigated in this part. The related parameters in \eqref{log-reg-sac} are set to be $\sigma=0.01, \delta=1e-18$, and the periodic boundary condition is considered with the initial data $u(0)=0.01\cos(\pi x)\sin(\pi y)$.

\begin{figure*}[htbp]
	\centerline{\includegraphics[scale=0.18]{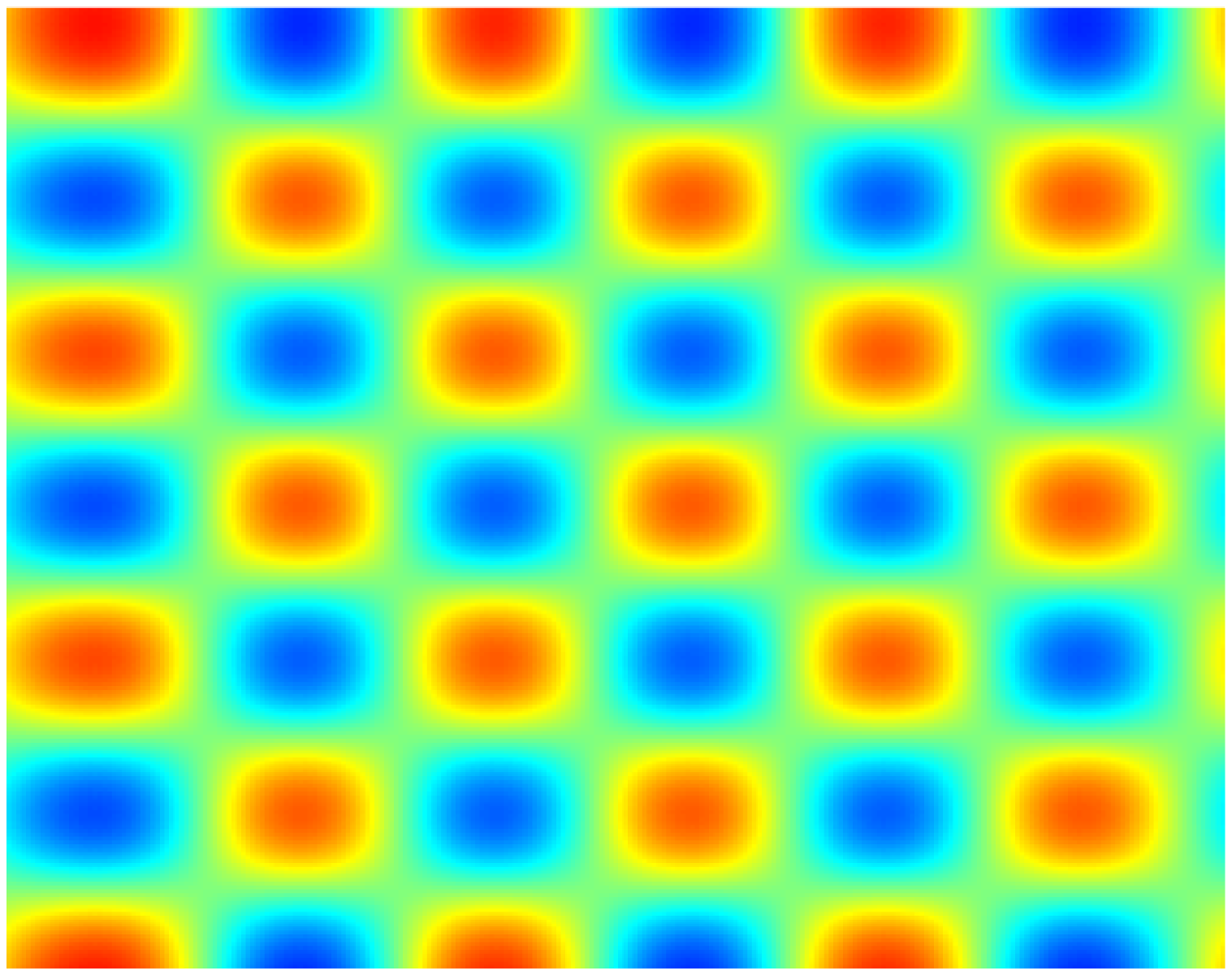}\includegraphics[scale=0.18]{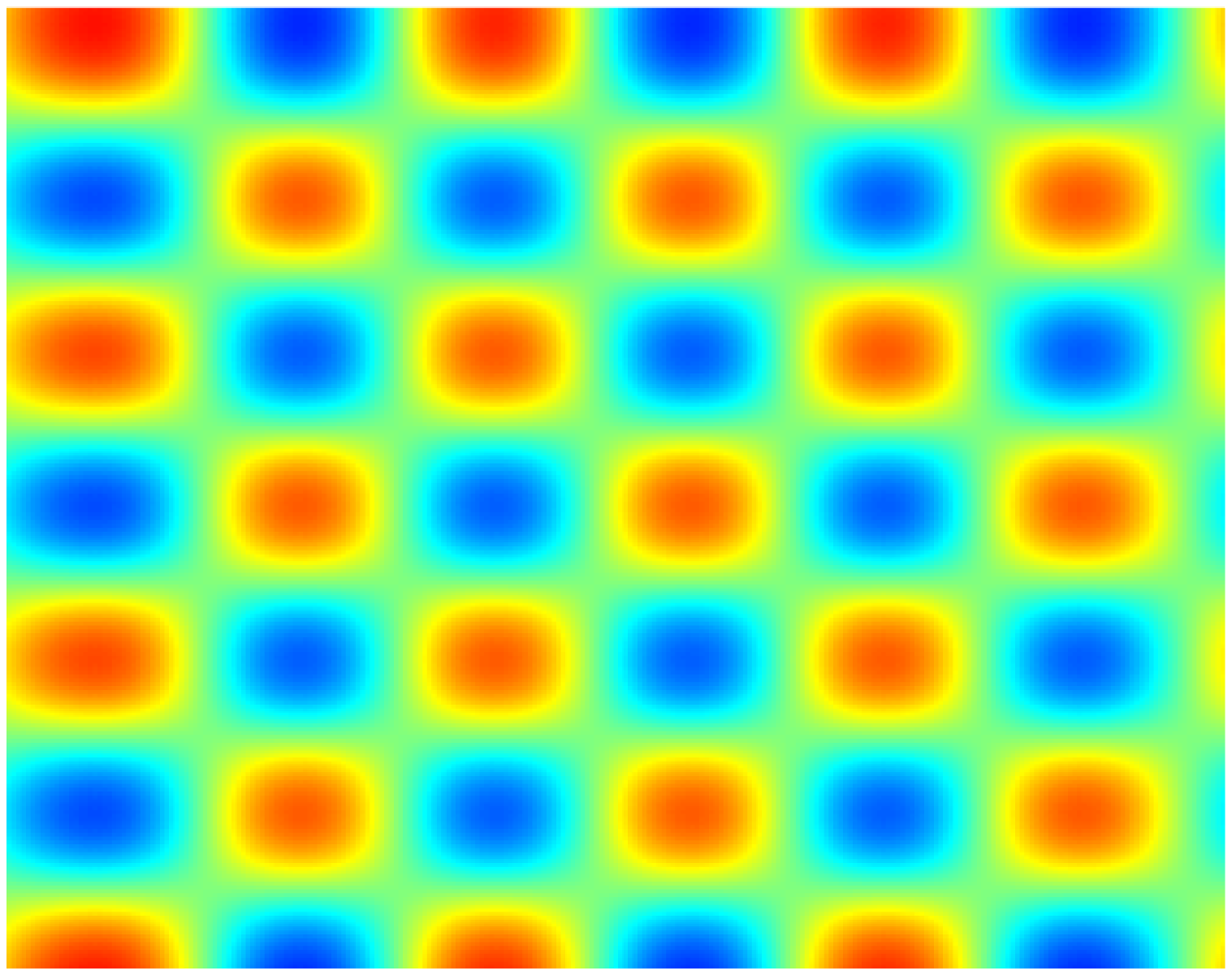}\includegraphics[scale=0.18]{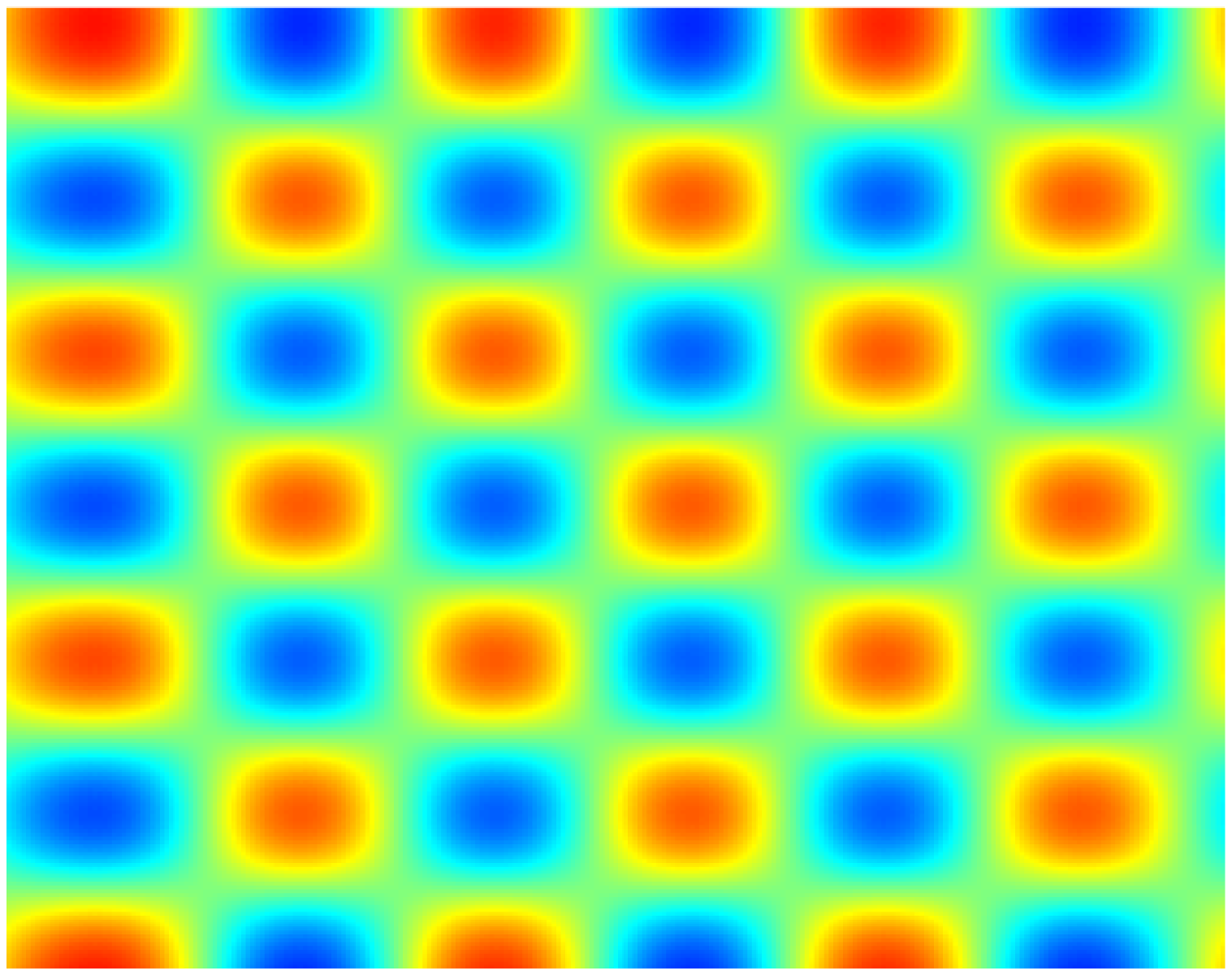}\includegraphics[scale=0.18]{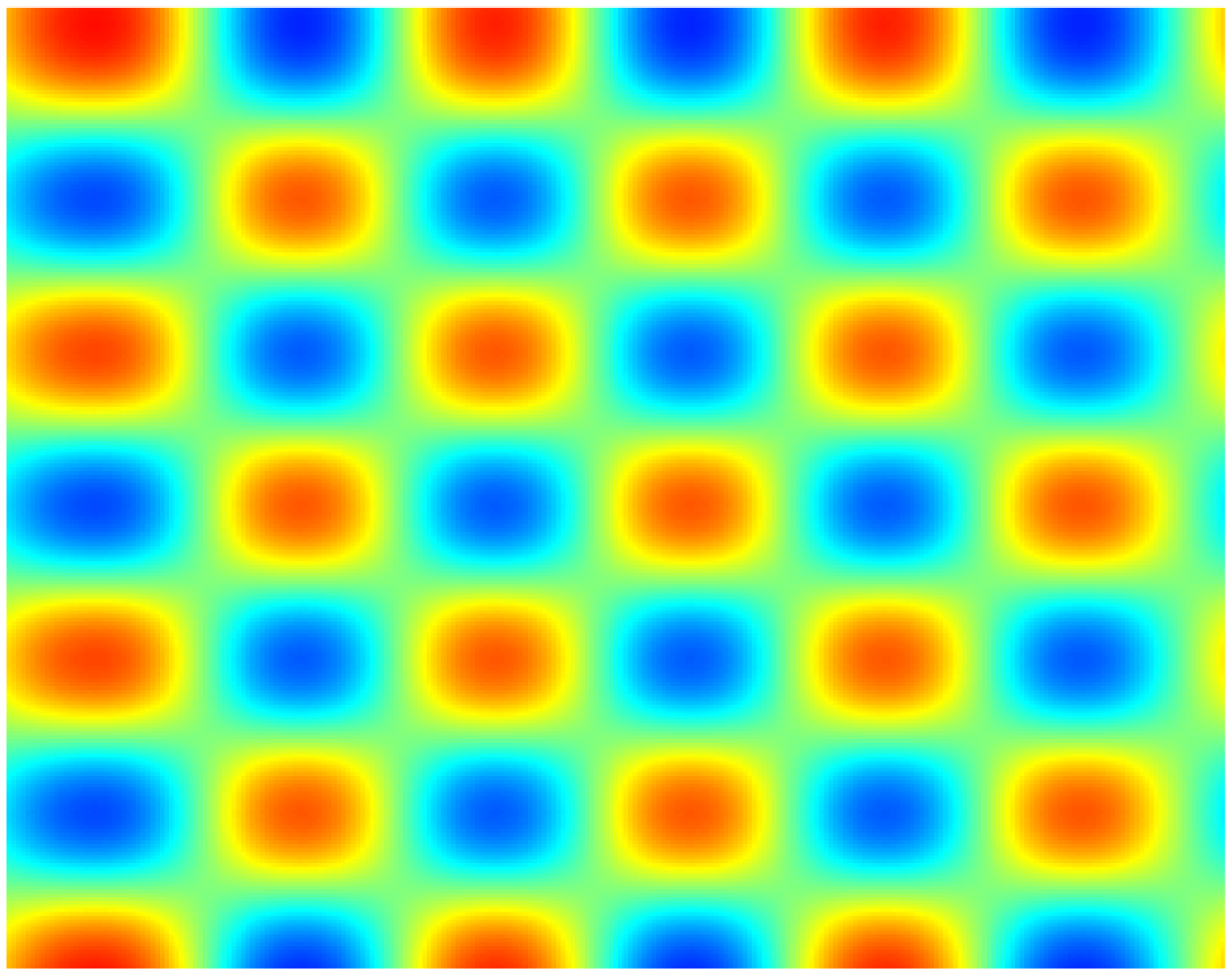}}
	\vskip -2mm
	\centerline{\includegraphics[scale=0.18]{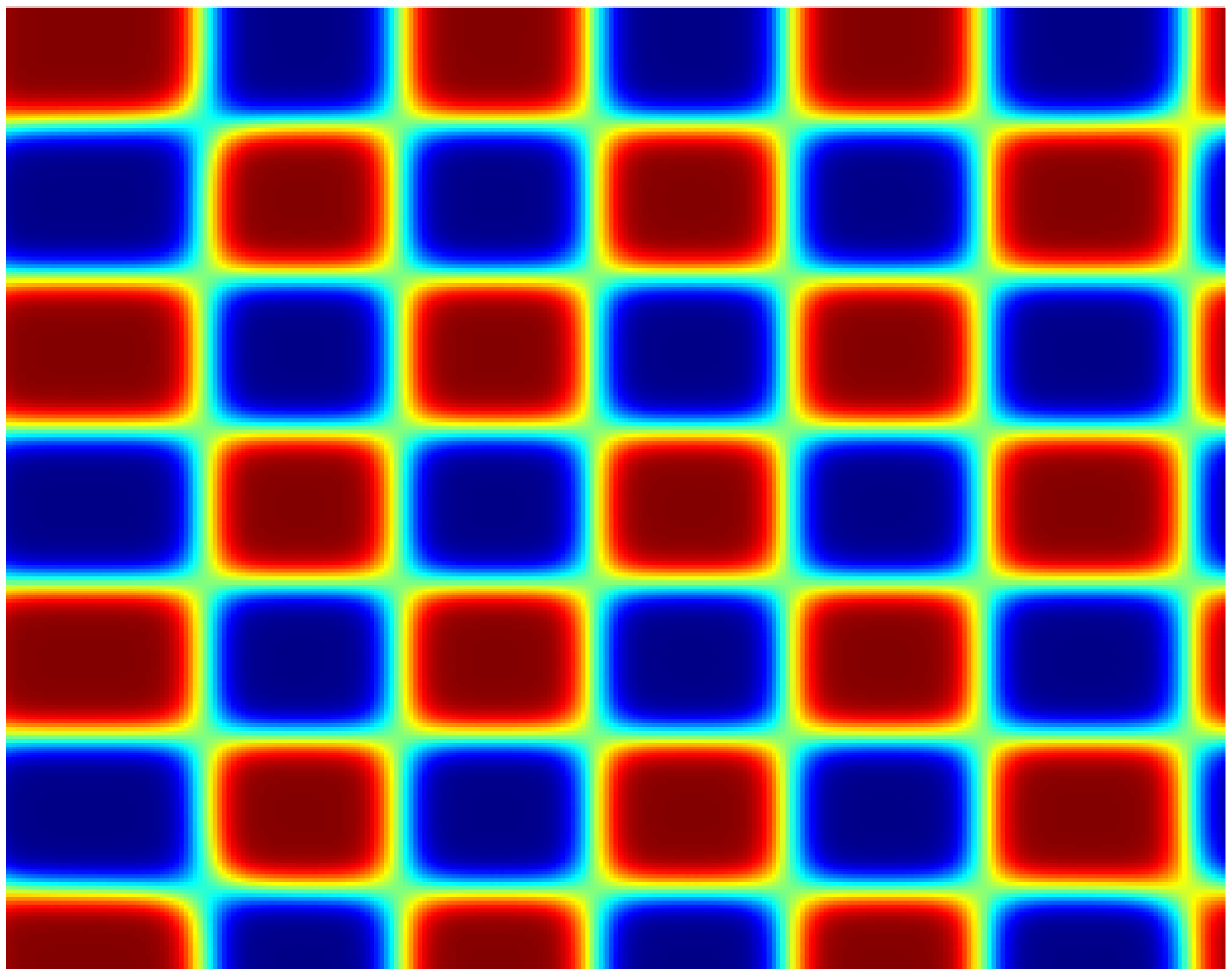}\includegraphics[scale=0.18]{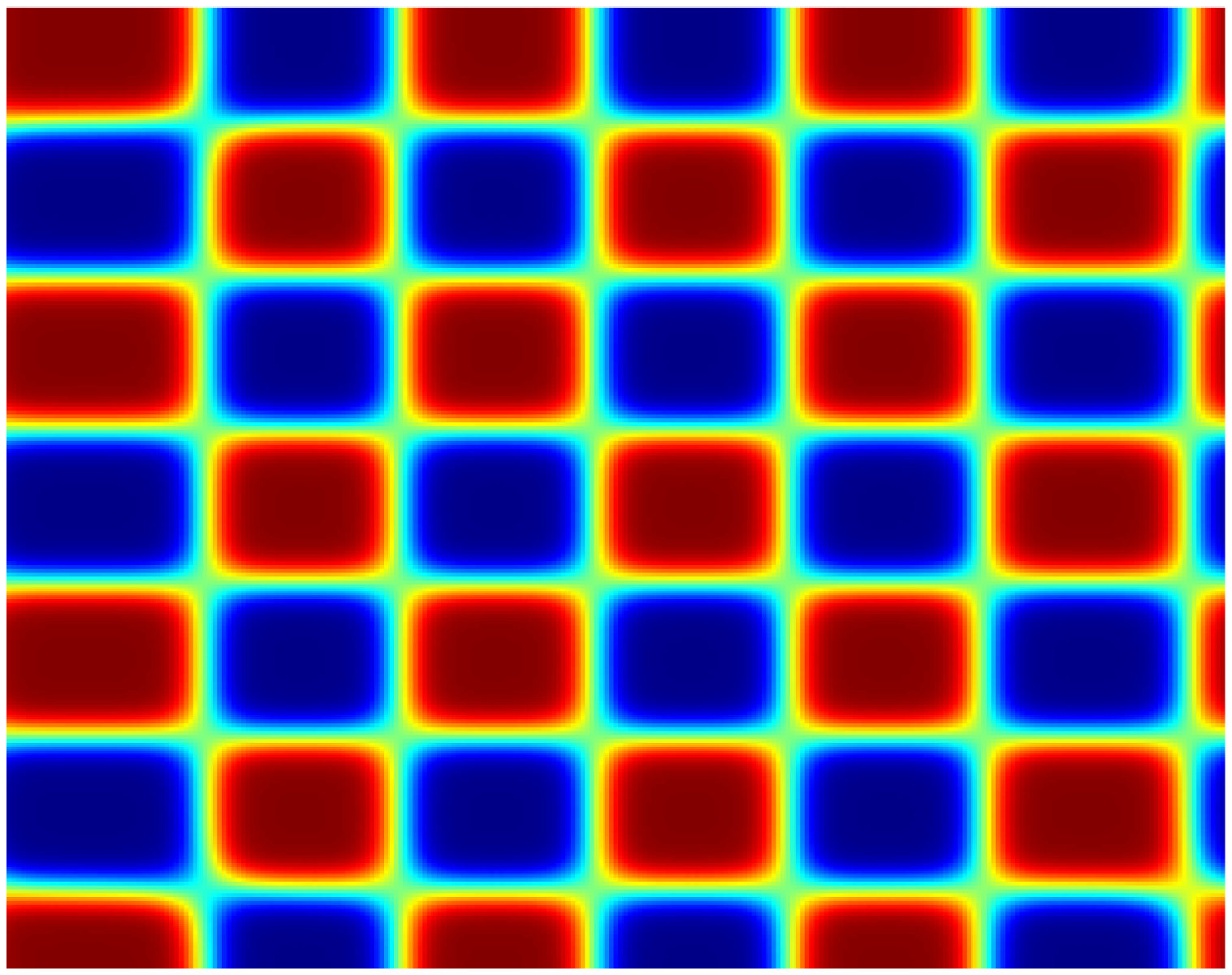}\includegraphics[scale=0.18]{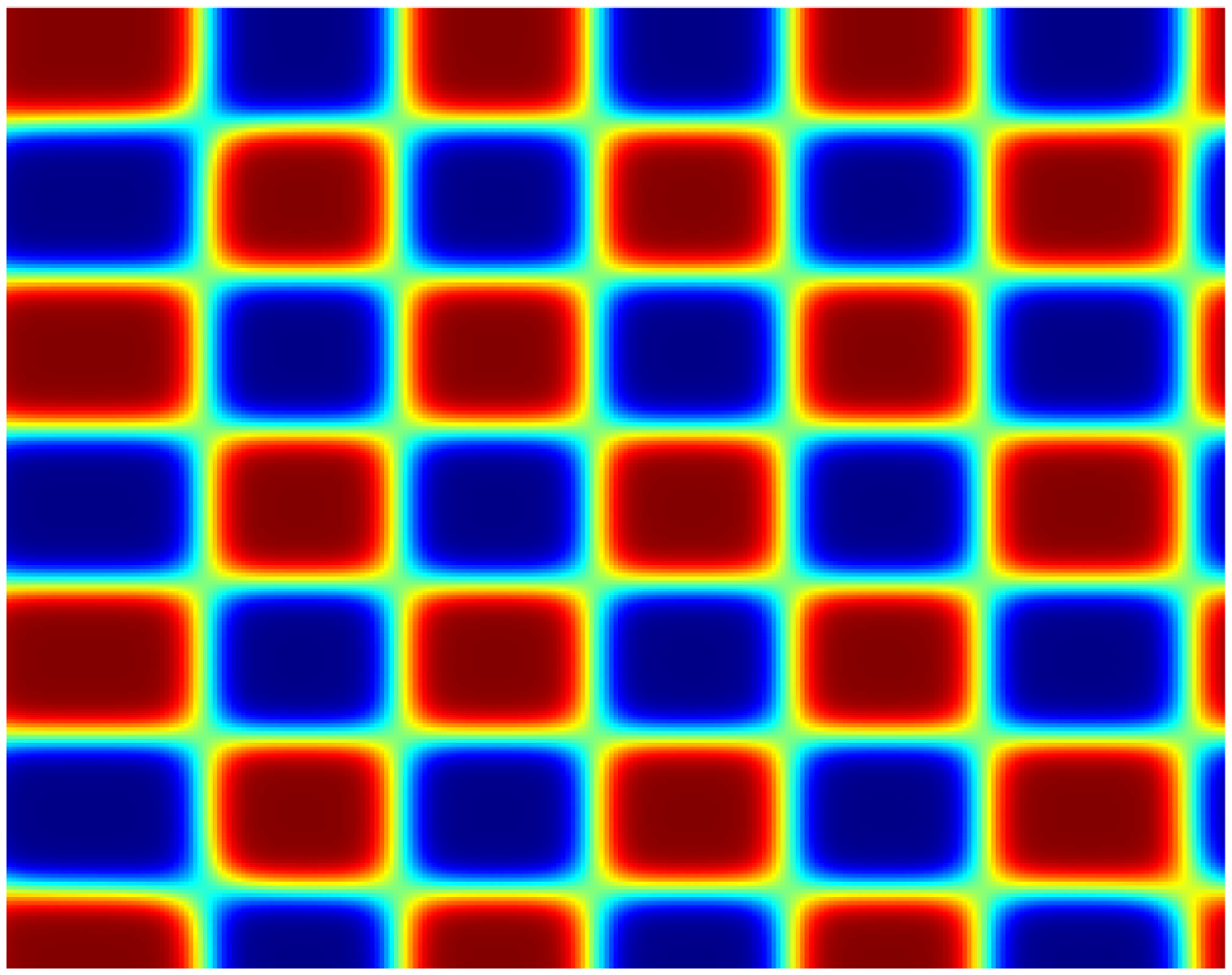}\includegraphics[scale=0.18]{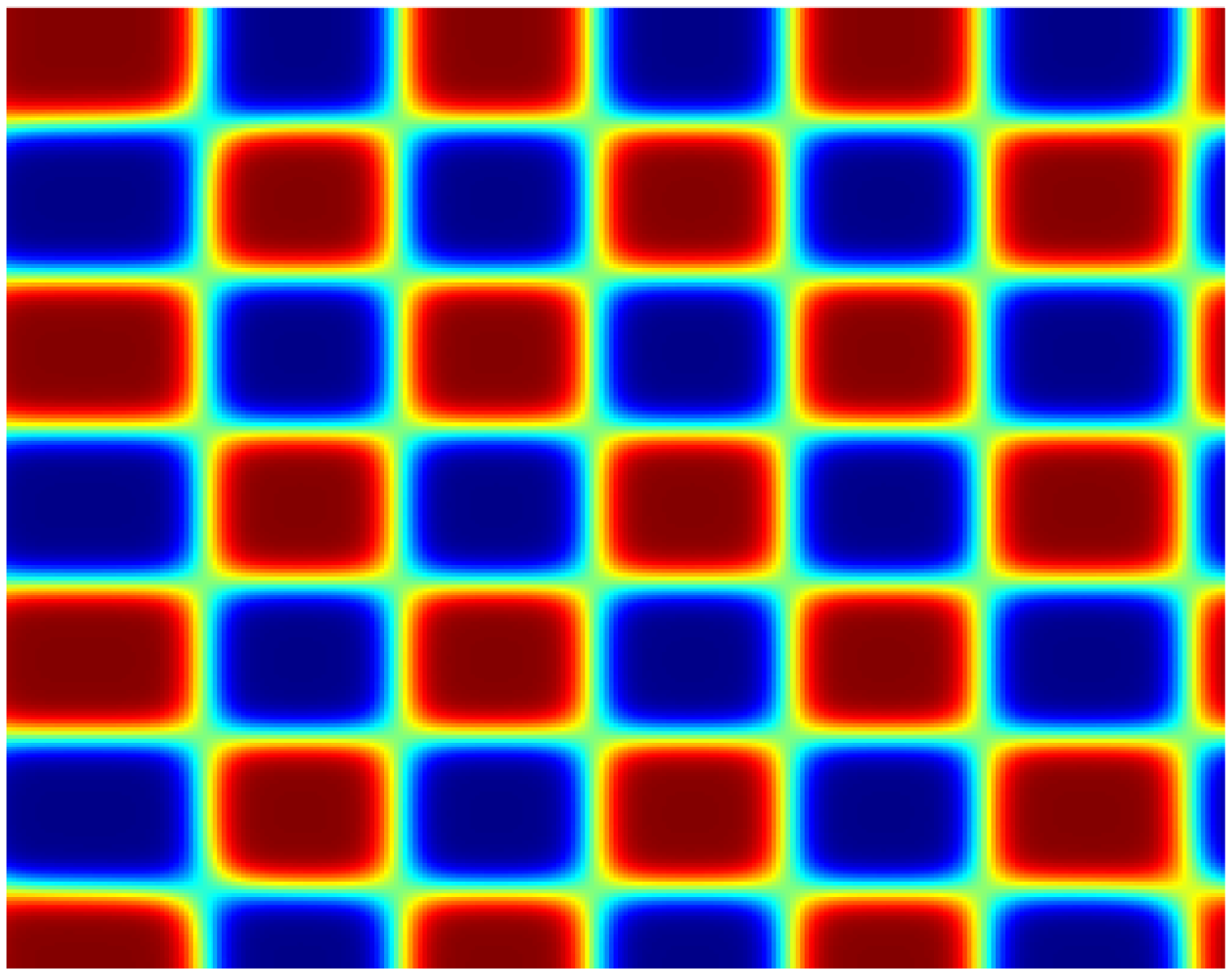}}
	\vskip -2mm
	\centerline{\includegraphics[scale=0.18]{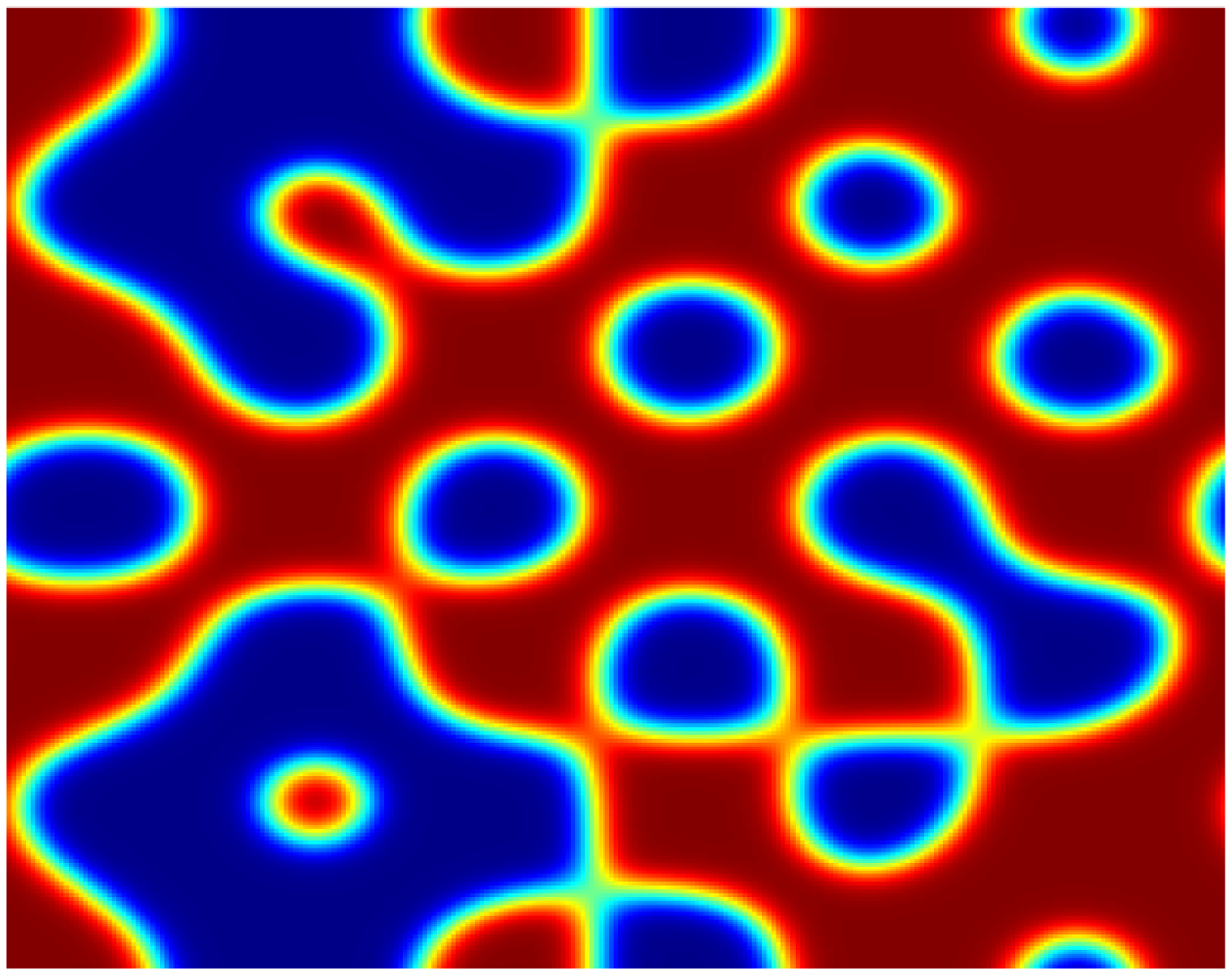}\includegraphics[scale=0.18]{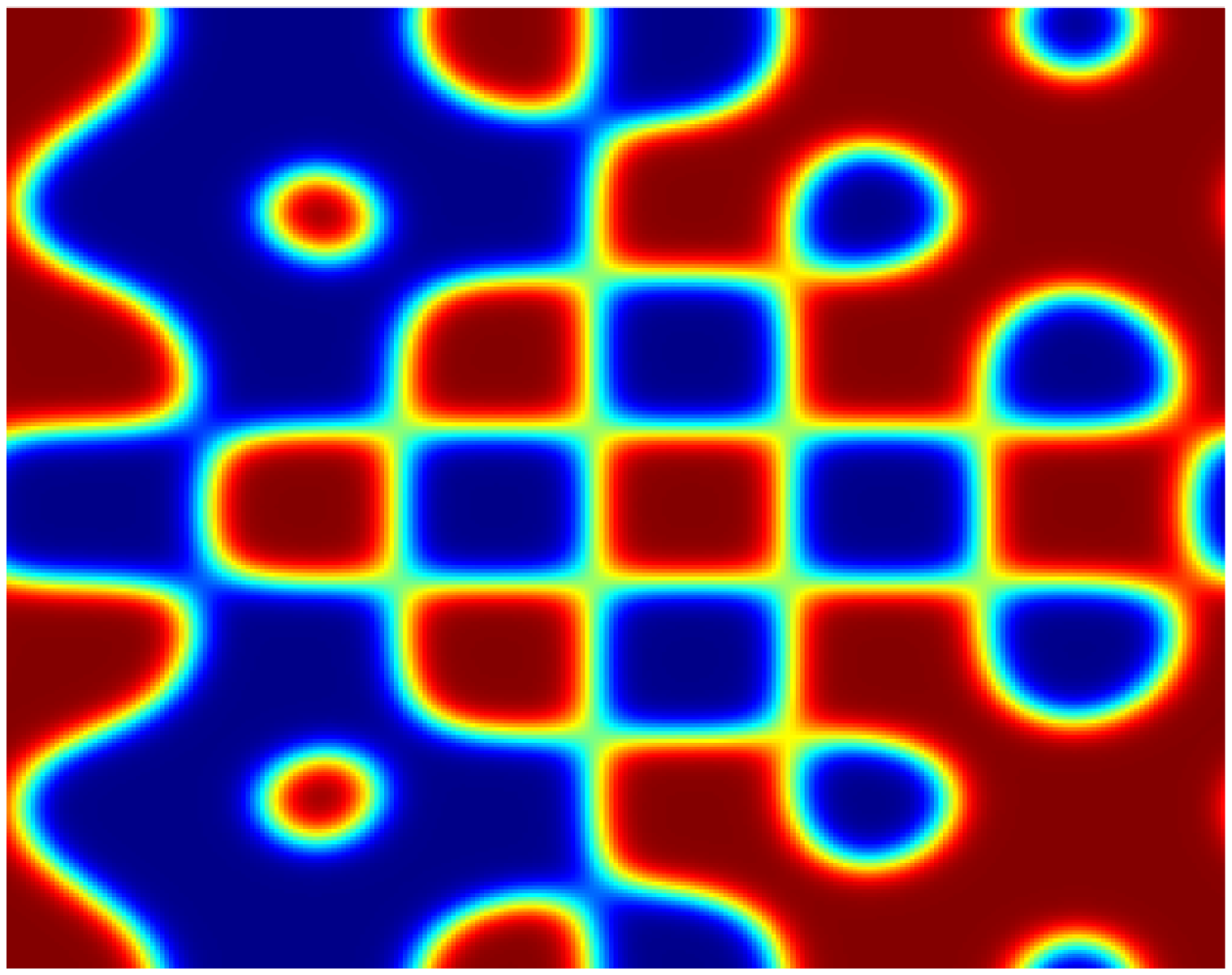}\includegraphics[scale=0.18]{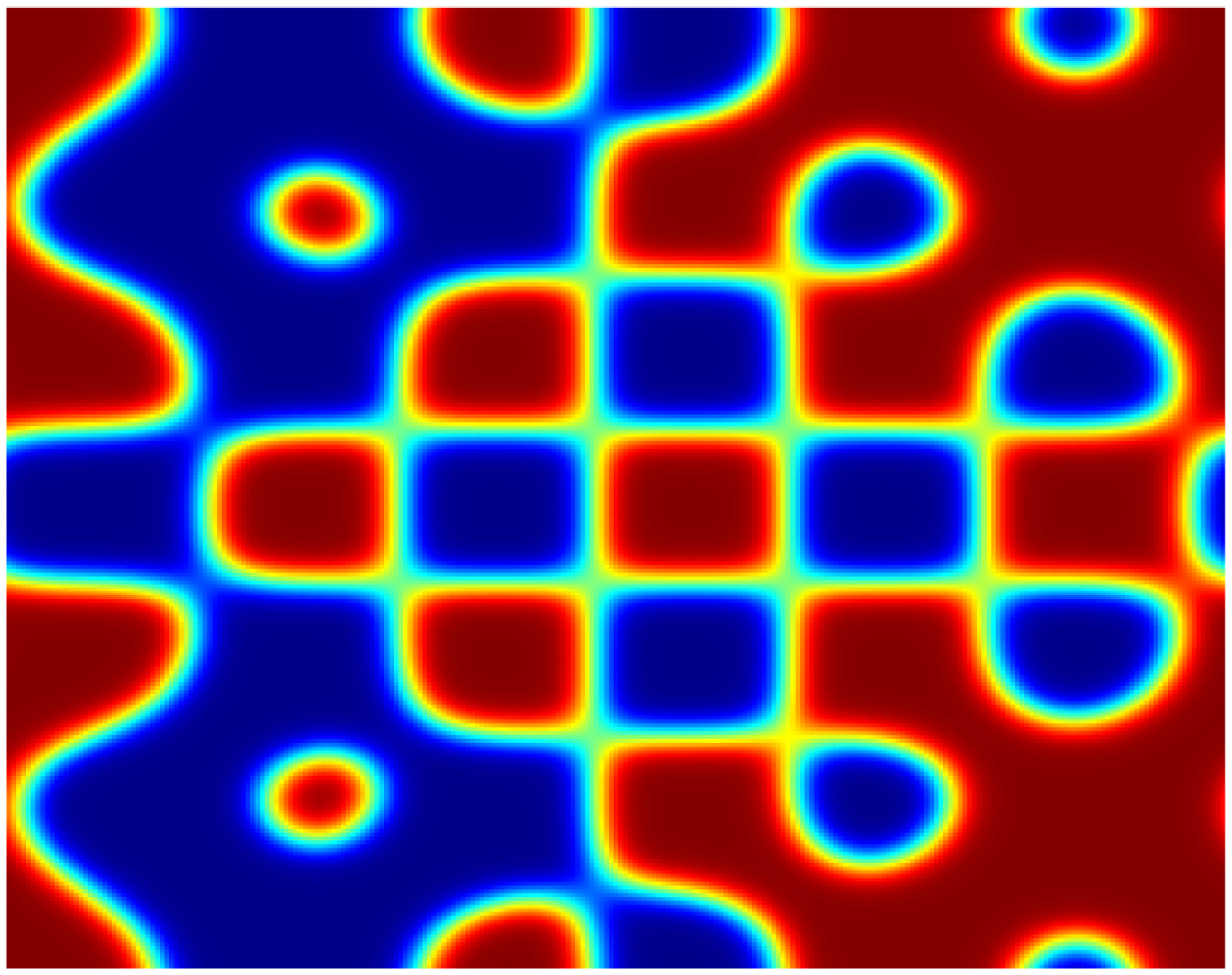}\includegraphics[scale=0.18]{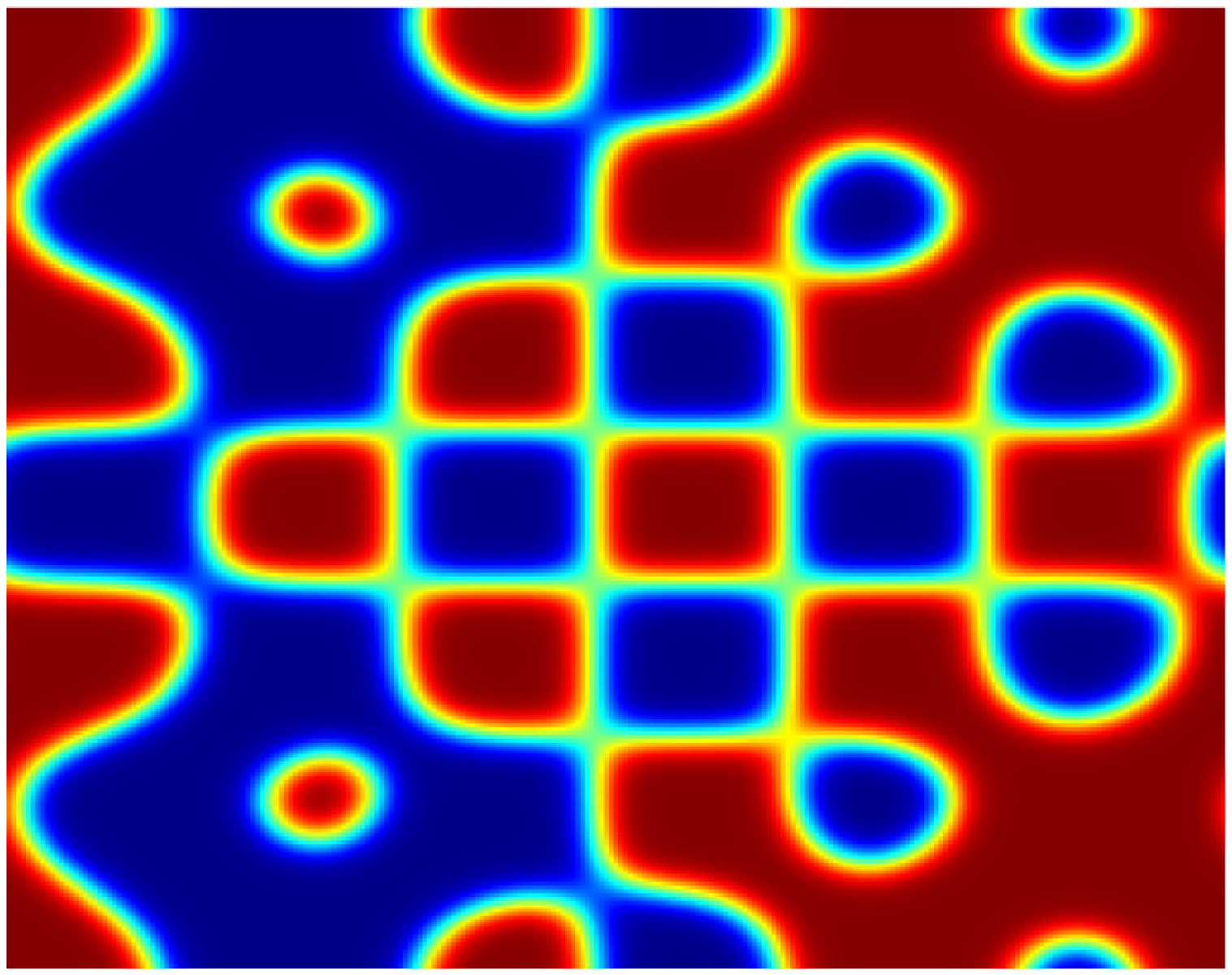}}
	\vskip -2mm
	\centerline{\includegraphics[scale=0.18]{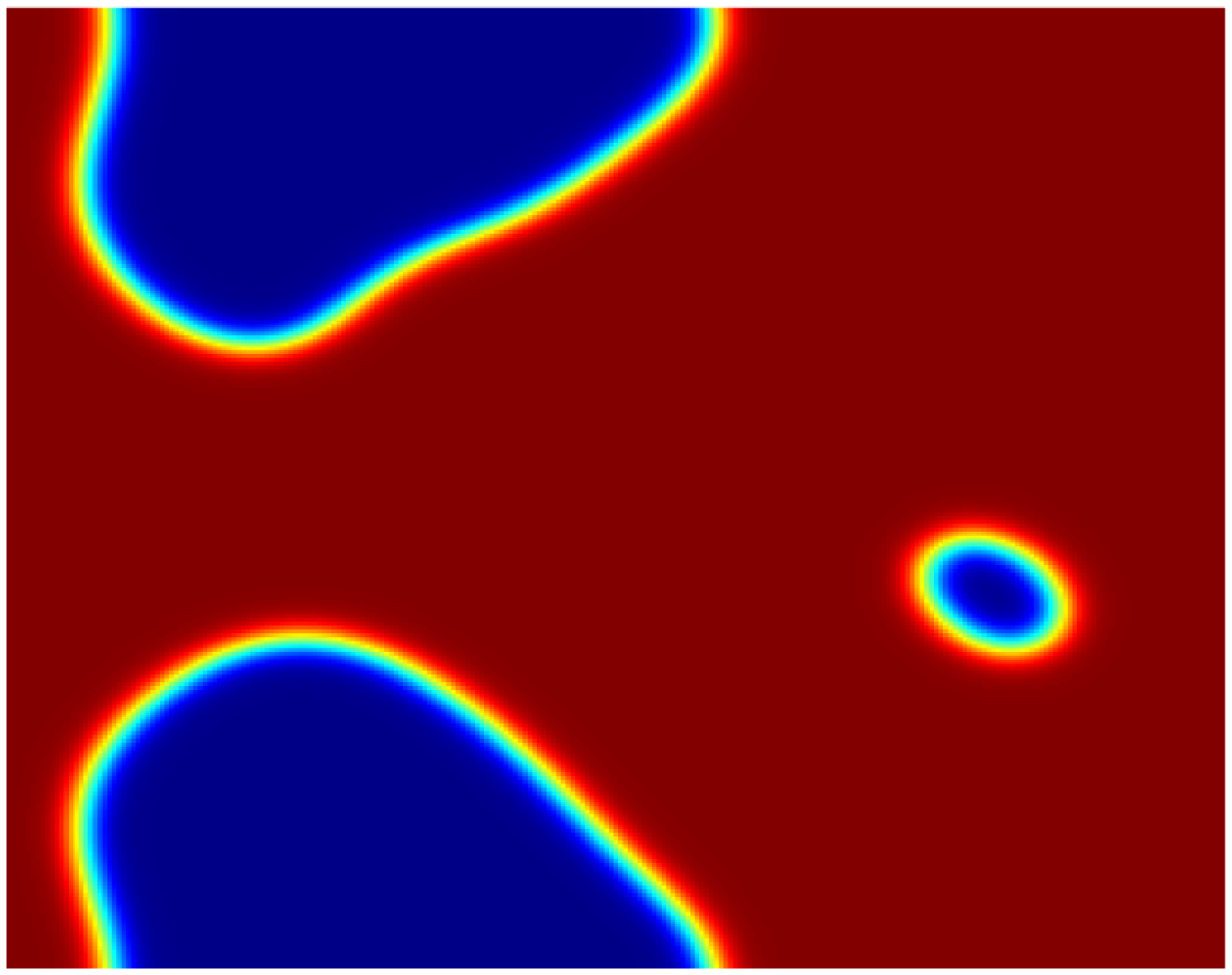}\includegraphics[scale=0.18]{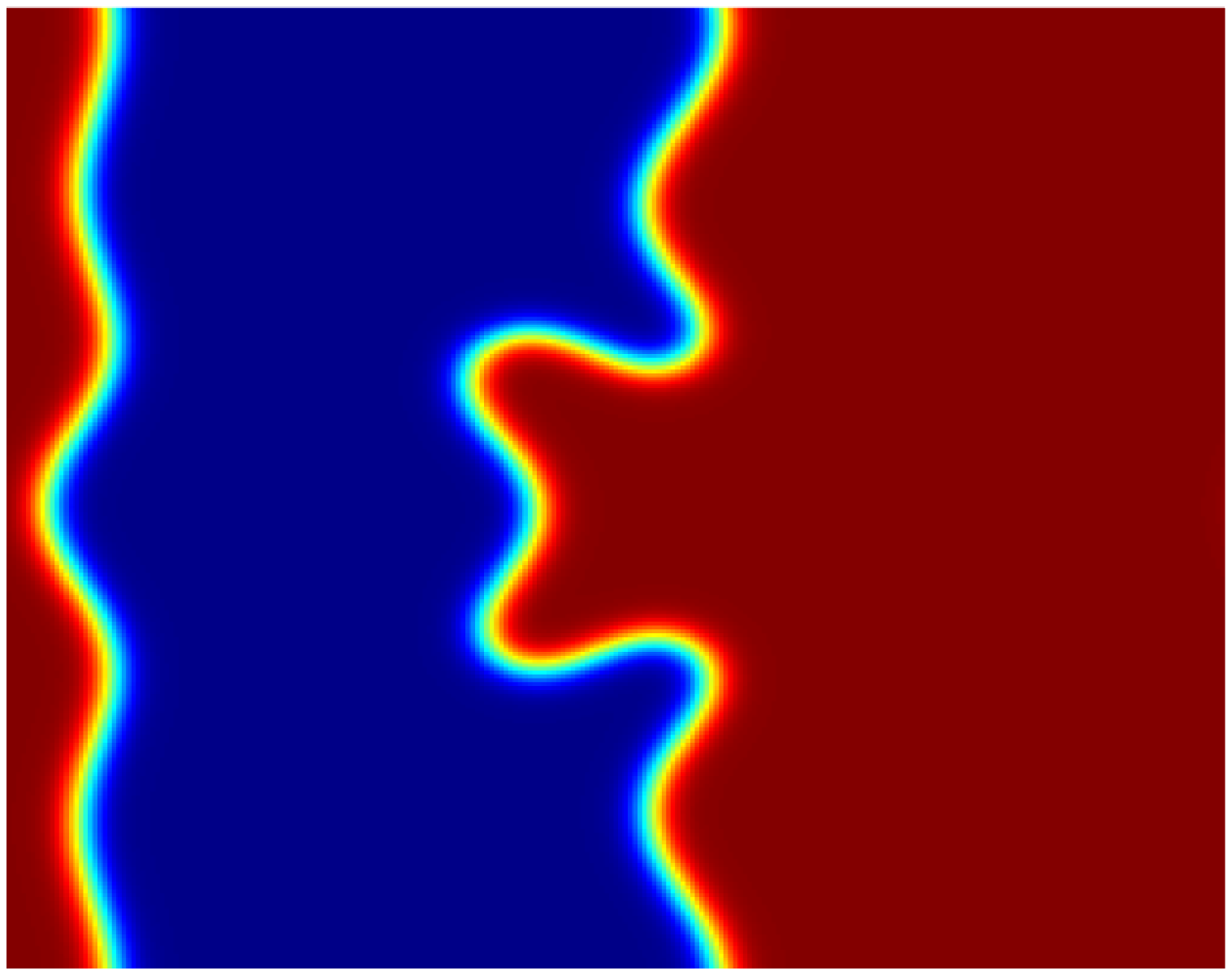}\includegraphics[scale=0.18]{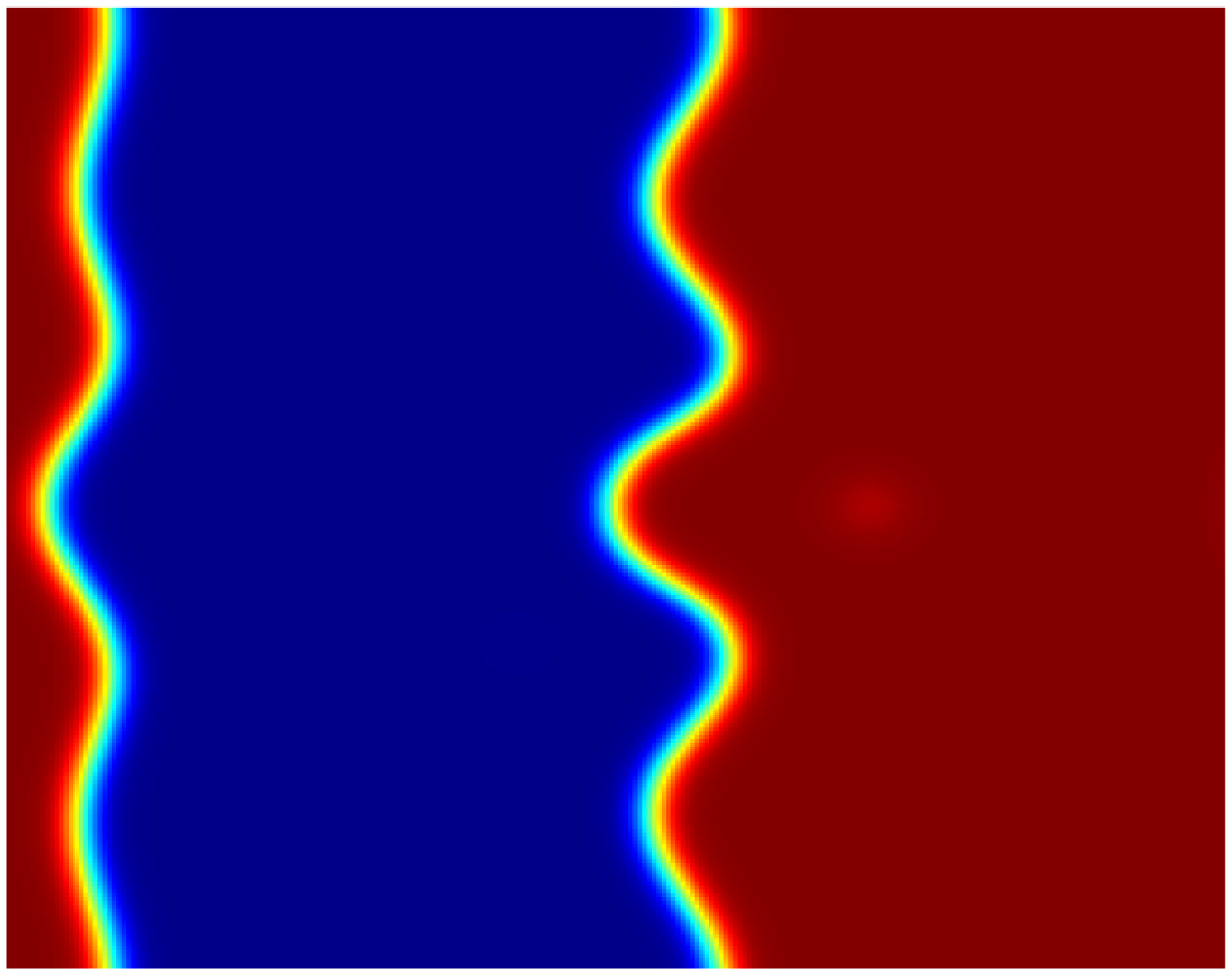}\includegraphics[scale=0.18]{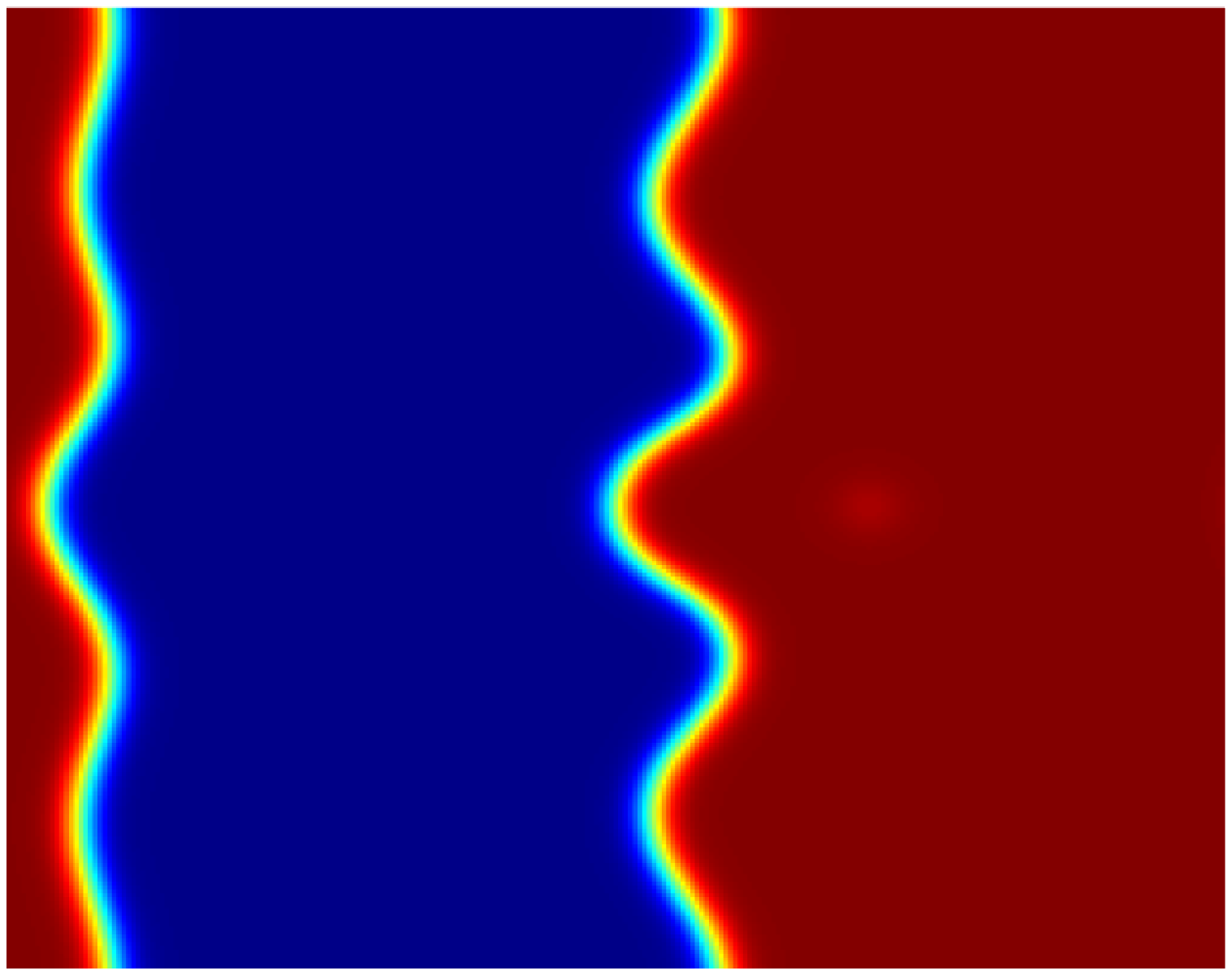}}
	\vskip -2mm
	\centerline{\includegraphics[scale=0.18]{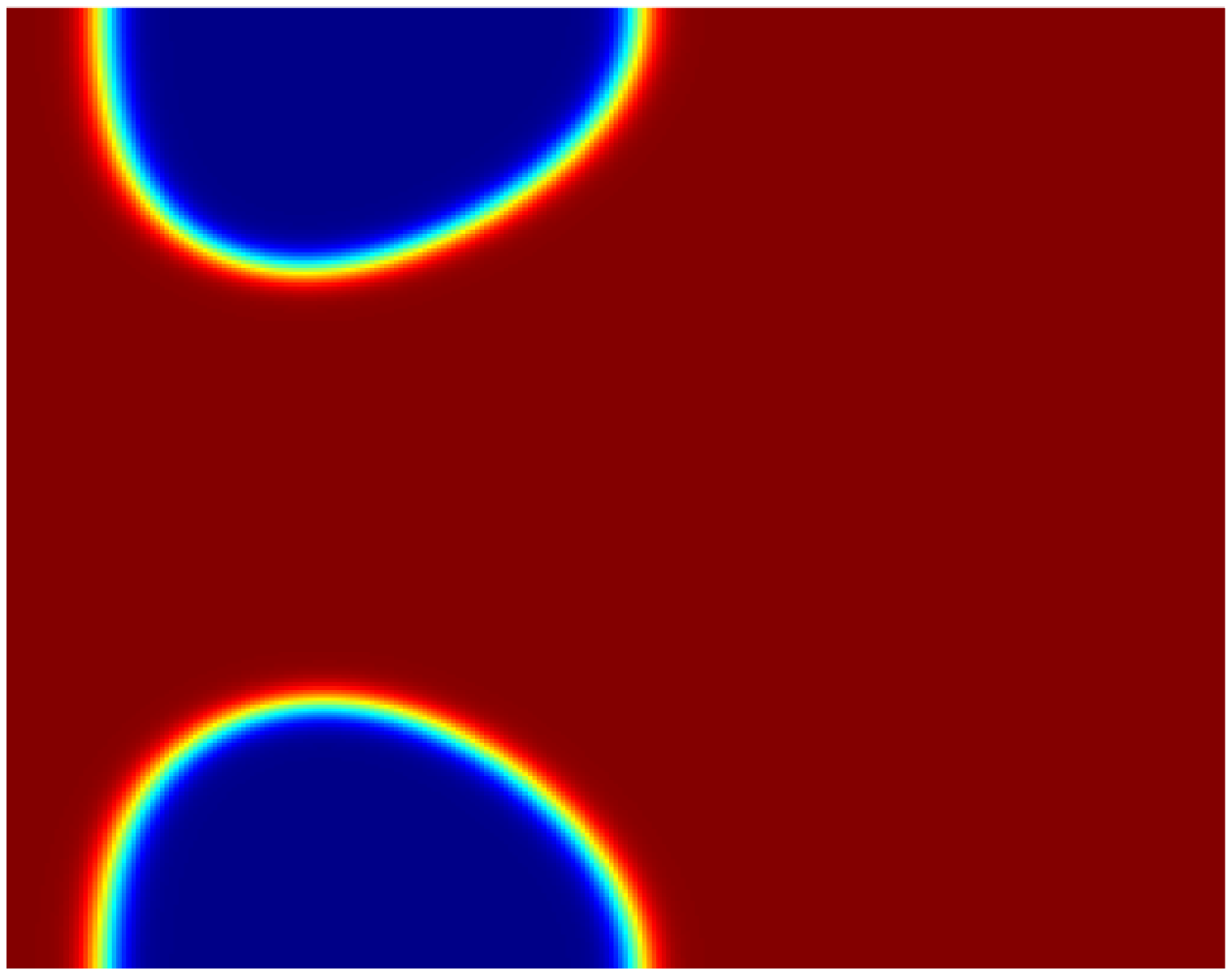}\includegraphics[scale=0.18]{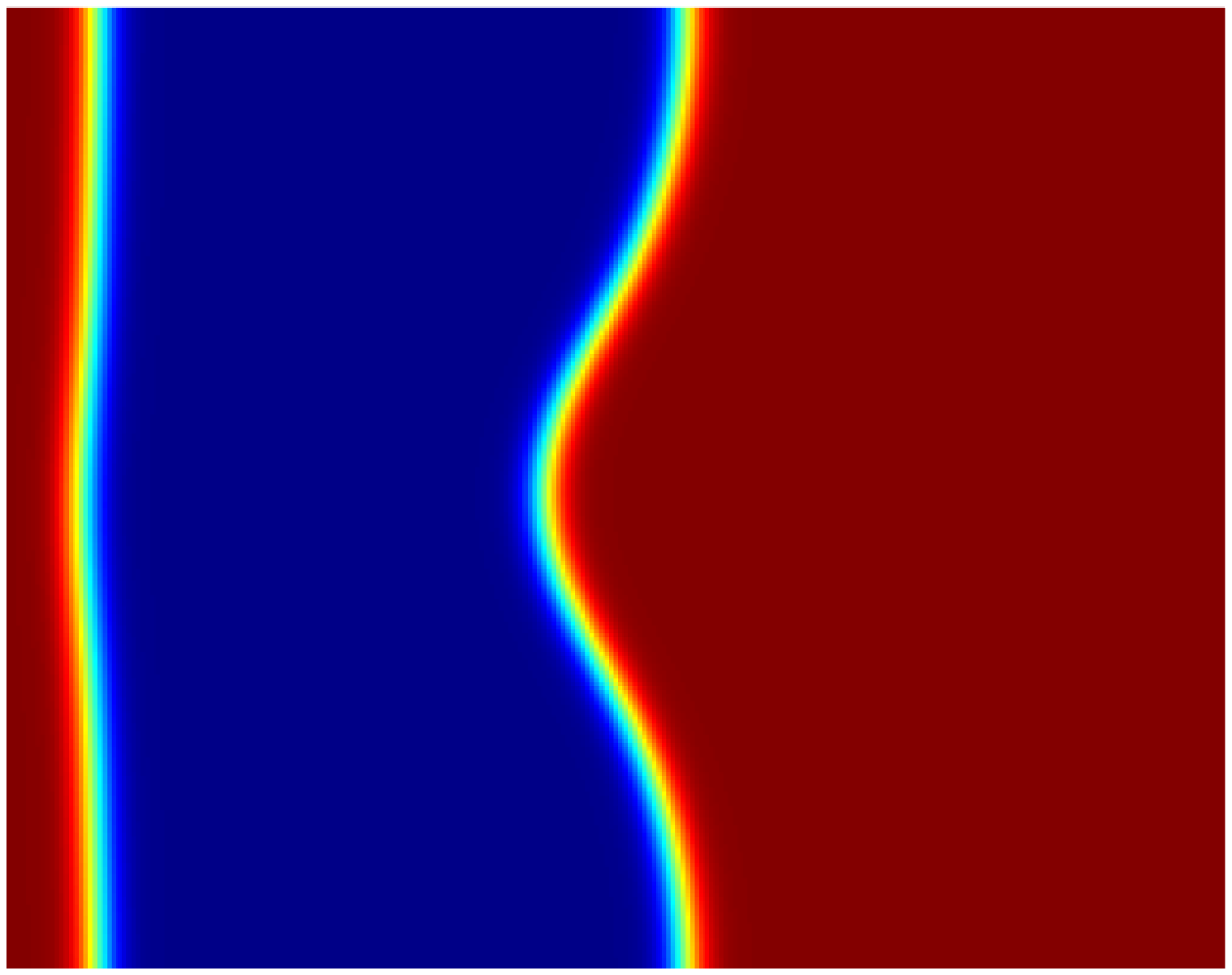}\includegraphics[scale=0.18]{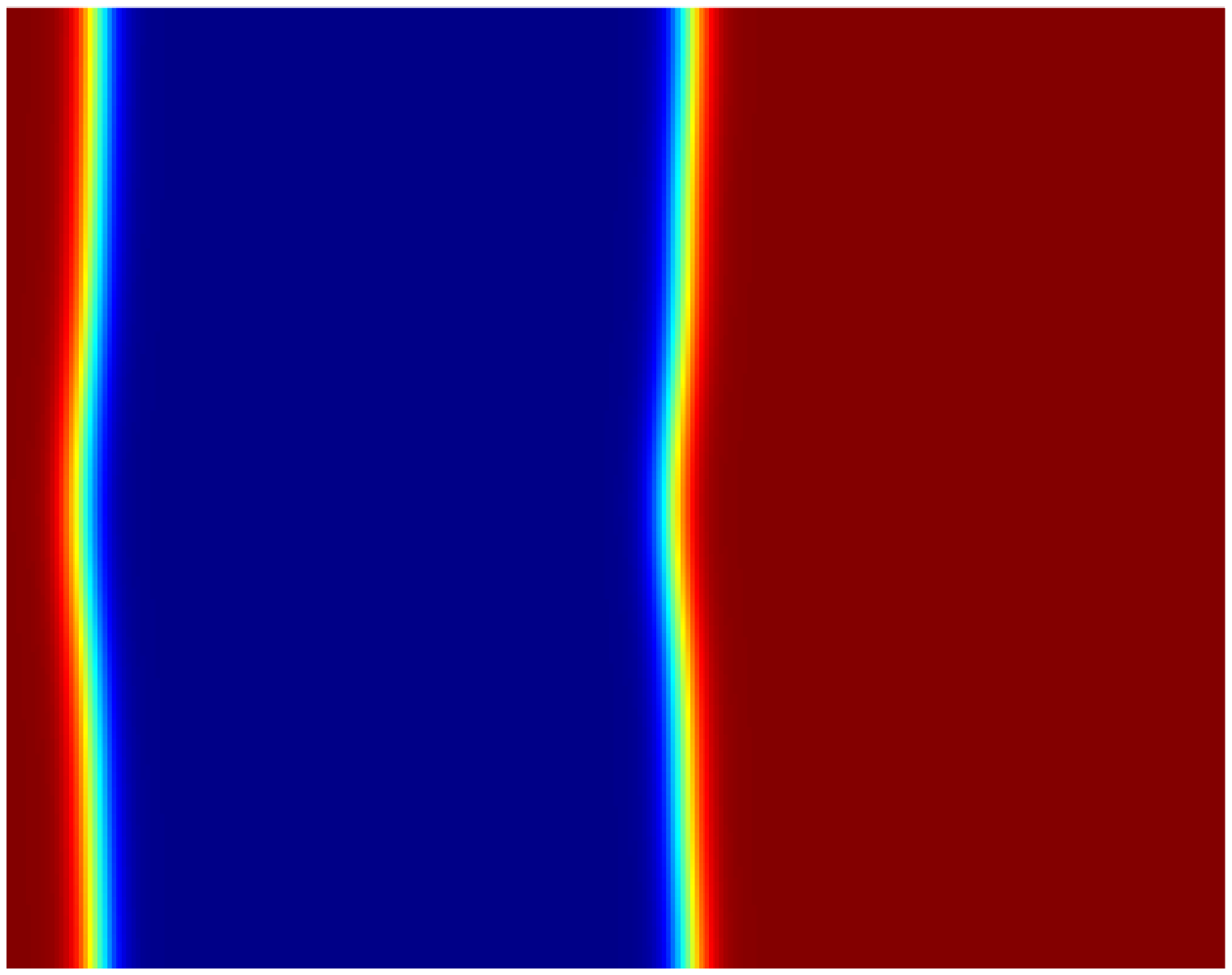}\includegraphics[scale=0.18]{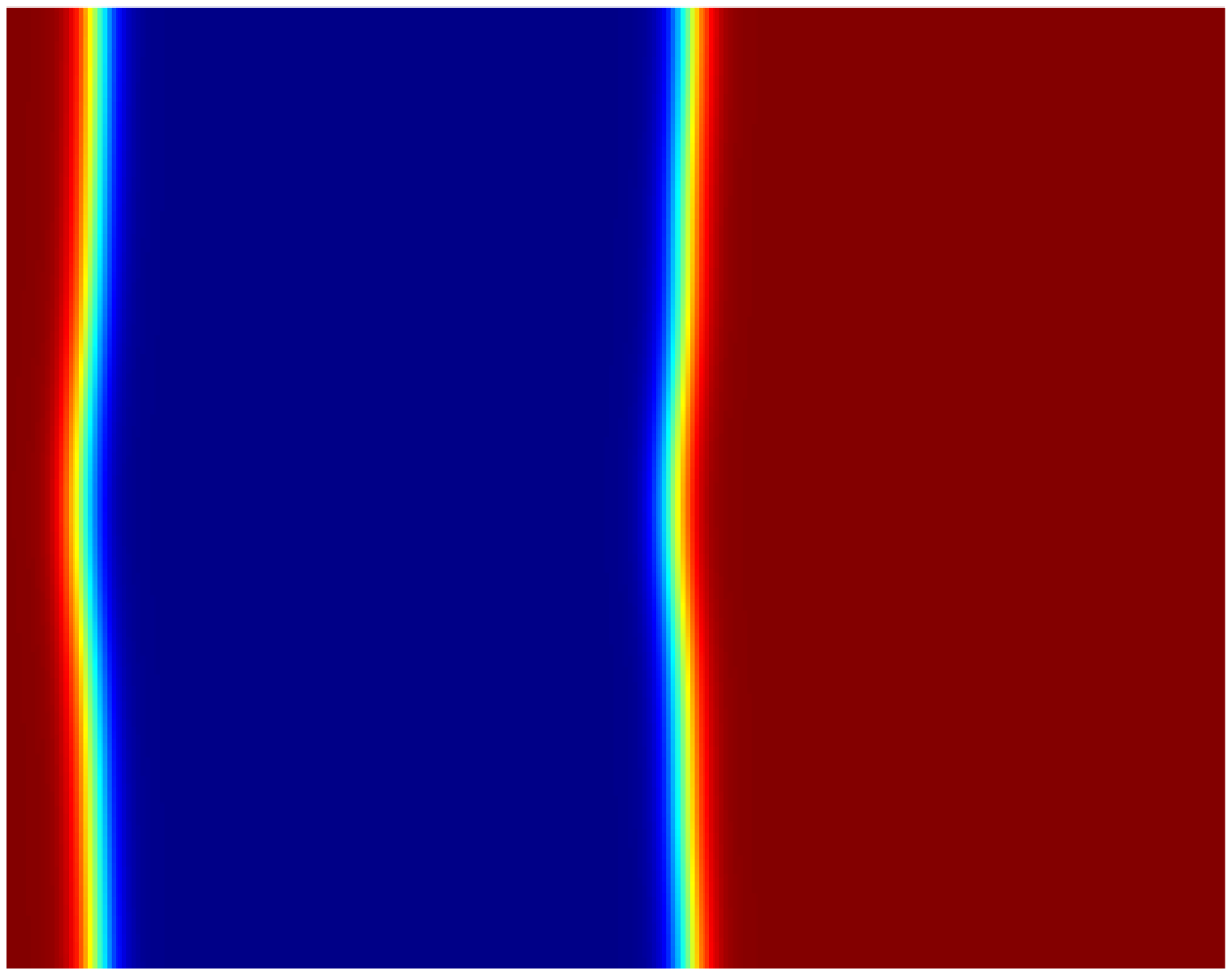}}
	\vskip -2mm
	\centerline{\includegraphics[scale=0.18]{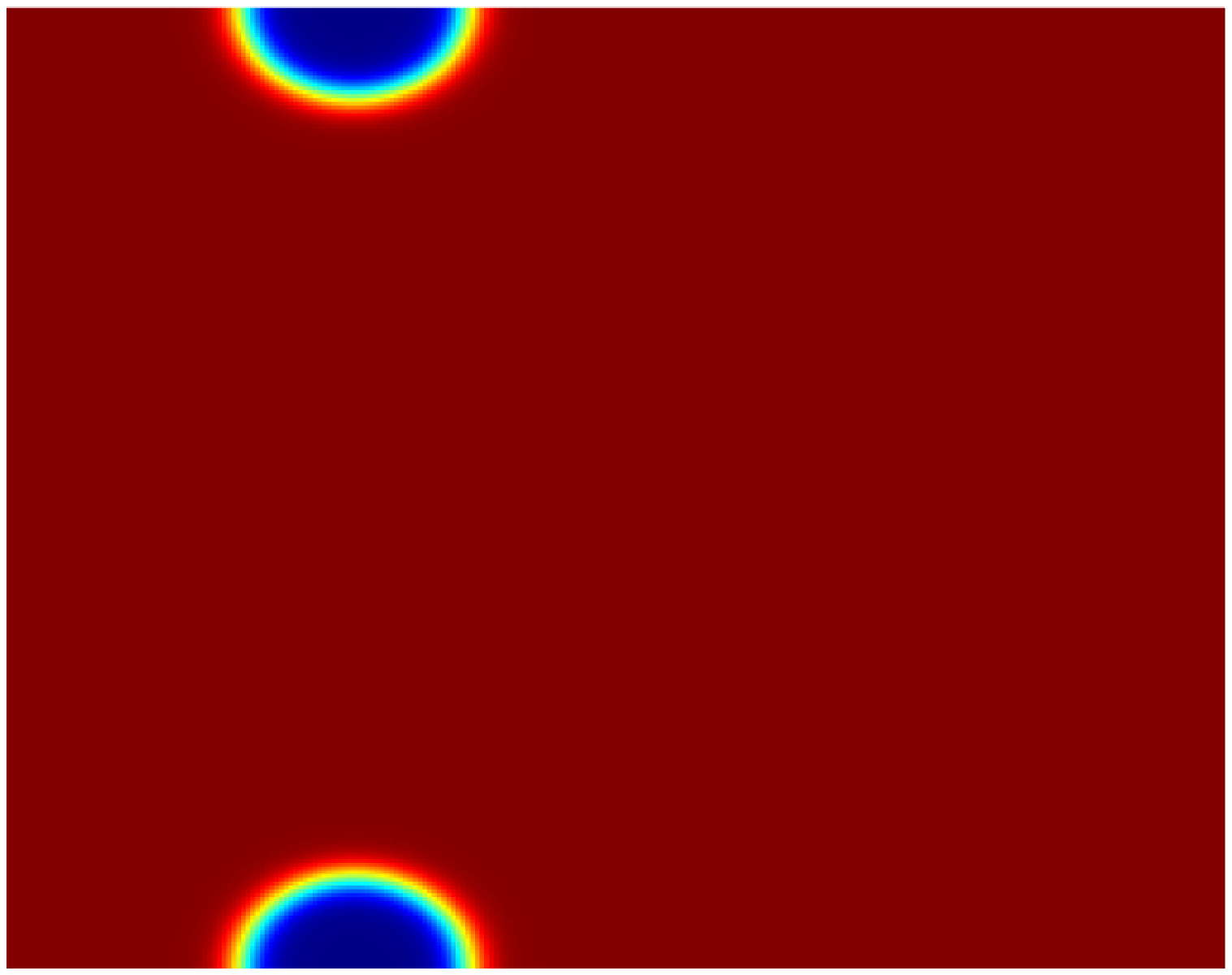}\includegraphics[scale=0.18]{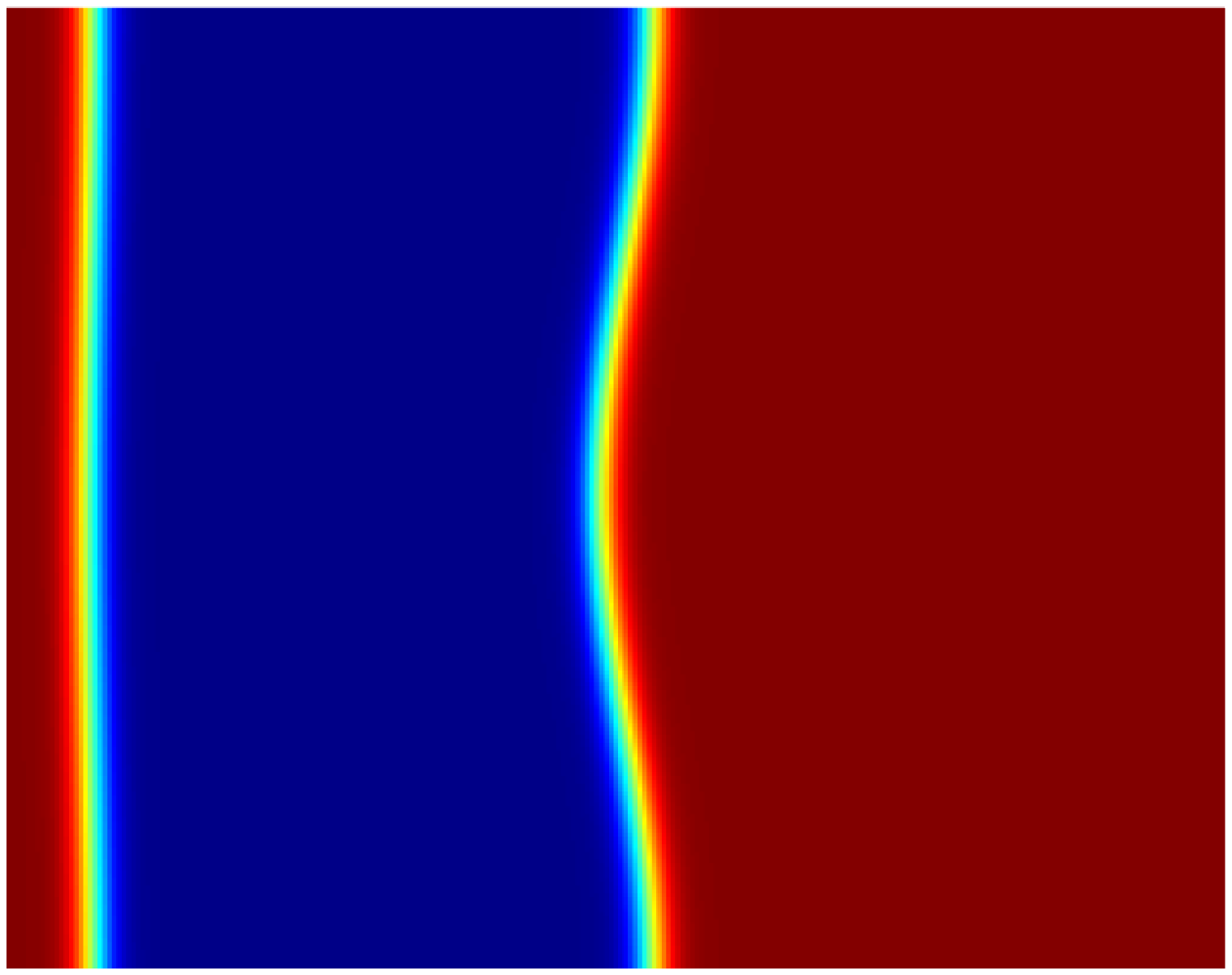}\includegraphics[scale=0.18]{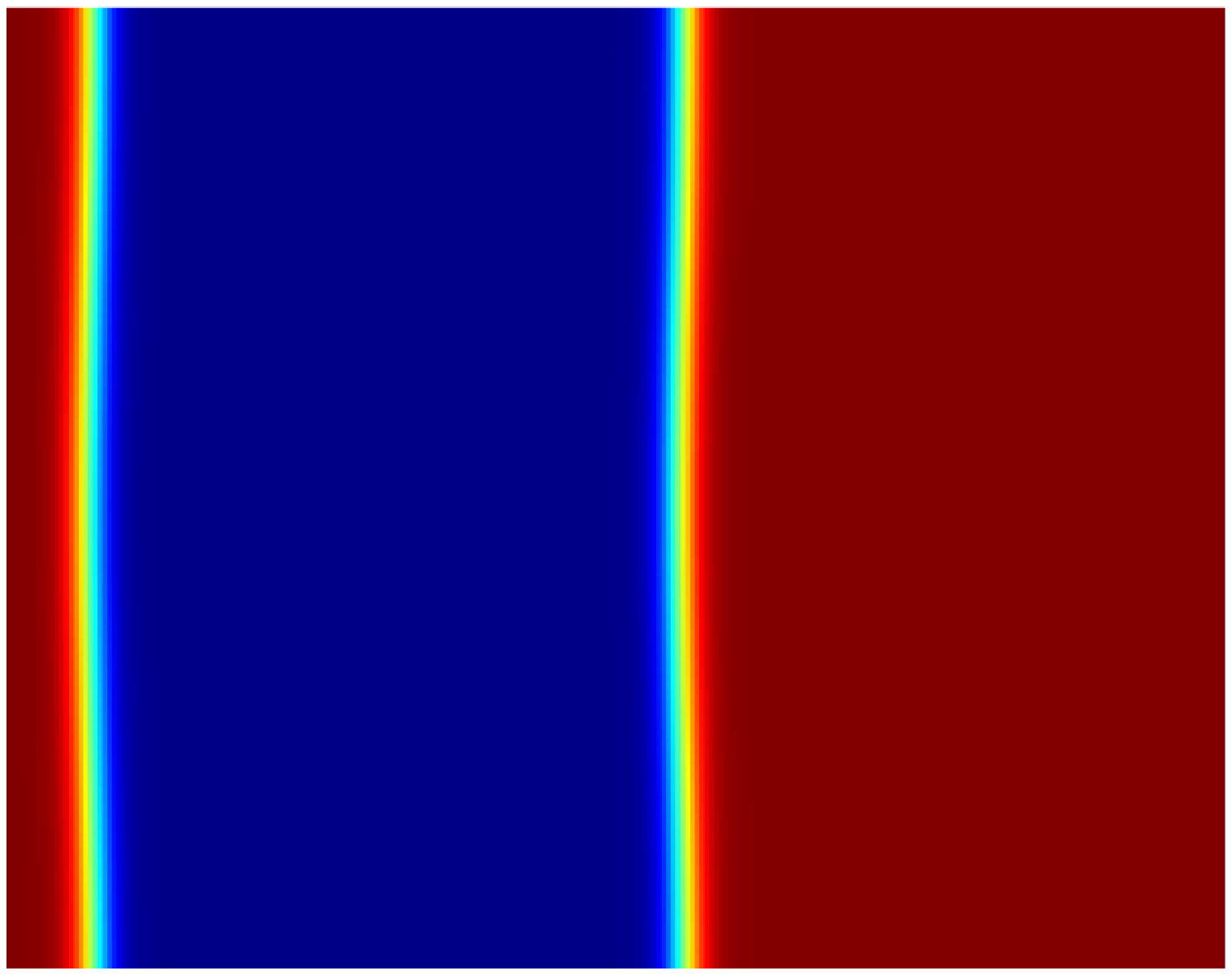}\includegraphics[scale=0.18]{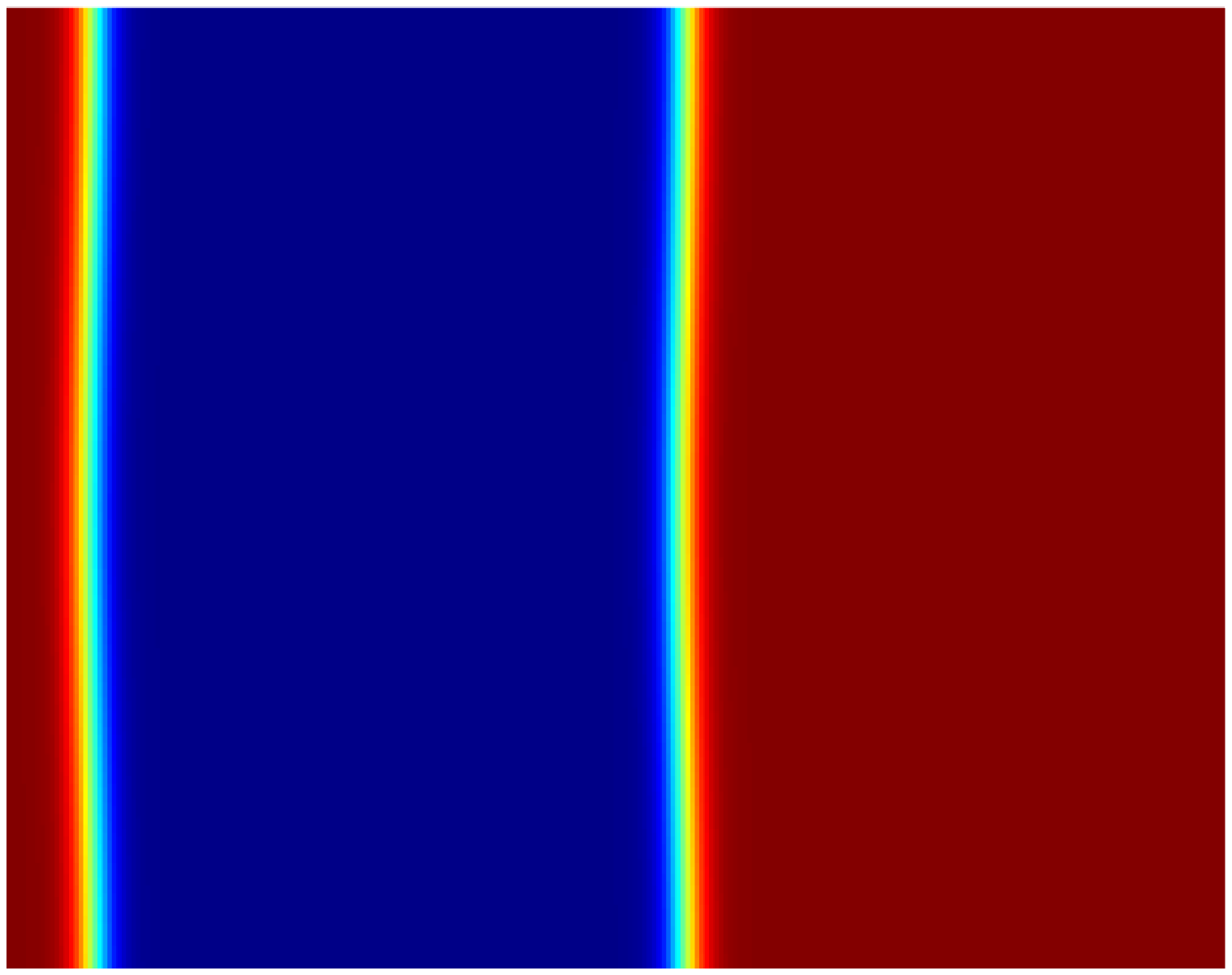}}
	\vskip -2mm
	\centerline{\includegraphics[scale=0.18]{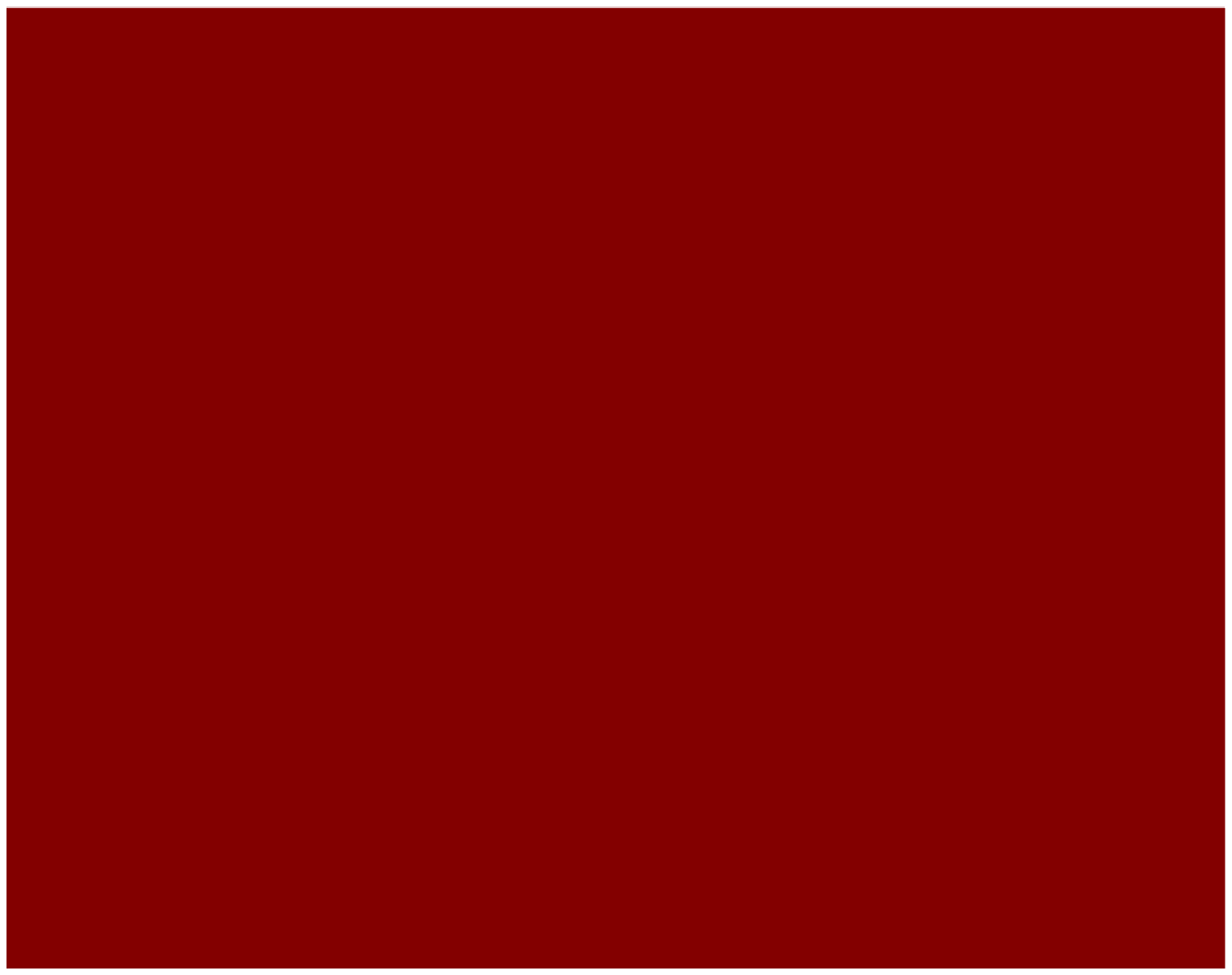}\includegraphics[scale=0.18]{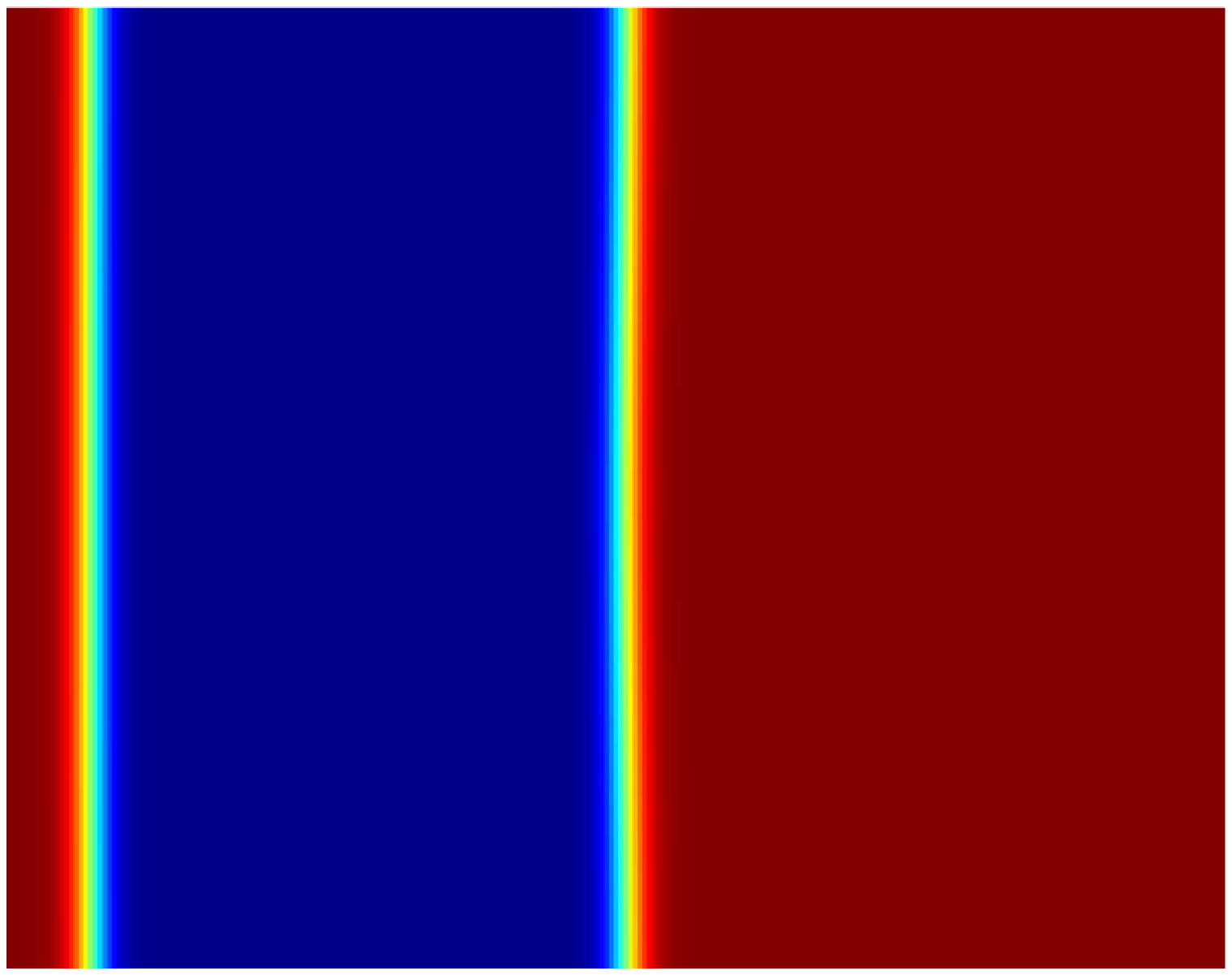}\includegraphics[scale=0.18]{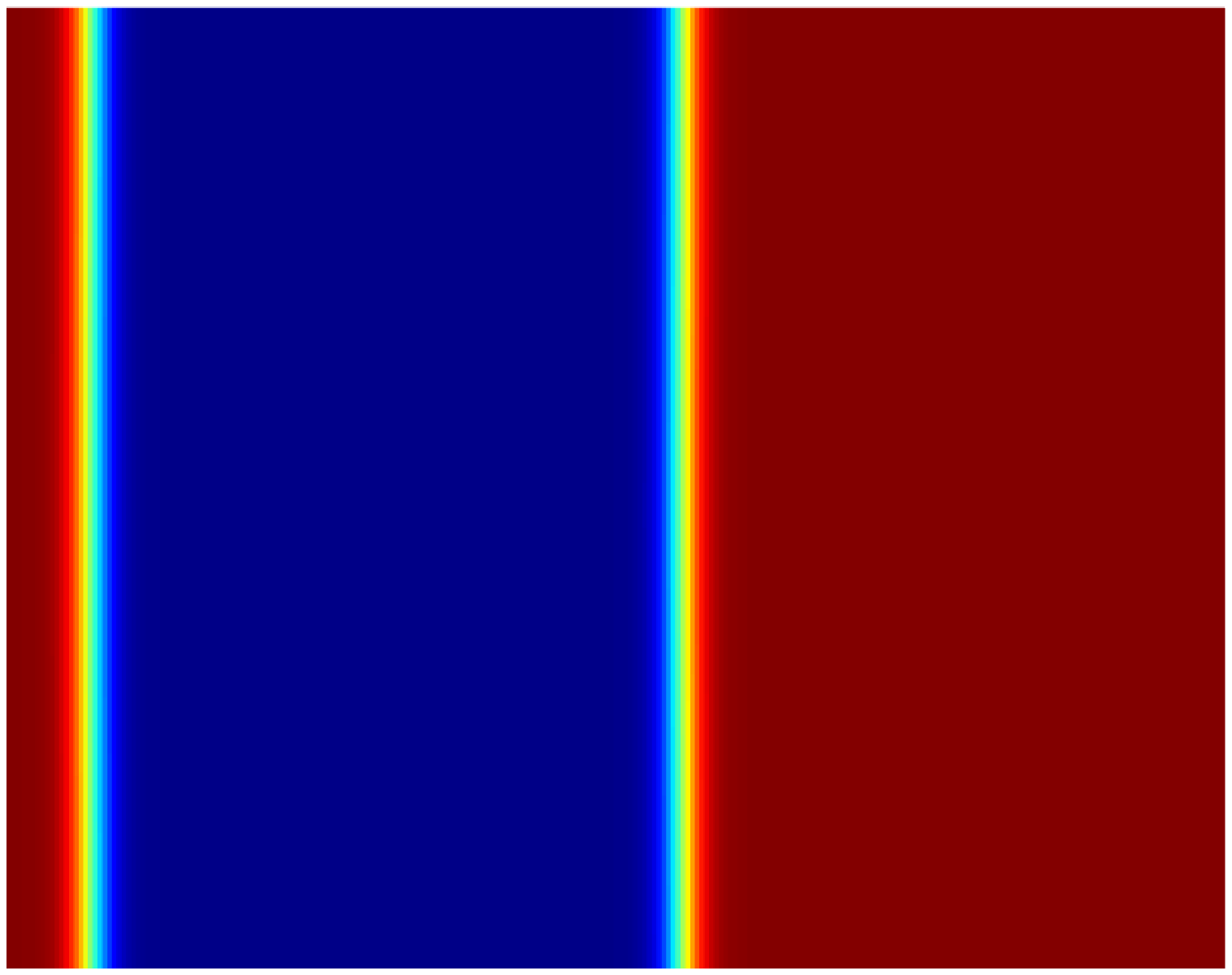}\includegraphics[scale=0.18]{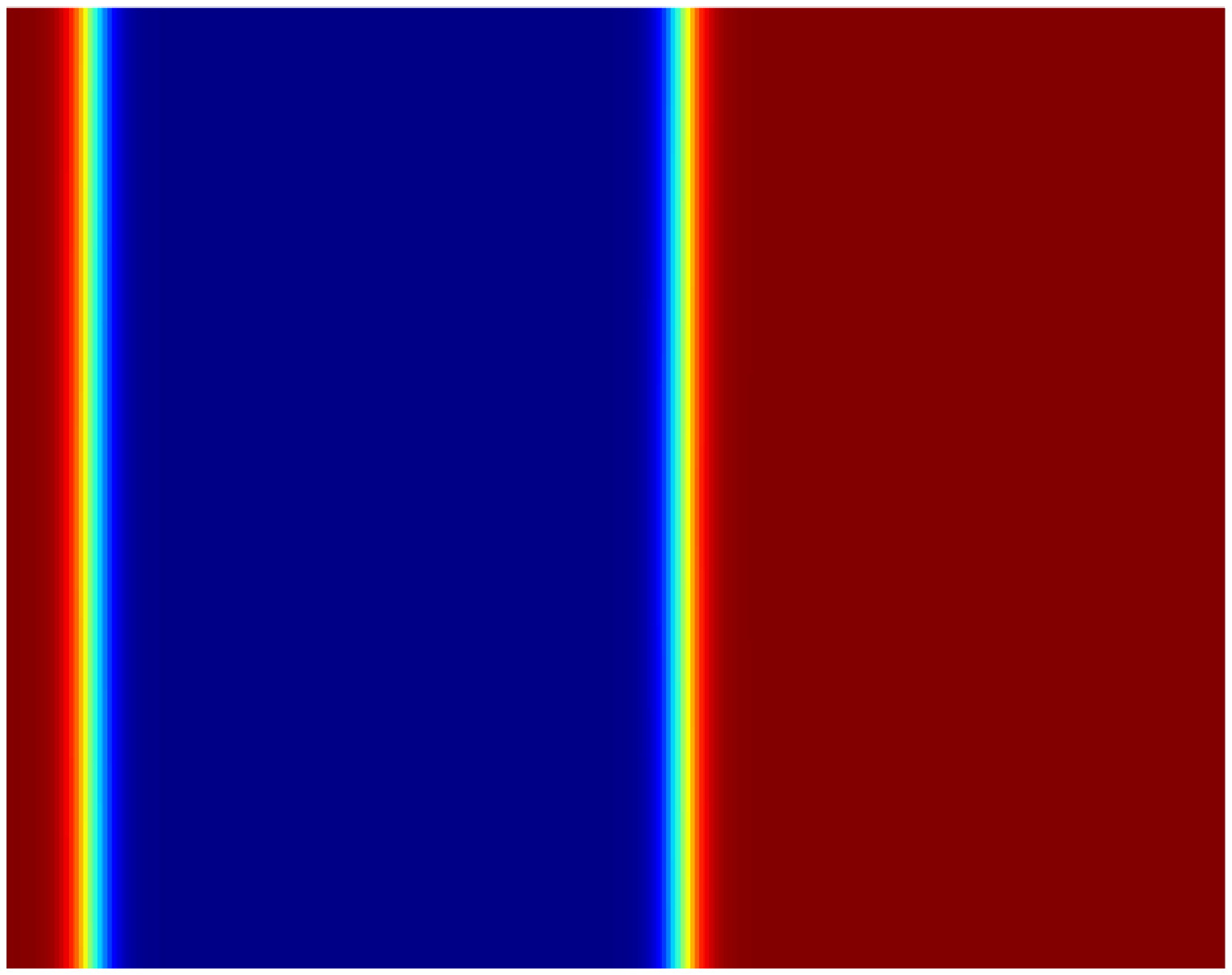}}
	\caption{Snapshots of the phase function $u$ produced by scheme \eqref{sta-semi} with $\epsilon=1e-6$(the first column), $\epsilon=1e-8$(the second column), $\epsilon=1e-10$(the third column) for \eqref{log-reg-sac} and the stabilized ETDRK2 scheme in \cite{DJLQ21} (the last column) for the corresponding deterministic model at the times $t=5$, $20$, $40$, $60$, $100$, $200$ and $500$ from top to bottom.}
	\label{fig6}
\end{figure*}

The simulations of coarsening dynamics for the model \eqref{log-reg-sac} and the deterministic model are performed by the stabilized semi-implicit scheme \eqref{sta-semi} and the second-order stabilized exponential time-differencing Runge--Kutta (ETDRK2) scheme proposed in \cite{DJLQ21}, respectively. This time step size used for the testing schemes is set to be $\tau=1e-3$.  For the spatial discretization, the Fourier spectral method with $256\times256$ modes is used for the model \eqref{log-reg-sac}, and the central finite difference method with uniform spatial grid spacing $h=2\pi/256$ is applied to the deterministic model.
Figure \ref{fig5} shows a comparison on the energy evolution between  \eqref{log-reg-sac} with several different noise densities $\epsilon$ and the deterministic model. It can be seen that the energy produced by stabilized semi-implicit scheme \eqref{sta-semi} converges to the one of the determined model computed by the stabilized ETDRK2 scheme as $\epsilon\rightarrow0$. Moreover, it is shown in Figure \ref{fig5} (c) that the computed energies for the regularized model \eqref{log-reg-sac} are consistent with the determined ones within the range of time-stepping scheme error about $\tau$, if the noise density $\epsilon\leq1e-10$ for the tested $\epsilon$. This is also shown in columns 3 and 4 of Figure \ref{fig6} showing several almost consistent snapshots of $u$ for the case of $\epsilon=1e-10$ and the deterministic problem along the coarsening process, respectively.
Another notable observation in Figure \ref{fig5} is that the computed energies are consistent with each other for all tested cases at early time (about $t=12$), and the ones of the tested cases $\epsilon\geq 1e-6$ exhibit completely different behaviors at the later time, which leads a totally different phase transition process seeing the first column of Figure \ref{fig6}.
For the tested cases of $\epsilon\leq1e-8$, the computed energies along the time are almost similar behaviors except the case of $\epsilon=1e-8$ with  $t\in(60,100)$, as shown in Figure \ref{fig5} (b). Furthermore, it  also leads to different phase transition behaviors for the case of $\epsilon=1e-8$ at the later time comparing with the tested case $\epsilon=1e-10$ and the deterministic problem, as shown in the last four rows of Figure \ref{fig6} column 2-4.

\section{Conclusions}
In this paper, we propose and study two multiscale models for a parabolic SPDE with a Flory--Huggins logarithmic potential which emerges from the soft matter and phase separation. The key tools are the
energy regularized technique for the Flory--Huggins logarithmic potential and the Stampacchia maximum principle for studying the possible singularity of the solution. Then we show the stability and strong convergence of a stabilized scheme for the considered logarithmic SPDE.
Following this work, many open problems deserve further investigation.
 For example, it is unknown how to establish the optimal strong and weak convergence rate of the energy regularization model and energy regularized numerical approximations.
What can we expect by extending the energy-regularization technique to  stochastic Cahn--Hillard equation with a Flory--Huggins logarithmic potential? These problems are very crucial to improve accuracy
of numerical simulation and to design high-order convergent schemes for logarithmic SPDEs. We plan to investigate them in the future.

\bibliographystyle{siamplain}
\bibliography{ref}

\end{document}